\documentclass{article}

   \renewcommand{\footnote}[1]{
\textsuperscript{%ecriture en exposant
\addtocounter{footnote}{1}%incrementation du compteur
(\thefootnote)% impression au format "(compteur)"
}
\footnotetext{#1}% la note de bas de page
}

\usepackage{url}
\usepackage{bm}
\usepackage{amsmath}
\usepackage{amsthm}
\usepackage{trfsigns}
\usepackage{wasysym}
\usepackage{amssymb}
\usepackage{amsfonts}
\usepackage{graphicx}  
\usepackage[frenchb,english]{babel}
\usepackage[utf8]{inputenc}%a la place d'applemac utf8
\usepackage[T1]{fontenc}
\usepackage{mathrsfs}
\usepackage{pdfsync}
\usepackage{enumerate}
\usepackage{version}
\usepackage{calc}
\usepackage{subfigure}
 \usepackage{pstricks,pstricks-add,pst-math,pst-xkey}
 \usepackage{float,graphics}
 \usepackage{indentfirst}
 \usepackage[colorlinks=true,linkcolor=blue,urlcolor=blue,citecolor=blue]{hyperref}

\newtheorem{prop}{Proposition}%[chapter]

\newtheorem{thm}{Theorem}%[chapter]
\newtheorem{lem}[prop]{Lemma}
\newtheorem{cor}[prop]{Corollary}%[chapter]

\theoremstyle{remark}
\newtheorem{remark}[prop]{Remark}%[chapter]
\newtheorem{notation}[prop]{Notation}

\newtheorem{defi}[prop]{Definition}%[chapter]

\def\mint{\mathop{\,\,\rlap{--}\!\!\!\int}\nolimits}% \avint_A f(x)\,dx = 1

\def\b{\beta}

\def\v{\varepsilon}

\def\O{\Omega}

\def\p{\partial}

\def\w{\omega}
\def\o{\omega}
\def\R{\mathbb{R}}
\def\C{\mathbb{C}}
\def\S{\mathbb{S}}
\def\Div{{\rm  div}}
\def\deg{ {\rm deg}}
\def\n{\nabla}
\def\dist{{\rm dist}}
\def\tr{{\rm tr}}

\def\weak{\rightharpoonup}
\def\E{\mathcal{E}}

\def\N{\mathbb{N}}
\def\D{\mathbb{D}}
\def\Z{\mathbb{Z}}

\def\K{\mathcal{K}}
\def\I{\mathcal{I}}

\def\1{\textrm{1\kern-0.25emI}}
\def\num{P}
\def\Ring{\mathscr{R}}
\def\H{\mathscr{H}}

\def\grz{\gamma_{\rho,{\bf z},{\bf d}}}
\def\z{\zeta}
\def\dom{\mathcal{D}}
\def\ost{(\o^N)^*}
\def\dst{(\D^N)^*}
\def\omrz{\o_{\rho,{\bf z}}}
\def\deph{\xi_{{\bf z},{\bf d}}}
\def\di{\displaystyle}
\def\Deph{\Psi}
\def\B{\mathcal{B}}
\def\A{\mathcal{A}}

\date{}

\title{Explicit expression of the microscopic renormalized energy for a pinned Ginzburg-Landau functional}
\author{Micka\"el {\sc Dos\,Santos}\footnote{Laboratoire d'Analyse et de Mathématiques Appliquées (LAMA). Universit\'e Paris Est-Cr\'eteil, 61 avenue du G\'en\'eral de Gaulle, 94010 Cr\'eteil Cedex, FRANCE}\\ {\tt mickael.dos-santos@u-pec.fr}}
\begin{document}
\maketitle

\begin{abstract}
We get a new expression of the microscopic renormalized energy for a pinned Ginzburg-Landau type energy modeling small impurities. This is done by obtaining a sharp decomposition for the minimal energy of a  Dirichlet type functional with an $L^\infty$-weight.

In particular we get an explicit expression of the microscopic renormalized energy for a circular impurity. We proceed also to the minimization of this renormalized  energy in some cases.
\end{abstract}

%\tableofcontents

\section{Introduction}
\subsection{Main results}
The goal of this article is to give an explicit formula for a {\it microscopic renormalized energy} in the context of the study of a {\it pinned Ginzburg-Landau type energy}. 

This renormalized energy allows to know the location of vorticity defects inside small impurities in an heterogenous superconductor. The microscopic renormalized energy may be defined {\it via } an auxiliary minimization problem involving unimodular maps. 

The study of this auxiliary problem is the heart of this work. The main result of this article is the following theorem:\begin{thm}\label{ThmMicroEstm}
Let
\begin{itemize}
\item $\o\subset\R^2\simeq\C$ be a smooth  bounded simply connected open set s.t. $0\in\o$,% and $R_0:=\max\{1;10^2\cdot{\rm diameter}({\o})\}$, 
\item $N\in\N^*$ and  $\ost:=\{(z_1,...,z_N)\in\o^N\,|\,z_i\neq z_j\text{ for }i\neq j\}$, 
\item $B\in(0;1)$, $b\in[B;B^{-1}]$ and $\alpha\in L^\infty(\R^2,[B^2;B^{-2}])$ be s.t. $\alpha\equiv b^2$ in $\o$.
\end{itemize}
Then there exist
\begin{itemize}
\item $f\,:\,]R_0,\infty[\to\R^+$ which satisfies $B^2\pi\ln(R)\leq f(R)\leq B^{-2}\pi\ln(R)$ [with $R_0$ sufficiently large], 
\item $\begin{array}{cccc}W^{\rm micro}:&\ost\times\Z^N&\to&\R\\&({\bf z},{\bf d})&\mapsto& W^{\rm micro}({\bf z},{\bf d})\end{array}$
\end{itemize} s.t.  for ${\bf d}=(d_1,...,d_N)\in\Z^N$ and ${\bf z}=(z_1,...,z_N)\in\ost$, when $R\to\infty$ and $\rho\to0^+$, we have
\begin{eqnarray*}
\inf_{\substack{u\in H^1(B(0,R)\setminus\cup_i\overline{B(z_i,\rho)},\S^1)\\\deg_{\p B(z_i,\rho)}(u)=d_i,\,i=1,...,N}}\left\{\dfrac{1}{2}\int_{B(0,R)\setminus\overline{\o}}\alpha|\n u|^2+\frac{b^2}{2}\int_{\o\setminus\cup_i\overline{B(z_i,\rho)}}|\n u|^2\right\}\phantom{aaajsdjdhdjhdjdjkjqhdfkqsdhfq}\\\phantom{aaajsdjdhdjhdjdjkjqhdfkqsdhfq}=\left(\sum_{i=1}^Nd_i\right)^2f(R)+b^2\left(\sum_{i=1}^Nd_i^2\right)|\ln\rho|+W^{\rm micro}({\bf z},{\bf d})+o(1).
\end{eqnarray*}
\end{thm}
[Note that the {\it degree} of a function is defined in Section \ref{SectDeg}].
\begin{remark}\label{VeryFirstRk}
\begin{enumerate}
\item The expression of $W^{\rm micro}$ is given in \eqref{DefMicrEnREn}. The map $W^{\rm micro}:\ost\times\Z^N\to\R$ depends only on $\alpha$, $\o$ and $N$. 
\begin{comment}On a $W:\ost \times\Z^N\to\R$ qui est l'énergie renormalisée de BBH associée à $\o$. C'est à dire pour $({\bf z},{\bf d})\in\ost\times\Z^N$ on a 
\[
\lim_{\rho\to0}\left\{\inf_{\substack{w\in H^1(\o\setminus\cup_{i=1}^d\overline{B(z_i,\rho)},\S^1)\\\deg_{\p B(z_i,\rho)}(w)=d_i\text{ pour tout }i}}\frac{1}{2}\int_{\o\setminus\cup_{i=1}^d\overline{B(z_i,\rho)}}|\n w|^2-\pi \left(\sum_{i=1}^Nd_i^2\right)|\ln\rho|\right\}=W({\bf z},{\bf d},\o)
\]
\end{comment}
\item\label{VeryFirstRk1} The function $f(\cdot)$ is defined by $\displaystyle f(R):=\inf_{\substack{v\in H^1(B_R\setminus \overline{\o},\S^1)\\\deg(v)=1}}\frac{1}{2}\int_{B_R\setminus \overline{\o}}\alpha|\n v|^2$. 
\end{enumerate}
\end{remark}
In the circular case, {\it i.e.}, the set $\o$ is the unit disk $\D$ and $\alpha\equiv1$ outside $\o$, we may obtain an explicit expression for $W^{\rm micro}$.
\begin{prop}\label{Prop-Exp-Circ}
If $\o$ is the unit disk $\D$ and $\alpha=\begin{cases}b^2&\text{if }x\in\o\\1&\text{if }x\notin\o\end{cases}$, then the microscopic renormalized energy with $N$ vortices $({\bf z},{\bf d})=\{(z_1,d_1), ...,(z_N,d_N)\}$ is 
\[
W^{\rm micro}({\bf z},{\bf d})=-b^2\pi\left[\sum_{i\neq j}d_id_j\ln|z_i-z_j|+\dfrac{1-b^2}{1+b^2}\sum_{j=1}^N d_j^2\ln(1-|z_j|^2)+\dfrac{1-b^2}{1+b^2}\sum_{i\neq j}d_id_j\ln|1-z_i\overline{z_j}|\right].
\]
\end{prop}
\begin{remark} Section \ref{RadialSection} is dedicated to the case of the weight considered in Proposition \ref{Prop-Exp-Circ}. Proposition \ref{Prop-Exp-Circ} is proved Section \ref{ProofOf-Prop-Exp-Circ}. The minimization of the renormalized energy $W^{\rm micro}$ in this situation is presented in some particular cases Section \ref{PartMinCasPartDisque}.
\end{remark}

Theorem \ref{ThmMicroEstm} may have several applications. For us, the main motivation appears in the study of a pinned Ginzburg-Landau type energy modeling a superconductor with impurities.

 \subsection{Motivations}
\par {\bf Vorticity defects}

The superconductivity phenomenon is an impressive property that appears on some materials called {\it superconductors}. When a superconductor is cooled below a critical temperature,  it carries  electric currents without dissipation [no electrical resistance] and expels magnetic fields from its body [Meissner effect]. 

But if the conditions imposed on the material are too strong [{\it e.g.} a strong magnetic field] then the superconductivity properties may be destroyed: the material has a classical behavior. According to the response of the superconductor to intense magnetic fields, essentially two kinds of superconductors are distinguished. The {\it type I} superconductors are those which pass abruptly to the superconducting state everywhere to the normal state everywhere. The {\it type II} superconductors admits an intermediate state called {\it mixed state}. Namely, for a type II superconductor, there exists  intermediate critical fields $0<H_{c_1}<H_{c_2}$  s.t. if the intensity of the applied field $H$ is less than $H_{c_1}$ then the superconductor is everywhere in the superconducting state. While if $H>H_{c_2}$, then the superconductor is everywhere in the normal state. For the intermediate regime [$H_{c_1}<H<H_{c_2}$] there are "small" areas where the superconductivity is destroyed. While the rest of the sample  is in a superconducting state. See \cite{SS1} for a rigorous and quite complete presentation of these facts.

The areas where the superconductivity is destroyed are called {\it vorticity defects}. In an {\it homogeneous} superconductor, the vorticity defects arrange themselves into triangular Abrikosov lattice. In the presence of current,  vorticity defects may move,  generating dissipation, and destroying zero-resistance state. A way to prevent this motion is to trap the vorticity defects in small areas called {\it pinning sites}.  In practice, pinning sites are often impurities which are present in a non perfect sample or intentionally introduced by irradiation, doping of impurities. \\

 In order to prevent displacements in the superconductor, the key idea is to consider very small impurities. The heart of this article is to answer to the following question: {\it Once the vorticity defects are trapped by small impurities, what is their locations inside the impurities [microscopic location] ?}

\par {\bf The simplified Ginzburg-Landau functional}

The mathematical theory of the {superconductivity} knew a increasing popularity with the pioneering work of Bethuel, Brezis and Hélein \cite{BBH}$\&$\cite{BBH1}. They studied the minimizers of the simplified Ginzburg-Landau energy
\[
\begin{array}{cccc}
E_\v:&H^1(\O,\C)&\to&\R^+\\&u&\mapsto&\di\dfrac{1}{2}\int_\O|\n u|^2+\dfrac{1}{2\v^2}(1-|u|^2)^2
\end{array}
\]
submitted to a Dirichlet boundary condition in the asymptotic $\v\to0$. In their works $\O$ is a simply connected domain which is a cross section of an homogenous superconducting cylinder $\O\times\R$. The number $\v>0$ is a characteristic parameter of the superconductor; the case $\v\to0$ consists in considering {\it extrem type II superconductor}. 

In this simplified model, a map $u$ which minimizes $E_\v$ [under boundary conditions] models the state of the superconductor in the mixed state. The superconducting area is the set $\{|u|\simeq1\}$ and the vorticity defects are the connected components of $\{|u|\simeq0\}$. One may mention that a quantization of the vorticity defects may be done by observing the {\it degree} of a minimizers around the connected components of $\{|u|\simeq0\}$. In this context we say that $z$ is a {\it vortex} of $u$ when it is an isolated zero of $u$ with a non zero degree. Namely, a vorticity defect may be seen as a small disc [with radius of order of $\v$]  centered at a vortex. A Dirichlet boundary condition [with a non zero degree] mimics the application of a magnetic field by forcing the presence of vorticity defects.

A part of the main results of \cite{BBH}  concerns {\it quantization} $\&$ {\it location} of the {vorticity defects} and an asymptotic estimate of the energy of a minimizer. All these results are related with the crucial notion of {\it renormalized energy}. 
\begin{thm}\label{ThmBBH}[Bethuel-Brezis-Hélein]
Let $\O$ be a smooth and bounded simply connected open set and let $g\in C^\infty(\p\O,\S^1)$ with degree $d>0$.

For $\v>0$ we let $u_\v$ be a minimizer of ${E}_\v$ in $H^1_g$. 
Then:
\begin{enumerate}
\item There exist $\v_0>0$ and $C>1$ s.t. for $0<\v<\v_0$:
\begin{itemize}
\item[$\bullet$]  $u_\v$ has exactly $d$ zeros $x_1^\v,...,x_d^\v$ and $\{x\in\O\,|\,|u_\v(x)|\leq1/2\}\subset \cup_i B(x_i^\v,C\v)$; 
[Here $B(z,r)\subset\R^2$ is the open ball with center $z$ and radius $r$]% for $C>0$ sufficiently large;
%\item We have $|u_\v|\leq 1$ in $\O$ and  for fixed $0<r<\v_0$ we have 
%\begin{center}
%{$|1-|u_\v||\leq C_r\v^2$ in $\O\setminus\cup B(x_i^\v,r)$\footnote{The estimate $|1-|u_\v||\leq C\left(\dfrac{\v}{r}\right)^2$ in $\O\setminus\cup B(x_i^\v,r)$ with $C$ independent on $r$ was later obtained by Comte and Mironescu in \cite{CM1}}}
%\end{center}
\item[$\bullet$] Each zero is of degree $1$: $\deg(u_\v,x_i^\v)=1$ for all $i=1,...,d$; %\footnote{Since $u_\v$ is continuous and because it has exactly $d$ zeros, we may define $\deg(u_\v,x_i^\v):=\deg_{ C_\rho(x_i^\v)}(u_\v/|u_\v|)$ independently of small $\rho$ with $C_\rho(x_i^\v)=\{x\in\C\,|\,|x-x_i^\v|=\rho\}$.} ;
\item As $\v\to0$, up to extraction of a subsequence, there exist $d$ distinct points $a_1,...,a_d\in\O$ s.t. (up to relabeling of the points $x_i^\v$) we have $x_i^\v\to a_i$.
\end{itemize}

\item There exists a smooth map $W_g:\{(x_1,...,x_d)\in\O^d\,|\, x_i\neq x_j\text{ for }i\neq j\}\to\R$, called {\it renormalized energy}, s.t.
\begin{itemize}
\item[$\bullet$]  $E_\v(u_\v)=\pi d|\ln\v|+W(a_1,...,a_d)+d\gamma+o_\v(1)$ where $\gamma\in\R$ is a universal constant;
\item[$\bullet$]  the set $(a_1,...,a_d)$ minimizes the renormalized energy $W_g$.%:\{(x_1,...,x_d)\in\O^d\,|\, x_i\neq x_j\text{ for }i\neq j\}\to\R$.
\end{itemize}
\end{enumerate}

\end{thm}

\begin{remark}\begin{enumerate}
\item In the work \cite{BBH}, the renormalized energy $W_g$ plays an important role. It is defined {\it via} auxiliary minimization problems involving $\S^1$-valued maps: for $N\in\N^*$ and ${\bf z}=(z_1,...z_N)\in\O^N$ s.t. $z_i\neq z_j$ for $i\neq j$, $(d_1,...,d_N)\in\Z^N$ s.t. $\sum d_i=d$ and $\rho\to0^+$ we have:
\begin{eqnarray}\label{MinDeg}
\pi\sum_{i=1}^Nd^2_i|\ln\rho|+W_g({\bf z},{\bf d})&=&\inf_{\substack{u\in H^1(\O\setminus\cup_i\overline{B(z_i,\rho)},\S^1)\\\tr_{\p\O}(u)=g\\\deg_{\p B(z_i,\rho)}(u)=d_i,\,i=1,...,N}}\dfrac{1}{2}\int_{\O\setminus\cup_i\overline{B(z_i,\rho)}}|\n u|^2+o(1)
\\\label{MinDir}
&=&\inf_{\substack{u\in H^1(\O\setminus\cup_i\overline{B(z_i,\rho)},\S^1)\\\tr_{\p\O}(u)=g\\u(z_i+\rho\e^{\imath\theta})=\alpha_i\e^{\imath d_i\theta},\,\alpha_i\in\S^1,\\\,i=1,...,N}}\dfrac{1}{2}\int_{\O\setminus\cup_i\overline{B(z_i,\rho)}}|\n u|^2+o(1).
\end{eqnarray}
In Theorem \ref{ThmBBH}, we have $N=d$ and $d_i=1$ for all $i$ and we wrote $W_g(z_1,...,z_d):=W_g((z_1,1),...,(z_d,1))$.
\item The minimization of $E_\v$ with a Dirichlet boundary condition is not relevant from the physical point of view since the Dirichlet boundary condition is not gauge invariant. In particular, the renormalized energy $W_g$ is not physically relevant. But, in their work, Bethuel-Brezis-Hélein introduced systematic tools and asymptotic estimates to study vorticity defects. 
\end{enumerate}
\end{remark}
\par {\bf The simplified Ginzburg-Landau functional with a pinning term}

One may modify the above model in order to consider a superconducting cylinder with impurities. This is done with the help of a {\it pinning term} $a:\O\to\R^+$ by considering 
 the functional
 \[
\begin{array}{cccc}
E^{\rm pinned}_\v:&H^1(\O,\C)&\to&\R^+\\&u&\mapsto&\di\dfrac{1}{2}\int_\O|\n u|^2+\dfrac{1}{2\v^2}(a^2-|u|^2)^2
\end{array}.
\]
There are a lot of works which deal with a such energy.  Some variants are studied in the literature with the function $a$ which is "smooth" or piecewise constant; independent of $\v$ or depending on $\v$... See the Introduction of \cite{dos2015microscopic} for a more complete presentation of this models. 

In order to present the interpretation of the pinning term, we focus on the case of a pinning term $a:\O\to\R$ piecewise constant. Say, for some $b\in(0;1)$ we have $a(\O)=\{1;b\}$ and $\overline{a^{-1}(\{b\})}$ is a smooth compact subset of $\O$ whose connected components represent the impurities. 

A possible interpretation of a such pinning term is an heterogeneity in temperature. Letting $T_c$ be the critical temperature of the superconductor, if $T_1<T_c$ is the temperature in $a^{-1}(\{1\})$, then $T_b=(1-b^2)T_c+b^2 T_1$ is the temperature in $a^{-1}(\{b\})$. Here the impurities are "heat" areas [note that $T_1<T_2<T_c$]. See Section 2.2 of the Introduction of \cite{TheseDosSantos}.

In order to consider "small" impurities we need to use an $\v$-dependent pinning term [$a_\v:\O\to\{b;1\}$ with $b$ independent of $\v$]. Then we may model {\it shrinking} impurities: the diameter of the connected components of  $\overline{a^{-1}(\{b\})}$ tend to $0$.\\

Essentially three kinds of pinning term may be used. \\

{\bf First kind of pinning term.} The first kind of pinning term are those having a {\bf fixed number of impurities} $\num\in\N^*$:
\begin{itemize}
\item[$\bullet$] the impurities have the same form given by a smooth simply connected open set $\o\subset\R^2,\,0\in\o$;
\item[$\bullet$] the impurities are "centered" at some distinct points $y_1,...,y_\num\in\O$;
\item[$\bullet$] the impurities have size $\lambda=\lambda(\v)\to0$ as $\v\to0$.
\end{itemize}
This kind of pinning term is represented in Figure \ref{Fig.FixedNumber} and it is studied in \cite{Publi3}.
\begin{figure}[H]
\psset{xunit=.45cm,yunit=.450cm,algebraic=true,dotstyle=o,dotsize=3pt 0,linewidth=0.8pt,arrowsize=3pt 2,arrowinset=0.25}
\begin{pspicture*}(-13,-6.5)(13.5,6.2)
{
\rput{168.73}(-6.14,1.53){\psellipse(0,0)(6.91,4.03)}
\rput{-121.33}(-10.22,2.46){\psellipse(0,0)(0.29,0.15)}
\rput{-121.33}(-5.98,-0.6){\psellipse(0,0)(0.29,0.15)}
\rput{-121.33}(-3.43,2.47){\psellipse(0,0)(0.29,0.15)}
\rput{-121.33}(-10,-12){\psellipse(0,0)(0.29,0.15)}
\psline(-10.62,-11.37)(-10.62,-12.57)
\psline(-10.15,2.53)(-0.64,4.5)
\psline(-3.4,2.59)(-0.64,4.5)
\psline(-5.92,-0.54)(-0.64,4.5)
\psline(-9.59,-1.09)(-10.27,-1.98)
\psline(-4.75,-0.62)(4.01,-0.62)
\psline(4.01,3.69)(4.01,-4.92)
\psline(4.01,-4.92)(12.62,-4.92)
\psline(12.62,3.69)(12.62,-4.92)
\psline(4.01,3.69)(12.62,3.69)
\psline(-7.22,0.62)(-4.75,0.62)
\psline(-4.75,-1.85)(-4.75,0.62)
\psline(-7.22,-1.85)(-4.75,-1.85)
\psline(-7.22,-1.85)(-7.22,0.62)
\rput{-121.33}(8.35,-0.55){\psellipse(0,0)(0.87,0.46)}
\psline[linecolor=black]{|-|}(7.67,-1.83)(8.95,-1.83)
\psline[linecolor=black]{|-|}(4.01,-5.6)(12.62,-5.6)
\rput(-10.4,-2.3){ $\textcolor{black}{a_\v=1}$}
\rput(0.6,4.6){$\textcolor{black}{a_\v=b}$}
\rput(8.4,-6){\textcolor{black}{$\simeq1$}}
\rput(8.4,-2.4){\textcolor{black}{$\simeq\lambda$}}
%\rput(-9,4.6){$\O$}
\rput(-5,-4){$\textcolor{black}{\o_\v=\{a_\v=b\}=\cup_{i=1}^\num\o_i^\v},$}
\rput(-5,-5.5){$\textcolor{black}{\o_i^\v:=y_i+\lambda\cdot\o}$}
}

\end{pspicture*}\caption{A pinning term modeling  $\num=3$ small impurities ($\lambda=\lambda(\v)\underset{\v\to0}{\to}0^+)$}\label{Intro.TermedeChevillageManyIncl}\label{Fig.FixedNumber}
\end{figure}

{\bf Second kind of pinning term.} The second kind of pinning term correspond to the  {\bf periodic case}. This case is studied in \cite{Publi4}. For $\delta=\delta(\v)\to0$ as $\v\to0$ we consider a $\delta\times\delta$ squared grid covering $\R^2$. In the center of each cells entirely contained in $\O$ we insert an impurity with size $\lambda\delta$. Here $\lambda$ may be equal to $1$ or $\lambda\to0^+$ as $\v\to0$; it is a factor of dilution when $\lambda\to0$. [See Figure \ref{Intro.FigureTermeChevillage}]

\begin{figure}[H]
% \begin{minipage}{.5\linewidth}\begin{flushleft}
%\centering
\subfigure[A $\delta\times\delta$-periodic pinning term]{
\psset{xunit=0.15cm,yunit=0.15cm,algebraic=true,dotstyle=*,dotsize=3pt 0,linewidth=0.8pt,arrowsize=3pt 2,arrowinset=0.25}
\begin{pspicture*}(-20,-14.3)(36.5,40)
\rput{64.54}(2.39,2.43){\psellipse(0,0)(0.72,0.64)}
%%%%%%%%%%%%%%%%%%%%%%%%%
\psline(12.45,17.15)(20,30)
\rput(25,31){$\textcolor{black}{a_\v=b}$}
%\psline(12.45,12.15)(20,30)
%\psline(7.45,12.15)(20,30)
%\psline(7.45,22.15)(20,30)
%%%%%%%%%%%%%%%%%%%%%%%%%%
\psline(-7,3)(-10,-5)
\rput(-10,-6){$\textcolor{black}{a_\v=1}$}
\psline(-4,1)(-10,-5)
%\psline(-11,17)(-10,-5)
%\psline(7.45,22.15)(-10,-5)
%%%%%%%%%%%%%%%%%%%%%%%%

\rput{64.54}(7.39,2.43){\psellipse(0,0)(0.72,0.64)}
\rput{64.54}(12.39,2.43){\psellipse(0,0)(0.72,0.64)}
\rput{64.54}(2.39,7.43){\psellipse(0,0)(0.72,0.64)}
\rput{64.54}(2.39,12.43){\psellipse(0,0)(0.72,0.64)}
\rput{64.54}(2.39,17.43){\psellipse(0,0)(0.72,0.64)}
\rput{64.54}(7.39,17.43){\psellipse(0,0)(0.72,0.64)}
\rput{64.54}(12.39,17.43){\psellipse(0,0)(0.72,0.64)}
\rput{64.54}(17.39,17.43){\psellipse(0,0)(0.72,0.64)}
\rput{64.54}(17.39,12.43){\psellipse(0,0)(0.72,0.64)}
\rput{64.54}(7.39,12.43){\psellipse(0,0)(0.72,0.64)}
\rput{64.54}(7.39,7.43){\psellipse(0,0)(0.72,0.64)}
\rput{64.54}(12.39,12.43){\psellipse(0,0)(0.72,0.64)}
\rput{64.54}(12.39,7.43){\psellipse(0,0)(0.72,0.64)}
\rput{64.54}(17.39,7.43){\psellipse(0,0)(0.72,0.64)}
\rput{64.54}(17.39,2.43){\psellipse(0,0)(0.72,0.64)}
\rput{-20.94}(7.62,11.61){\psellipse(0,0)(24.83,17.1)}

\rput{64.54}(-2.61,7.43){\psellipse(0,0)(0.72,0.64)}
\rput{64.54}(-7.61,7.43){\psellipse(0,0)(0.72,0.64)}
\rput{64.54}(-2.61,2.43){\psellipse(0,0)(0.72,0.64)}
\rput{64.54}(-2.61,12.43){\psellipse(0,0)(0.72,0.64)}
\rput{64.54}(-2.61,17.43){\psellipse(0,0)(0.72,0.64)}
\rput{64.54}(-2.61,22.43){\psellipse(0,0)(0.72,0.64)}
\rput{64.54}(-7.61,22.43){\psellipse(0,0)(0.72,0.64)}
\rput{64.54}(-7.61,17.43){\psellipse(0,0)(0.72,0.64)}
\rput{64.54}(-7.61,12.43){\psellipse(0,0)(0.72,0.64)}
\rput{64.54}(-12.61,17.43){\psellipse(0,0)(0.72,0.64)}
\rput{64.54}(2.39,22.43){\psellipse(0,0)(0.72,0.64)}
\rput{64.54}(7.39,22.43){\psellipse(0,0)(0.72,0.64)}
\rput{64.54}(12.39,22.43){\psellipse(0,0)(0.72,0.64)}
\rput{64.54}(22.39,12.43){\psellipse(0,0)(0.72,0.64)}
\rput{64.54}(22.39,7.43){\psellipse(0,0)(0.72,0.64)}
\rput{64.54}(27.39,7.43){\psellipse(0,0)(0.72,0.64)}
\rput{64.54}(22.39,2.43){\psellipse(0,0)(0.72,0.64)}
\rput{64.54}(7.39,-2.57){\psellipse(0,0)(0.72,0.64)}
\rput{64.54}(12.39,-2.57){\psellipse(0,0)(0.72,0.64)}
\rput{64.54}(17.39,-2.57){\psellipse(0,0)(0.72,0.64)}
%\rput{64.54}(-12.61,12.43){\psellipse(0,0)(0.72,0.64)}
%\rput{64.54}(-12.61,7.43){\psellipse(0,0)(0.72,0.64)}
%\rput{64.54}(-7.61,2.43){\psellipse(0,0)(0.72,0.64)}
\psline[linestyle=dotted,dash=18pt 18pt](-20,30)(35,30)
\psline[linestyle=dotted,dash=18pt 18pt](35,25)(-20,25)
\psline[linestyle=dotted,dash=18pt 18pt](-20,20)(35,20)
\psline[linestyle=dotted,dash=18pt 18pt](35,15)(-20,15)
\psline[linestyle=dotted,dash=18pt 18pt](35,10)(-20,10)
\psline[linestyle=dotted,dash=18pt 18pt](35,0)(-20,0)
\psline[linestyle=dotted,dash=18pt 18pt](35,-5)(-20,-5)
\psline[linestyle=dotted,dash=18pt 18pt](35,-10)(-20,-10)
\psline[linestyle=dotted,dash=18pt 18pt](35,35)(35,-10)
\psline[linestyle=dotted,dash=18pt 18pt](30,35)(30,-10)
\psline[linestyle=dotted,dash=18pt 18pt](25,35)(25,-10)
\psline[linestyle=dotted,dash=18pt 18pt](20,35)(20,-10)
\psline[linestyle=dotted,dash=18pt 18pt](15,35)(15,-10)
\psline[linestyle=dotted,dash=18pt 18pt](10,35)(10,-10)
\psline[linestyle=dotted,dash=18pt 18pt](5,35)(5,-10)
\psline[linestyle=dotted,dash=18pt 18pt](0,35)(0,-10)
\psline[linestyle=dotted,dash=18pt 18pt](-5,35)(-5,-10)
\psline[linestyle=dotted,dash=18pt 18pt](-10,35)(-10,-10)
\psline[linestyle=dotted,dash=18pt 18pt](-15,35)(-15,-10)
\psline[linestyle=dotted,dash=18pt 18pt](-20,35)(-20,-10)
\psline[linestyle=dotted,dash=18pt 18pt](-20,35)(35,35)
\psline[linestyle=dotted,dash=18pt 18pt](-20,5)(35,5)

\psline(5,-5)(20,-5)
%\psline(15,-5)(5,-5)
\psline(25,0)(-5,0)
%\psline(-5,0)(0,0)
\psline(-10,5)(30,5)
%\psline(30,5)(25,5)
\psline(-10,10)(30,10)
%\psline(15,15)(20,15)
\psline(25,15)(-15,15)
%\psline(20,20)(25,20)
%\psline(20,20)(15,20)
\psline(-15,20)(25,20)
\psline(-10,25)(15,25)

\psline(-15,20)(-15,15)
\psline(-10,5)(-10,25)
\psline(-5,0)(-5,25)
\psline(0,0)(0,25)
%\psline(0,20)(0,25)
\psline(5,-5)(5,25)
\psline(10,25)(10,-5)
%\psline(10,20)(10,25)
\psline(15,25)(15,-5)
\psline(20,-5)(20,20)
%\psline(20,15)(20,20)
%\psline(20,-5)(20,0)
\psline(25,0)(25,20)
%\psline(25,15)(25,0)
\psline(30,5)(30,10)
%\psline(15,-5)(15,25)

\rput{64.54}(22.39,17.43){\psellipse(0,0)(0.72,0.64)}
\rput(32,-8){\textcolor{black}{$\delta$}}
\psline[linecolor=black]{|-|}(30,-10)(35,-10)
\psline[linecolor=black]{|-|}(35,-10)(35,-5)
%\rput(-13,23){$\O$}
\end{pspicture*}
}\hfill\vline\hfill
 \subfigure[$\lambda$ controls the size of an impurity inside a cell]
 {
\psset{xunit=0.15cm,yunit=0.15cm,algebraic=true,dotstyle=*,dotsize=3pt 0,linewidth=0.8pt,arrowsize=3pt 2,arrowinset=0.25}
\begin{pspicture*}(36,2)(60,35)
\psline(38,31)(60,31)
\psline(38,31)(38,9)
\psline(38,9)(60,9)
\psline(60,9)(60,31)
\rput{64.54}(49,20){\psellipse(0,0)(3.6,3.2)}
\psline[linecolor=black]{|-|}(38,7)(60,7)
\rput(49,5){\textcolor{black}{$\delta$}}
%\rput(49,20){$\o_n^\v$}
%\rput(39,10){$C_n^\v$}
\psline[linecolor=black]{|-|}(45.5,16)(52.5,16)
%\psline(45.6,16)(45.6,22)
%\psline(52.4,16)(52.4,22)
\rput(49,14){\textcolor{black}{$\simeq\lambda\delta$}}
\end{pspicture*}
}
\caption{A periodic [rapidly oscillating] pinning term $(\lambda=\lambda(\v),\delta=\delta(\v)\underset{\v\to0}{\to}0)$}\label{Intro.FigureTermeChevillage}
\end{figure}
This periodic pinning term illustrate the fundamental notion of {\it dilution} when $\lambda\to0$. The {\it diluted impurities} are small impurities with the inter-distance between two impurities which is very larger than their diameters. Note that for the periodic pinning term and $\lambda\equiv1$, the size of the impurities is of order $\delta$ and two neighboring impurities have their inter-distance of order $\delta$. Consequently, despite the impurities are small, when $\lambda\equiv1$, the impurities are not diluted.\\

{\bf Third kind of pinning term.} The notion of diluted impurities leads to the third kind of pinning term: the {\bf general diluted pinning term} [See Figure \ref{RepresentationGenTC}]. This case is studied in \cite{Publi4}. The general diluted pinning term correspond to the presence of diluted impurities possibly having different sizes: $\lambda\delta,\dots,\lambda\delta^\num$ for some $\num\in\N^*$. 

\begin{figure}[h]
\begin{center}
\psset{xunit=.850cm,yunit=.850cm,algebraic=true,dotstyle=o,dotsize=3pt 0,linewidth=0.8pt,arrowsize=3pt 2,arrowinset=0.25}
\begin{pspicture*}(-14.5,-2.2)(-0.225,6.2)
\rput{-178.73}(-8.39,1.97){\psellipse(0,0)(5.72,4.07)}
\rput{0}(-11.7,3.72){\psellipse[fillstyle=hlines](0,0)(0.44,0.24)}
\rput{0}(-11.56,0.42){\psellipse[fillstyle=hlines](0,0)(0.44,0.24)}
\rput{0}(-8.2,-0.76){\psellipse[fillstyle=hlines](0,0)(0.44,0.24)}
\rput{0}(-5.54,2.12){\psellipse[fillstyle=hlines](0,0)(0.44,0.24)}
\rput{0}(-8.66,2.3){\psellipse[fillstyle=hlines](0,0)(0.18,0.09)}
\rput{0}(-6.66,4.84){\psellipse[fillstyle=hlines](0,0)(0.18,0.09)}
\rput{0}(-6.8,0.52){\psellipse[fillstyle=hlines](0,0)(0.18,0.09)}
\rput{0}(-10.82,2.32){\psellipse[fillstyle=hlines](0,0)(0.18,0.09)}
\rput{0}(-12.76,2.8){\psellipse[fillstyle=hlines](0,0)(0.18,0.09)}
\rput{0}(-9.2,3.6){\psellipse[fillstyle=hlines](0,0)(0.18,0.09)}
\rput{0}(-9.3,5.46){\psellipse[fillstyle=hlines](0,0)(0.18,0.09)}
\rput{0}(-6.94,3.1){\psellipse[fillstyle=hlines](0,0)(0.18,0.09)}
\rput{0}(-8.28,3.44){\psellipse[fillstyle=hlines](0,0)(0.18,0.09)}
\rput{0}(-3.88,3.06){\psellipse[fillstyle=hlines](0,0)(0.18,0.09)}
\rput{0}(-4.72,0.96){\psellipse[fillstyle=hlines](0,0)(0.18,0.09)}
\rput{0}(-9.8,1.18){\psellipse[fillstyle=hlines](0,0)(0.18,0.09)}
\psline(-1.94,2.88)(-3.88,3.06)
\psline(-5.54,2.12)(-1.94,2.88)
\psline(-5.38,4.64)(-1.94,4.6)
\rput(-1,2.88){$a_\v=b$}
\rput(-1,4.6){$a_\v=1$}

\psline{|-|}(-7.75,-1.2)(-8.65,-1.2)
\psline{|-|}(-6.6,.2)(-7,.2)

\psline{<->}(-8.55,-.7)(-11.2,0.3)
\rput(-9.5,.25){\small$\geq2\delta$}

\psline{<->}(-10.9,2.4)(-11.35,3.55)
\rput(-10,3){\small$\geq\delta+\delta^2$}

\rput(-8.2,-1.6){\small$\approx\lambda\delta$}
\rput(-6.6,-.2){\small$\approx\lambda\delta^2$}

\psline{<->}(-12.5,-.95)(-11.7,0.25)
\rput(-11.70,-.5){\small$\geq\delta$}
\rput(-1.5,2.8){\psline{<->}(-12.5,0)(-11.4,0)
\rput(-11.70,-.5){\small$\geq\delta^2$}}

\psline{<->}(-7.1,3.1)(-8.15,3.4)
\rput(-7.2,3.7){\small$\geq2\delta^2$}
\end{pspicture*}
\caption{Representation of the general diluted pinning term with $\num=2$}\label{RepresentationGenTC}
\end{center}
\end{figure}
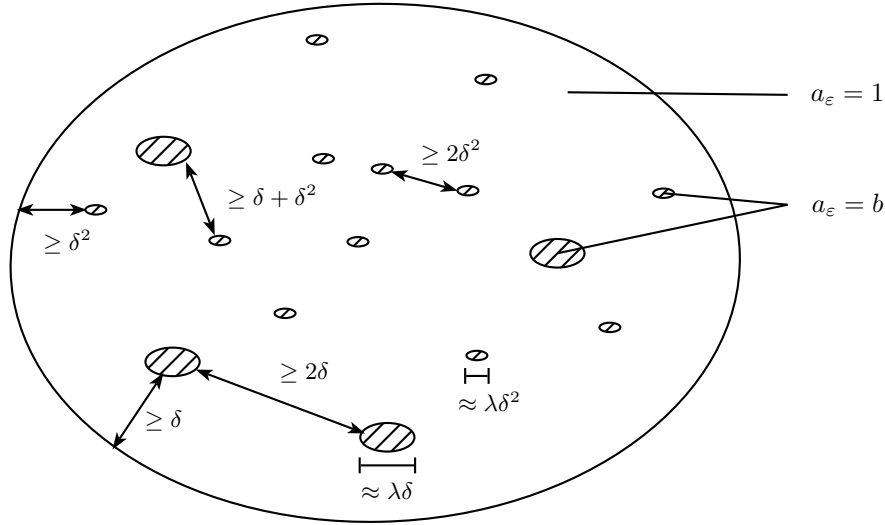

For all these pinning terms, from \cite{Publi3} and \cite{Publi4} we may state the following theorem:
\begin{thm}
Let $\O$ be a smooth and bounded simply connected open set and let $g\in C^\infty(\p\O,\S^1)$ with degree $d>0$. Let $b\in(0;1)$ and $a_\v:\O\to\{b;1\}$ be as in Figure \ref{Fig.FixedNumber} or Figure \ref{Intro.FigureTermeChevillage} or Figure \ref{RepresentationGenTC}. 

Assume that $[\ln(\lambda)]^3/\ln(\v)\to0$ for the first kind of pinning term and $[\ln(\lambda\delta)]^3/\ln(\v)\to0$ for both other cases.

For $\v>0$ we let $u_\v$ be a minimizer of ${E}_\v$ in $H^1_g$. Then there exist $\v_0>0$ and $C>1$ s.t. for $0<\v<\v_0$:
\begin{enumerate}
\item Quantization informations
\begin{itemize}
\item[$\bullet$]  $u_\v$ has exactly $d$ zeros $x_1^\v,...,x_d^\v$ and $\{x\in\O\,|\,|u_\v(x)|\leq b/2\}\subset \cup_i B(x_i^\v,C\v)$; 
% for $C>0$ sufficiently large;
%\item We have $|u_\v|\leq 1$ in $\O$ and  for fixed $0<r<\v_0$ we have 
%\begin{center}
%{$|1-|u_\v||\leq C_r\v^2$ in $\O\setminus\cup B(x_i^\v,r)$\footnote{The estimate $|1-|u_\v||\leq C\left(\dfrac{\v}{r}\right)^2$ in $\O\setminus\cup B(x_i^\v,r)$ with $C$ independent on $r$ was later obtained by Comte and Mironescu in \cite{CM1}}}
%\end{center}
\item[$\bullet$] Each zero is of degree $1$: $\deg(u_\v,x_i^\v)=1$ for all $i=1,...,d$; %\footnote{Since $u_\v$ is continuous and because it has exactly $d$ zeros, we may define $\deg(u_\v,x_i^\v):=\deg_{ C_\rho(x_i^\v)}(u_\v/|u_\v|)$ independently of small $\rho$ with $C_\rho(x_i^\v)=\{x\in\C\,|\,|x-x_i^\v|=\rho\}$.} ;
%\item As $\v\to0$, up to extraction of a subsequence, there exist $d$ distinct points $a_1,...,a_d\in\O$ s.t. (up to relabeling of the points $x_i^\v$) we have $x_i^\v\to a_i$.
\end{itemize}
\item Pinning
\begin{itemize}
\item For the first kind of pinning term: $\cup_i B(x_i^\v,\lambda/C)\subset \{a_\v=b\}$
\item For the second kind of pinning term: $\cup_i B(x_i^\v,\lambda\delta/C)\subset \{a_\v=b\}$
\item For the third kind of pinning term, in order to make a simple presentation of the situation, we assume that there is $\eta_0>0$ [independent of $\v$] s.t. for small $\v$
\begin{itemize}
\item  there are at least $d$ impurities with size $\lambda\delta$: $\o_1^\v,...,\o_d^\v$; 
%:
\item  $\dist(\o_i^\v,\o_j^\v),\dist(\o_i^\v,\p\O)>\eta_0$ for all $i,j\in\{1,...,d\}$, $i\neq j$.
\end{itemize}
Under these extra assumptions we have $\cup_i B(x_i^\v,\lambda\delta/C)\subset \{a_\v=b\}$. In particular the vorticity defects are trapped by the largest impurities.
\end{itemize}
\item Macroscopic location
\begin{itemize}
\item[$\bullet$]  For the first kind of pinning term, the vortices are distributed in the impurities according to the most uniform way. Namely each impurities contain $[d/\num]$ or $[d/\num]+1$ vortices. Here for $x\in\R$, $[x]$ is the integer part of $x$.

The choice between $[d/\num]$ or $[d/\num]+1$ is done {\it via} the minimization of $W_g$.% For $i\in\{1,...,\num\}$ we let $d_i=d_i^\v$ be the number of vortices inside the impurities centered at $y_i$. Then the configuration $\{(y_1,d_1),...,(y_\num,d_\num)\}$ minimizes the renormalized energy $W_g$ among the configurations $\{(y_1,\tilde{d}_1),...,(y_\num,\tilde{d}_\num)\}$ with $\tilde d_i \in \{[d/\num];[d/\num]+1\}$ satisfying $\sum_{i=1}^\num \tilde d_i = d$.
\item[$\bullet$]  For the second kind of pinning term, each impurities contain at most one vortex. Moreover if $\lambda\to0$, then $(x_1^\v,...,x_d^\v)$ tends to minimize $W_g$ with all the degree equal to $1$. If $\lambda\equiv 1$,  then there is no sharp macroscopic information except some classical vortex/vortex Columbian repulsion and confinent effect for the vortices [repulsion effect of $\p\O$].
\item[$\bullet$] For the third kind of pinning term and under the above assumption on the largest impurities, the vortices are trapped by the largest impurities and each impurities contain at most one vortex. Moreover, the choice of the impurities containing a vortex is related with the minimization of the renormalized energy $W_g$ with all the degree equal to $1$.
\end{itemize}
\item Microscopic location

When $\lambda\to0$, for the three kinds of pinning term, the asymptotic location of a vortices inside an impurities tends to be independent on the Dirichlet boundary condition $g$. The microscopic location of the vortices trapped by an impurity tends to minimize a microscopic renormalized energy $W^{\rm micro}$ which depends only on $b$, the form of the impurity and the number of vortices trapped by the impurity.

In the non diluted case [a periodic pinning term with $\lambda\equiv1$], there is no sharp microscopic informations except some classical vortex/vortex Columbian repulsion and confinent effect for the vortices.
\end{enumerate}
\end{thm}
In \cite{dos2015microscopic} [Section 2] it is explained in detailed the link between the  minimization problem considered in Theorem \ref{ThmMicroEstm} and the microscopic location of vortices in a diluted case. 
\begin{remark}
\begin{enumerate}
\item In \cite{Publi3}, the existence and the role of $W^{\rm micro}$ was established. But its expression was not really explicit.
\item In particular, in the easiest case, the case of an impurity which is a disk containing a unique vortex, we expect that the limiting location is the center of the disc. The expression of $W^{\rm micro}$ obtained in \cite{Publi3} does not allow to get this result easily. This result was obtained from scratch in \cite{dos2015microscopic}.
\item Theorem \ref{ThmMicroEstm} has a more general scope than needed. Indeed:
\begin{enumerate}[i.]
\item In Theorem \ref{ThmMicroEstm}, the points $z_i$'s corresponds to the location of the vortices inside an impurity. The weight $\alpha$ is $a^2_\v$ rescaled at the size of the impurity.  

Essentially, in the diluted case, we have to consider $\alpha=\begin{cases}1&\text{outside }\o\\b^2&\text{in }\o\end{cases}$ where $\o$ is the form of the impurity. 
\item With the help of the above theorem, in order to study $W^{\rm micro}$ in the context of a pinned Ginzburg-Landau type function, we may focus on the case $d_i=1$ for $i\in\{1,...,N\}$. But, since the minimization problem considered in Theorem \ref{ThmMicroEstm} is of its self-interest we treat the case of general degrees.
\end{enumerate}
\item In contrast with the renormalized energy $W_g$, we believe that the microscopic renormalized energy $W^{\rm micro}$ may play a role in a more physical problem.
\item If
\begin{itemize}
\item $\o\subset Y:=(-1/2;1/2]\times(-1/2;1/2]$ is as in Theorem \ref{ThmMicroEstm}, 
\item $\alpha=\begin{cases}1&\text{ in }Y\setminus\o\\b^2&\text{ in }\o\end{cases}$,
\item $\alpha$ is $1$-periodic, 
\end{itemize}
then $W^{\rm micro}$ [given in  Theorem \ref{ThmMicroEstm}] should govern the limiting location of vortices inside an impurity for the periodic non diluted case. But, there is no result which asserts that in the non diluted case the microscopic location of the vortices may be studied with this minimization problem. [Despite we believe that, in the non diluted periodic case, microscopic location of vortices should be given by minimal configuration of $W^{\rm micro}$ with degrees $1$]  

Note that in  \cite{dos2015microscopic} [Section 2] the key use of the dilution property is clearly mentioned.
\end{enumerate}
\end{remark}
\section{Notations and basic properties}
\subsection{General notations}
\subsubsection{Set and number}
 \begin{itemize}
\item[$\bullet$] For $z\in\C$ , $|z|$ is the modulus of $z$, ${\rm Re}(z)\in\R$ is the real part of $z$, ${\rm Im}(z)\in\R$ is the imaginary part of $z$, $\overline{z}$ is the conjugate of $z$. 
\item[$\bullet$] "$\wedge$" stands for the vectorial product in $\mathbb{C}$, \emph{i.e.} $z_1 \wedge z_2= {\rm Im}(\overline{z_1}z_2)$, $z_1,z_2\in\mathbb{C}$.
\item[$\bullet$]For $z\in\C$ and $r>0$, $B(z,r)=\{\tilde{z}\in\C\,|\,|z-\tilde{z}|<r\}$. When $z=0$ we write $B_r:=B(0,r)$.
\item[$\bullet$] For a set $A\subset\R^2\simeq\C$, we let $\overline{A}$ be the adherence of $A$ and $\p A$ be the boundary of $A$ [with respect to the usual Euclidean distance in $\R^2$].
\item[$\bullet$] We denote by $\D=B(0,1)$ the unit open disk and $\S^1=\p\D$ the unit circle.
\end{itemize}
\subsubsection{Asymptotic}
 \begin{itemize}
\item[$\bullet$] In this article $R>1$ is a "large" number and $\rho\in(0;1)$ is a small number. We are essentially interested in the asymptotic $R\to\infty$ and $\rho\to0^+$.
\item[$\bullet$] The notation $o_R(1)$ [resp. $o_\rho(1)$] means a quantity depending on $R$ [resp. $\rho$] which tends to $0$ when $R\to+\infty$ [resp. $\rho\to0^+$]. When there is no ambiguity we just write $o(1)$.
\item[$\bullet$] The notation $o[f(R)]$ [resp. $o[f(\rho)]$] means a quantity $g(R)$ [resp. $g(\rho)$] s.t. $\dfrac{g(R)}{f(R)}\to0$ when $R\to+\infty$ [resp. $\dfrac{g(\rho)}{f(\rho)}\to0$ when $\rho\to0$]. When there is no ambiguity we just write $o(f)$.
\item[$\bullet$] The notation $\mathcal{O}[f(R)]$ [resp. $\mathcal{O}[f(\rho)]$] means a quantity $g(R)$ [resp. $g(\rho)$] s.t. $\dfrac{g(R)}{f(R)}$ [resp. $\dfrac{g(\rho)}{f(\rho)}$] is bounded (independently of the variable) when $R$ is large [resp. $\rho>0$ is small]. When there is no ambiguity we just write $\mathcal{O}(f)$.
\end{itemize}

\subsection{Data of the problem}
Along this article we fix:
\begin{itemize}
\item[$\bullet$]$\o\subset\R^2\simeq\C$ be a smooth  bounded simply connected open set s.t. $0\in\o$,% et $\overline{\o}\subset \D$, $\D$ est le disque unité [ouvert],
\item[$\bullet$] $N\in\N^*$, ${\bf d}=(d_1,...,d_N)\in\Z^N$ and we let $ \di d:=\sum_{i=1}^N d_i\in\Z$,%\neq0$
\item[$\bullet$]${\bf z}\in\ost:=\{(z_1,...,z_N)\in\o^N\,|\,z_i\neq z_j\text{ for }i\neq j\}$,% des points deux à deux distincts. %On note ${\bf z}=(z_1,...,z_d)$. % et ${\bf x}=(x_1,...,x_d)\in\O^d$ un $d$-uplet de points deux à deux distincts et 
%\[
%\rho_0:=10^-{2}\min_{i,j}\{|x_i-x_j|,\dist(x_i,\p\O),10^{-2}\},
%\]
%on note alors $\O_{\rho}=\O\setminus\cup \overline{B(x_i,{\rho})}$.
\item[$\bullet$] $B\in(0;1)$, $b\in[B;B^{-1}]$ and $\alpha\in L^\infty(\R^2,[B^2;B^{-2}])$ s.t. $\alpha\equiv b^2$ in $\o$.% et $\alpha\equiv1$ dans $\tilde{\o}\setminus\o$, avec $\tilde\o$ un ouvert régulier simplement connexe tel que $\overline\o\subset\tilde\o$
\end{itemize}
 We define
\begin{center}
$R_0:=\max\{1;10^2\cdot{\rm diameter}({\o})\}$ and $\di\rho_0:=10^{-2}\cdot\min\left\{1,\min_{i\neq j}|z_i-z_j|,\min_i\dist(z_i,\p\o)\right\}$.
\end{center}

 For $R>R_0$ and $\rho_0>\rho>0$, we denote  $\dom_{R,{\bf z}}=B_R\setminus\cup_{i=1}^N\overline{B(z_i,\rho)}$.

The main purpose of this article is the following minimization problem :

\begin{equation}\label{AuxMicroPb}
I(R,\rho,{\bf z},{\bf d}):=\inf_{u\in \I({R,\rho},{\bf z},{\bf d})}\frac{1}{2}\int_{\dom_{R,{\bf z}}}\alpha|\n u|^2%\text{ avec }\dom_{R,{\bf z}}=B_R\setminus\cup_{i=1}^d\overline{B(z_i,\rho)}
\end{equation}
where 
\[
\I(R,\rho,{\bf z},{\bf d}):=\{u\in H^1(\dom_{R,{\bf z}},\S^1)\,|\,\deg_{\p B(z_i,\rho)}(u)=d_i,\,i=1,...,N\}.
\]

Namely, we are interested in the asymptotic behavior of $I(R,\rho,{\bf z},{\bf d})$ when $R\to\infty$ and $\rho\to0$. 

Without loss of generality and for simplicity of the presentation, $R>R_0$ is considered as the major parameter writing $\rho=\rho(R)$.

Before going further we recall some basic facts related with this minimization problem.

\subsection{Test functions and degree}\label{SectDeg}
The functions we consider are essentially defined on {\it perforated domains}:
\begin{defi}\label{DefPerfDom}
We say that $\dom\subset\R^2$ is a {\it perforated domain} when $\dom=\O\setminus\cup_{i=1}^\num\overline{\o_i}$ where $  \num \in\N^*$ and $\O,\o_1,...,\o_\num $ are smooth simply connected bounded open sets s.t. for $i\in\{1,...,\num \}$ we have $\overline{\o_i}\subset\O$ and, for $i\neq j $, $\overline{\o_i}\cap\overline{\o_j}=\emptyset$. 

If $\num =1$ we say that $\dom$ is an annular type domain.
\end{defi}

In this article the test functions stand in the standard Sobolev space of order $1$ with complex values modeled on $L^2$, $H^1(\O,\C)$, where $\O$ is a smooth open set. 

We use the standard norm on $H^1(\O,\C)$ :
\[
\|u\|_{H^1}=\left(\int_\O|u|^2+|\n u|^2\right)^{1/2}.
\]

Our main interest is based on unimodular map, {\it i.e}, the test functions are $\S^1$-valued. Thus we focus on maps lying in 
\[
H^{1}(\O,\S^1):=\{u\in H^1(\O,\C)\,|\,|u|=1\,\text{a.e in }\O\}
\]
where $\O$ is a smooth open set.

For $\O\subset\R^2$ a smooth open set, we let $\tr_{\p \O}:H^1(\O,\C)\to H^{1/2}(\p \O,\C)$ be the {\it trace} operator. Here $H^{1/2}(\p \O,\C)$ is the {\it trace space}

Note if $u\in H^{1}(\O,\S^1)$ then $\tr_{\p \O}(u)\in H^{1/2}(\p \O,\S^1)$.

Recall that for $\Gamma\subset\R^2$ a Jordan curve and $g\in H^{1/2}(\Gamma,\S^1)$, the degree (winding number) of $g$ is defined as
\begin{equation}\label{defDegree}
\deg_{\Gamma}(g):=\frac{1}{2\pi}\int_{\Gamma}g\wedge\p_\tau g\in\Z.
\end{equation}
Here\begin{itemize}
\item[$\bullet$]  $\tau$ is the direct unit tangent vector of $\Gamma$ ($\tau=\nu^\bot$ where $\nu$ is the outward normal unit vector of ${\rm int}(\Gamma)$, the bounded open set whose boundary is $\Gamma$),
\item[$\bullet$] $\p_\tau:=\tau\cdot\n$ is the tangential derivative on $\Gamma$. For further use we denote $\p_\nu=\nu\cdot\n$ the normal derivative on $\Gamma$.
\end{itemize}
\begin{remark}
\begin{enumerate}
\item Note that \eqref{defDegree} may be understood {\it via} $H^{1/2}-H^{-1/2}$ duality. Another way to define the degree of an $H^{1/2}(\Gamma)$-map consists in using a density argument [see Appendix in \cite{boutet_georgescu_purice}].
\item Let $\dom=\O\setminus\cup_{i=1}^\num \overline{\o_i}$ be a perforated domain [see Definition \ref{DefPerfDom}]. The orientation with respect to which we calculate degrees is counter-clockwise on $\p\o_i$ and clockwise on $\p\O$.
\item  If $\dom$ is a perforated domain  and if $u\in H^1(\dom,\S^1)$ then we write 
\[
\deg(u):=(\deg_{\p\o_1}(u),...,\deg_{\p\o_\num }(u))\in\Z^\num .
\]
\end{enumerate}
\end{remark}
For the convenience of the reader we recall some basic properties related with the degree.
\begin{prop}[\cite{glcoursep6}]\label{PropRappelDeg}
Let $\Gamma\subset\R^2$ be a Jordan curve and let $\di\dom:=\O\setminus \cup_{i=1}^\num \overline{\o_i}$ be a perforated domain.
\begin{enumerate}
\item\label{PropRappelDeg0} For ${\bf d}\in\Z^\num $ we have 
\[
\E_{\bf d}:=\left\{u\in H^1(\dom,\S^1)\,|\,\deg(u)={\bf d}\right\}\neq\emptyset.
\]
\item\label{PropRappelDeg1} If $u\in H^{1/2}(\Gamma,\S^1)$ then we have
\[
\exists \phi\in H^{1/2}(\Gamma,\R)\text{ s.t. }u=\e^{\imath\phi}\Longleftrightarrow\deg_\Gamma(u)=0.
\]
Moreover, if for $\phi_1,\phi_2\in H^{1/2}(\Gamma,\R)$ we have $\e^{\imath\phi_1}=\e^{\imath\phi_2}$ then $\phi_1-\phi_2=\lambda\in2\pi\Z$.
\item\label{PropRappelDeg2} If $u,v\in H^{1/2}(\Gamma,\S^1)$, then we have 
\[
\deg_\Gamma(uv)=\deg_\Gamma(u)+\deg_\Gamma(v)\text{ and }\deg_\Gamma({1}/{u})=-\deg_\Gamma(u).
\]
\item\label{PropRappelDeg3} If $u\in H^1(\dom,\S^1)$ then $\di\deg_{\p\O}(u)=\sum_{i=1}^\num \deg_{\p\o_i}(u)$.
\item\label{PropRappelDeg4} If $u\in H^1(\dom,\S^1)$ then there exists $\varphi\in H^1(\dom,\R)$ s.t. $u=\e^{\imath\varphi}$ if and only if $\deg_{\p\o_i}(u)=0$ for $i\in\{1,...,\num \}$.
\begin{itemize}
\item[$\bullet$] In particular for $u_0\in H^1(\dom,\S^1)$ we have
\[
\left\{u\in H^1(\dom,\S^1)\,|\,\deg(u)=\deg(u_0)\right\}=\left\{u_0\e^{\imath\varphi}\,|\,\varphi\in H^1(\dom,\R)\right\}.
\]
\item[$\bullet$] Moreover, if for $\varphi_1,\varphi_2\in H^{1}(\dom,\R)$ we have $\e^{\imath\varphi_1}=\e^{\imath\varphi_2}$ then $\varphi_1-\varphi_2=\lambda\in2\pi\Z$.
\end{itemize}
\item\label{PropRappelDeg5}For ${\bf d}=(d_1,...,d_\num )\in\Z^\num $ and $(z_1,...,z_\num )\in\o_1\times\cdots\times\o_\num $ we have 
\[
\prod_{i=1}^\num \left(\dfrac{z-z_i}{|z-z_i|}\right)^{d_i}\in\E_{\bf d}.
\]
Locally, one may define $\theta_i$, a determination of the argument of $z-z_i$ in $\R^2\setminus\{z_i\}$. Note that $\n\theta_i$ is globally defined in $\dom$ and $\displaystyle\sum_{i=1}^\num  d_i\n\theta_i\in L^2(\dom,\R^2)$. 

Therefore, letting $\Theta:=\displaystyle\sum_{i=1}^\num  d_i\theta_i$, we have, for $u\in \E_{\bf d}$, the existence of $\varphi\in H^1(\dom,\R)$ s.t. $u=\e^{\imath(\Theta+\varphi)}$. 

In other words, for $u\in H^1(\dom,\S^1)$ there exists $\psi$ which is locally defined in $\dom$ and whose gradient is in $L^2(\dom,\R^2)$ s.t. $u={\rm e}^{\imath\psi}$.
\end{enumerate}
\end{prop}
\begin{notation}
\begin{enumerate}
\item It is important to note that for $u\in H^1(\dom,\S^1)$, the function $\psi$ given by Proposition \ref{PropRappelDeg}.\ref{PropRappelDeg5} is locally defined modulo $2\pi$ while $\n\psi$ is   globally well defined. Namely $\n\psi=u\wedge\n u$.
\item For simplicity of the presentation, when there is no ambiguity, we may omit the dependance on the Jordan curve in the notation of the degree. For example:
\begin{itemize}
\item[$\bullet$] if $\Gamma$ is a Jordan curve and if $h\in H^{1/2}(\Gamma,\S^1)$, then we may write $\deg(h)$ instead of $\deg_\Gamma(h)$.
\item[$\bullet$] If $\dom=\O\setminus\overline{\o}$ is an annular type domain and $u\in H^1(\dom,\S^1)$, then $\deg_{\p\O}(u)=\deg_{\p\o}(u)$. Consequently, without ambiguity, we may write $\deg(u)$ instead of $\deg_{\p\O}(u)$ or $\deg_{\p\o}(u)$.
\end{itemize}
\end{enumerate}
\end{notation}
\subsection{Minimization problems}
One of the main issue in this article is the study of minimization problems of weighted Dirichlet functionals with prescribed degrees :
\begin{equation}\label{EqGenFormMinPb}
\inf_{u\in \E_{\bf d}}\frac{1}{2}\int_\dom\alpha|\n u|^2
\end{equation}
where
\begin{itemize}
\item[$\bullet$] $\dom:=\O\setminus\cup_{i=1}^\num \overline{\o_i}$ is a perforated domain as in Definition \ref{DefPerfDom},
\item[$\bullet$] ${\bf d}=(d_1,...,d_\num )\in\Z^\num $,
\item[$\bullet$] $\E_{\bf d}:=\{u\in H^1(\dom,\S^1)\,|\,\deg_{\p\o_i}(u)=d_i\text{ for }i\in\{1,...,\num \}\}$,
\item[$\bullet$] $\alpha\in L^\infty(\dom;[B^2;B^{-2}])$, $B\in(0;1)$.
\end{itemize}
Problem \eqref{EqGenFormMinPb} admits solutions which are  unique up to a constant rotation. Namely we have the following proposition:
\begin{prop}\label{PropExistResultMainTypPB}
Minimisation problem \eqref{EqGenFormMinPb} admits solutions. Moreover if $u$ is a solution of \eqref{EqGenFormMinPb} then $v$ is a solution of \eqref{EqGenFormMinPb} if and only if there exists $\lambda\in\S^1$ s.t. $v=\lambda u$. 

Moreover a minimizer $u_{\bf d}$ solves
\begin{equation}\label{EqGenEq}
\begin{cases}
-\Div(\alpha\n u_{\bf d})=\alpha u_{\bf d}|\n u_{\bf d}|^2&\text{ in }\dom
\\
\p_\nu u_{\bf d}=0&\text{ on }\p\dom
\end{cases}.
\end{equation}
\end{prop}
\begin{proof}Since from Proposition \ref{PropRappelDeg}.\ref{PropRappelDeg4}, the set $\E_{\bf d}$ is closed under the weak-$H^1$ convergence, the existence of solution of \eqref{EqGenFormMinPb} is obtained by direct minimization.

If $u_{\bf d}$ is  a solution of \eqref{EqGenFormMinPb}, then from standard computations of directional derivatives we get that \eqref{EqGenEq} is satisfied [see {\it e.g.} Appendix A in \cite{Publi4}].

Let $u_{\bf d}$ be a  solution of \eqref{EqGenFormMinPb}. From Proposition \ref{PropRappelDeg}.\ref{PropRappelDeg5}, there exists $\psi_{\bf d}$ which is locally defined in $\dom$ and whose gradient is in $L^2(\dom,\R^2)$ s.t. $u_{\bf d}=\e^{\imath\psi_{\bf d}}$. In terms of $\psi_{\bf d}$, Equations \eqref{EqGenEq} reads  :
\begin{equation}\label{EqGenEqPhase}
\begin{cases}
-\Div(\alpha\n \psi_{\bf d})=0&\text{ in }\dom
\\
\p_\nu \psi_{\bf d}=0&\text{ on }\p\dom
\end{cases}.
\end{equation}
Thus, if $v_{\bf d}$ is a minimizers, then, with the help of Proposition \ref{PropRappelDeg}.\ref{PropRappelDeg4}, there exists $\varphi\in H^1(\dom,\R)$ s.t. $v_{\bf d}=\e^{\imath(\psi_{\bf d}+\varphi)}$. Then, using the minimality of $v_{\bf d}$ we get
\[
\begin{cases}
-\Div[\alpha\n (\psi_{\bf d}+\varphi)]=0&\text{ in }\dom
\\
\p_\nu (\psi_{\bf d}+\varphi)=0&\text{ on }\p\dom
\end{cases}.
\]
Consequently, using \eqref{EqGenEqPhase} we obtain
\begin{equation}\label{EqGenEqDephase}
\begin{cases}
-\Div(\alpha\n\varphi)=0&\text{ in }\dom
\\
\p_\nu \varphi=0&\text{ on }\p\dom
\end{cases}.
\end{equation}
With the help an integration by parts, we easily get that $\varphi\in H^1(\dom,\R)$ solves \eqref{EqGenEqDephase} if and only if $\varphi$ is a constant. This argument yields  the uniqueness of the solution up to a constant rotation.

\end{proof}
%The main result of this article is the following theorem
\section[First step in the proof of Theorem \ref{ThmMicroEstm}]{First step in the proof of Theorem \ref{ThmMicroEstm}: splitting of the domain}\label{Section.FirstStep}
The first step in the proof of Theorem \ref{ThmMicroEstm} is standard. The strategy employed was already used in \cite{dos2015microscopic}. It consists in splitting the integral over $\dom_{R,{\bf z}}$ [in \eqref{AuxMicroPb}] in two parts: the integral over $\O_R:=B_R\setminus\overline{\o}$ and the one over $\o_{\rho,{\bf z}}:=\o\setminus\cup_{i=1}^N\overline{B(z_i,\rho)}$ [as presented in Theorem \ref{ThmMicroEstm}].

For each integrals we consider a mixed minimization problem by adding an arbitrary Dirichlet boundary condition on $\p\o$: $h\in H^{1/2}(\p\o,\S^1)$ s.t. $\deg(h)=d=\sum d_i$. 

We then claim that these mixed minimization problems admit "unique" solutions. 

In the next steps we will  solve these problems, we will minimize among $h\in H^{1/2}(\p\o,\S^1)$ s.t. $\deg(h)=d$ and finally we will decouple the minimal energy according to the different  data.

The splitting consists in the following obvious equality:
\begin{equation}\label{DecompositionEnergy}
I(R,\rho,{\bf z},{\bf d})= \inf_{\substack{h\in H^{1/2}(\p\o,\S^1)\\\text{s.t. }\deg(h)=d}}\left\{\inf_{\substack{v\in H^1(\O_R,\S^1)\\\tr_{\p\o}(v)=h}}\frac{1}{2}\int_{\O_{R}}\alpha|\n v|^2+\inf_{\substack{w\in H^1(\o_{\rho,{\bf z}},\S^1)\\\tr_{\p\o}(w)=h\\\deg_{\p B(z_i,\rho)}(w)=d_i\,\forall i}}\frac{b^2}{2}\int_{\o_{\rho,{\bf z}}}|\n w|^2\right\}.
\end{equation}
The three previous minimization problems admit "unique" solutions:
\begin{prop}\label{Exist+Eq-SplittingDomain}
\begin{enumerate}
\item Both minimization problems in \eqref{DecompositionEnergy} having a [partial] Dirichlet boundary condition $h\in H^{1/2}(\p\o,\S^1)$ in \eqref{DecompositionEnergy} admit each a unique solution. For $h\in H^{1/2}(\p\o,\S^1)$ let $v_{R,h}$ be the solution of 
\begin{equation}\label{GlobExt}
\inf_{\substack{v\in H^1(\O_R,\S^1)\\\tr_{\p\o}(v)=h}}\int_{\O_{R}}\alpha|\n v|^2
\end{equation}
and $w_{\rho,h}$ be the one of 
\begin{equation}\label{GlobInt}
\inf_{\substack{w\in H^1(\o_{\rho,{\bf z}},\S^1)\\\tr_{\p\o}(w)=h\\\deg_{\p B(z_i,\rho)}(w)=d_i\,\forall i}}\int_{\o_{\rho,{\bf z}}}|\n w|^2.
\end{equation}
Then $v_{R,h}$ is the unique solution of
\begin{equation}\label{EqGlobExt}
\begin{cases}
-\Div(\alpha\n v_{R,h})=\alpha v_{R,h}|\n v_{R,h}|^2&\text{ in }\O_R
\\v_{R,h}=h&\text{ on }\p\o
\\\p_\nu v_{R,h}=0&\text{ on }\p B_R
\end{cases}
\end{equation}
and $w_{\rho,h}$ is the unique solution of
\begin{equation}\label{EqGlobInt}
\begin{cases}
-\Delta w_{\rho,h}= w_{\rho,h}|\n w_{\rho,h}|^2&\text{ in }\o_{\rho,{\bf z}}
\\w_{\rho,h}=h&\text{ on }\p\o
\\\p_\nu w_{\rho,h}=0&\text{ on }\p B(z_i,\rho),\,i\in\{1,...,N\}
\\\deg_{\p B(z_i,\rho)}(w_{\rho,h})=d_i&i\in\{1,...,N\}
\end{cases}.
\end{equation}
\item The minimization problem in \eqref{DecompositionEnergy} among $h\in H^{1/2}(\p\o,\S^1)$ s.t. $\deg(h)=d$ admits solutions. Moreover if $h_0$ is a solution, then $\tilde{h}_0$ is a minimizer if and only if there exists $\lambda\in\S^1$ s.t. $\tilde{h}_0=\lambda h_0$.
\end{enumerate} 
\end{prop}
\begin{proof}
It is clear [by considering a minimizing sequence] that \eqref{GlobExt} and \eqref{GlobInt} admit solutions. Moreover these minimizers solve the corresponding equations \eqref{EqGlobExt} and \eqref{EqGlobInt}.

We now prove that \eqref{EqGlobExt} admits a unique solution. The argument is similar to prove that the solution of \eqref{EqGlobInt} is unique. Let $v_0$ be a solution of \eqref{EqGlobExt} and $v\in H^1(\O_R,\S^1)$ s.t. $\tr_{\p\o}(v)=h$. On the one hand, writing $v_0=\e^{\imath\psi_0}$ where $\psi_0$ is locally defined in $\O_R$ and $\n\psi_0\in L^2(\O_R)$ is globally defined [Proposition \ref{PropRappelDeg}.\ref{PropRappelDeg5}], it standard to get that 
\[
\begin{cases}v\in H^1(\O_R,\S^1)\\\tr_{\p\o}(v)=h\end{cases}\Longleftrightarrow\begin{cases}v=\e^{\imath(\psi_0+\varphi)}\\\varphi\in H^1(\O_R,\R)\text{ s.t. }\tr_{\p\o}(\varphi)=0
\end{cases}.%\begin{cases}\exists\,\varphi\in H^1(\O_R,\R)\text{ s.t. }\tr_{\p\o}(\varphi)=0\\v=\e^{\imath(\psi_0+\varphi)}\end{cases}.
\]  
On the other hand, from direct calculations, we have, for $v=\e^{\imath(\psi_0+\varphi)}$ s.t. $\varphi\in H^1(\O_R,\R)\&\tr_{\p\o}(\varphi)=0$, the following equivalence
\[
\begin{cases}-\Div(\alpha\n v)=\alpha v|\n v|^2\text{ in }\O_R\\\p_\nu v=0\text{ on }\p B_R\end{cases}\Longleftrightarrow\begin{cases}-\Div[\alpha\n(\psi_0+\varphi)]=0\text{ in }\O_R\\\p_\nu (\psi_0+\varphi)=0\text{ on }\p B_R\end{cases}.
\]
Thus $-\Div(\alpha\n\psi_0)=0$ in $\O_R$ and $\p_\nu\psi_0=0$ on $\p B_R$. Consequently, if $v=\e^{\imath(\psi_0+\varphi)}$ is a solution of \eqref{EqGlobExt}, then $\varphi$ solves 
\[
\begin{cases}-\Div(\alpha\n\varphi)=0\text{ in }\O_R\\\p_\nu\varphi=0\text{ on }\p B_R\end{cases}. 
\]
Noting that $\tr_{\p\o}(\varphi)=0$ we immediatly obtained that $\varphi=0$, {\it i.e.}, $v=v_0$.

The second part of the proposition is a direct consequence of Proposition \ref{PropExistResultMainTypPB} and of the first part of Proposition \ref{Exist+Eq-SplittingDomain}.
\end{proof}
\section[Second step in the proof of Theorem \ref{ThmMicroEstm}]{Second step in the proof of Theorem \ref{ThmMicroEstm}: the key ingredient}\label{SectionSpeSol}
The key ingredient in this article is the use of {\it special solutions}. In order to motivate their use, we focus on the fully radial homogeneous case: $\o=\D$,\,$\alpha\equiv1$, $N=1$, $z=0$.

It is easy to check that, letting $\Ring(R,\rho,0):= B_R\setminus \overline{B_\rho}$ with $R>\rho>0$, for $d\in\Z$, the map 
\[
\begin{array}{cccc}
u_d:&\Ring(R,\rho,0)&\to&\S^1\\&x&\mapsto&\left(\dfrac{x}{|x|}\right)^d
\end{array}
\]
is a global minimizer of the Dirichlet functional $\displaystyle\dfrac{1}{2}\int_{\Ring(R,\rho,0)}|\n\cdot|^2$ in the space 
\[
\E_d:=\{u\in H^1[\Ring(R,\rho,0),\S^1]\,|\,\deg(u)=d\}.
\]
%Moreover, it is direct to get
%\[
%\E_d=\{u_d\e^{\imath\varphi}\,|\,\varphi\in H^1[\Ring(R,\rho,0),\R]\}.
%\]
Letting $\theta(x)$ be a determination of the argument of $x\in\C\setminus\{0\}$ which is locally defined in $\Ring(R,\rho,0)$ and whose gradient is globally defined, we have $u_d=\e^{\imath d\theta}$. 

Let $u\in\E_d$ and $\varphi\in H^1[\Ring(R,\rho,0),\R]$ be s.t. $u=u_d\e^{\imath\varphi}=\e^{\imath (d\theta+\varphi)}$ [Proposition \ref{PropRappelDeg}.\ref{PropRappelDeg4}]. Since $\theta$ solves $-\Delta\theta=0$ in $\Ring(R,\rho,0)$ and $\p_\nu\theta=0$ on $\p \Ring(R,\rho,0)$ with the help of an integration by parts we obtain
\begin{eqnarray*}
\dfrac{1}{2}\int_{\Ring(R,\rho,0)}|\n u|^2&=&\dfrac{1}{2}\int_{\Ring(R,\rho,0)}|\n (d\theta+\varphi)|^2
\\&=&\dfrac{d^2}{2}\int_{\Ring(R,\rho,0)}|\n \theta|^2+\dfrac{1}{2}\int_{\Ring(R,\rho,0)}|\n \varphi|^2
\\&=&\dfrac{1}{2}\int_{\Ring(R,\rho,0)}|\n u_d|^2+\dfrac{1}{2}\int_{\Ring(R,\rho,0)}|\n \varphi|^2.
\end{eqnarray*}
These calculations are standard and give an easy decoupling for the energy of $u=u_d\e^{\imath\varphi}$ as the energy of $u_d$ pulse those of the dephasing $\varphi$.

The main argument of  this article consists in the fact that this argument is not restricted to the fully radial homogeneous case. Indeed we have the following proposition:
\begin{prop}\label{PropMainArgDec}
Let $\dom$ be a perforated domain,  $B\in(0;1)$, $\alpha\in L^\infty(\dom;[B^2;B^{-2}])$ and ${\bf d}\in\Z^N$. We let $u_{\bf d}$ be a minimizer of \eqref{EqGenFormMinPb}.
%\begin{equation}\label{EqGenFormMinPb}
%\inf\left\{\frac{1}{2}\int_\dom\alpha|\n u|^2\,|\,u\in H^1(\dom,\S^1),\,\deg_{\p\o_i}(u)=d_i\text{ for }i\in\{1,...,N\}\right\}.
%\end{equation}
Then for $\varphi\in H^1(\dom,\R)$ we have
\[
\frac{1}{2}\int_\dom\alpha|\n (u_{\bf d}\e^{\imath\varphi})|^2=\frac{1}{2}\int_\dom\alpha|\n u_{\bf d}|^2+\frac{1}{2}\int_\dom\alpha|\n \varphi|^2.
\]
\end{prop}
\begin{proof}
We fix $\dom,B,\alpha,{\bf d}$ be as in the proposition. First note that, from Proposition \ref{PropExistResultMainTypPB}, we get the existence of $u_{\bf d}$. 
Moreover $u_{\bf d}$ is a solution of
\[
\begin{cases}
-\Div(\alpha\n u_{\bf d})=\alpha u_{\bf d}|\n u_{\bf d}|^2&\text{ in }\dom
\\
\p_\nu u_{\bf d}=0&\text{ on }\p\dom
\end{cases}.
\]
We may write $u_{\bf d}=\e^{\imath\psi}$ where $\psi$ is locally defined in $\dom$ and $\n\psi\in L^2(\dom,\R^2)$ [Proposition \ref{PropRappelDeg}.\ref{PropRappelDeg5}]. 

Thus $\psi$ solves
\begin{equation}\label{EqPropPsiGeneric}
\begin{cases}
-\Div(\alpha\n \psi)=0&\text{ in }\dom
\\
\p_\nu \psi=0&\text{ on }\p\dom
\end{cases}.
\end{equation}
Let $\varphi\in H^1(\dom,\R)$. We have 
\begin{eqnarray*}
\frac{1}{2}\int_\dom\alpha|\n (u_{\bf d}\e^{\imath\varphi})|^2&=&\frac{1}{2}\int_\dom\alpha|\n (\e^{\imath(\psi+\varphi)})|^2
\\&=&\frac{1}{2}\int_\dom\alpha|\n (\psi+\varphi)|^2
\\&=&\frac{1}{2}\int_\dom\alpha|\n \psi|^2+\int_\dom\alpha\n\psi\cdot\n\varphi+\frac{1}{2}\int_\dom\alpha|\n \varphi|^2.
\end{eqnarray*}
From \eqref{EqPropPsiGeneric} and  an integration by parts  we get $\displaystyle\int_\dom\alpha\n\psi\cdot\n\varphi=0$ and this equality ends the proof of the proposition since
\[
\frac{1}{2}\int_\dom\alpha|\n \psi|^2=\frac{1}{2}\int_\dom\alpha|\n u_{\bf d}|^2.
\]
\end{proof}
\begin{remark}
It is easy to check that Proposition \ref{PropMainArgDec} allows to prove in a "different" way the uniqueness, up to a constant rotation, of a minimizer of \eqref{EqGenFormMinPb}.
\end{remark}
Because minimizers of \eqref{EqGenFormMinPb} are not unique, in order to fix such a minimizer we add an extra condition. This choice leads to the crucial notion of {\it special solution}.

In both next sections we define the special solutions in $\O_R= B_R\setminus\overline{\o}$ [Section \ref{Sec.SpecSolExt}] and in $\o_{\rho,{\bf z}}=\o\setminus\cup \overline{B(z_i,\rho)}$ [Section \ref{Sec.SpecSolIn}].
\subsection{The special solution in $\O_R$}\label{Sec.SpecSolExt}
In this section we focus on the annular type domain  $\O_R=B_R\setminus\overline{\o}$. We first treat the case  $d=1$ by considering:
\begin{equation}\label{MinGlobExt}
\inf_{\substack{v\in H^1(\O_R,\S^1)\\\deg(v)=1}}\frac{1}{2}\int_{\O_{R}}\alpha|\n v|^2.
\end{equation}
\begin{comment}
As explained above we have the following proposition:
\begin{prop}
The minimization problem \eqref{MinGlobExt} admits solution which are unique up to the multiplication by constant in $\S^1$. 
\end{prop}
\end{comment}
With the help of Proposition \ref{PropExistResultMainTypPB}, we may fix a map  $v_R\in H^1(\O_R,\S^1)$  s.t. $\deg(v_R)=1$ which is a solution of \eqref{MinGlobExt}. We freeze the non-uniqueness of  $v_R$ by letting $v_R$ be in the form
\begin{equation}\label{DefSolPartExtInclusmkj}
\text{$v_R=\dfrac{x}{|x|}{\rm e}^{\imath \gamma_R}$ with $\gamma_R\in H^{1}(\O_R,\R)$ s.t. $\int_{\p\o}\gamma_R=0$.}
\end{equation}

It is clear that such map  $v_R$ is uniquely and well defined. 

It is easy to check that, for  $d\in\Z$, we have $v_R^d$ which is a solution of the minimization problem: 
\begin{equation}\label{MinGlobExtDegd}
\inf_{\substack{v\in H^1(\O_R,\S^1)\\\deg(v)=d}}\frac{1}{2}\int_{\O_{R}}\alpha|\n v|^2.
\end{equation}
Moreover it is the unique solution of the form $v_R^d=\left(\dfrac{x}{|x|}\right)^d{\rm e}^{\imath \tilde\gamma}$ with $\tilde\gamma\in H^{1}(\O_R,\R)$ s.t. $\int_{\p\o}\tilde\gamma=0$.

We have the following proposition:
\begin{prop}
For $x=|x|{\rm e}^{\imath\theta}\in\O_R$ we have $v_R(x)={\rm e}^{\imath(\theta+\gamma_R(x))}$ with $\gamma_R\in H^1(\O_R)$ which is a solution of  
\begin{equation}\label{EquationExt}
\begin{cases}
-\Div\left[\alpha\n(\theta+\gamma_R)\right]=0\text{ in }\O_R
\\
\n(\theta+\gamma_R)\cdot\nu=0\text{ on }\p\O_R
\\
\int_{\p \o}\gamma_R=0
\end{cases}.
\end{equation}
\end{prop}
The special solution $v_R$ is fundamental in the analysis since it allows to get a decoupling of weighted Dirichlet energy. Namely, from Proposition \ref{PropMainArgDec} we have:
\begin{lem}\label{LemDecOmRmlkjh}
For $d\in\Z$ and $\varphi\in H^1(\O_R,\R)$ we have:
\[
\dfrac{1}{2}\int_{\O_R}\alpha|\n (v_R^d\e^{\imath\varphi})|^2=\dfrac{d^2}{2}\int_{\O_R}\alpha|\n v_R|^2+\dfrac{1}{2}\int_{\O_R}\alpha|\n\varphi|^2.
\]
\end{lem}
\begin{comment}
\begin{proof}
Soit $R>R_0$ et $\varphi\in H^1(\O_R,\R)$. On a
\begin{eqnarray*}
\dfrac{1}{2}\int_{\O_R}\alpha|\n (v_R\e^{\imath\varphi})|^2&=&\dfrac{1}{2}\int_{\O_R}\alpha|\n \e^{\imath(\theta+\gamma_R+\varphi)}|^2
\\&=&\dfrac{1}{2}\int_{\O_R}\alpha|\n (\theta+\gamma_R+\varphi)|^2
\\&=&\dfrac{1}{2}\int_{\O_R}\alpha|\n (\theta+\gamma_R)|^2+\int_{\O_R}\alpha\n(\theta+\gamma_R)\cdot\n\varphi+\dfrac{1}{2}\int_{\O_R}\alpha|\n\varphi|^2.
\end{eqnarray*}
En utilisant \eqref{EquationExt} on obtient $\di \int_{\O_R}\alpha\n(\theta+\gamma_R)\cdot\n\varphi=0$ d'où le résultat.
\end{proof}
\end{comment}
The above lemma allows to get a crucial information on the asymptotic behavior of $(\gamma_R)_R$:
\begin{prop}\label{PropTheArgum}
There exists $\gamma_\infty\in H^{1}_{\rm loc}(\R^2\setminus\o,\R)$ s.t. when $R\to\infty$ we have $\gamma_R\to\gamma_\infty$ in $H_{\rm loc}^{1}(\R^2\setminus\o)$.
\end{prop}
\begin{proof}
Let $R'>R>R_0$ and $\varphi_R=\gamma_{R'}-\gamma_R$ in order to have $v_{R'}=v_R\e^{\imath\varphi_R}$ in $\O_R$.

From Lemma \ref{LemDecOmRmlkjh} we have 
\begin{equation}\label{CroisonsFORT}
\dfrac{1}{2}\int_{\O_R}\alpha|\n v_{R'}|^2=\dfrac{1}{2}\int_{\O_R}\alpha|\n (v_R\e^{\imath\varphi_R})|^2=\dfrac{1}{2}\int_{\O_R}\alpha|\n v_R|^2+\dfrac{1}{2}\int_{\O_R}\alpha|\n\varphi_R|^2.
\end{equation}
We need the following lemma:
\begin{lem}\label{AhhAhhAhh1}
There exists a constant $C_{B,\o}>0$ depending only on $B$ and $\o$ s.t.
\begin{equation}\nonumber
\dfrac{1}{2}\int_{\O_R}\alpha|\n\varphi_R|^2\leq C_{B,\o}.
\end{equation}
\end{lem}

For the convenience of the reader the proof of this lemma in postponed in Appendix [see Appendix \ref{SecProofAsyBeha}].

From Lemma \ref{AhhAhhAhh1} we have
\[
\dfrac{1}{2}\int_{B_{\sqrt R}\setminus\overline{B_{R^{1/4}}}}\alpha|\n\varphi_R|^2\leq C_{B,\o}.
\]
\begin{notation}
%:
In the rest of this proof, $C_0$ stands for a constant depending only on $\o$ and $B$ derived from $C_{B,\o}$ and with universal multiplicative constants. Its values may change from line to line.
\end{notation}
Therefore, with the help of a mean value argument, we have the existence of  $r\in(R^{1/4},\sqrt R)$ and of a constant $C_0$ depending only on $B$ and $\o$ s.t.:
\[
\int_0^{2\pi}|\p_\theta\varphi_R(r\e^{\imath\theta})|^2\,{\rm d}\theta\leq\frac{C_0}{\ln R}.
\]
We denote $\di m_R:=\mint_0^{2\pi}\varphi_R(r\e^{\imath\theta})\,{\rm d}\theta$.

From the above estimate and with the help of a  Poincaré-Wirtinger inequality, we have  
\[
\int_0^{2\pi}\left(\varphi_R(r\e^{\imath\theta})-m_R\right)^2\,{\rm d}\theta\leq\frac{C_0}{\ln R}.
\]

We now define $\tilde\varphi_R\in H^1(B_R,\R)$ using polar coordinates: 
\[
\tilde\varphi_R(s,\theta)=\begin{cases}m_R&\text{for }s\in[0,r/2]
\\\di\dfrac{s-r/2}{r/2}\varphi_R(r,\theta)+\dfrac{r-s}{r/2}m_R&\text{for }s\in]r/2,r[ 
\\\varphi_R(s,\theta)&\text{for } s\in[r,R[
\end{cases}.
\]
It is easy to check that $\tilde\varphi_R\in H^1(B_R,\R)$ and with direct calculations we obtain:
\begin{eqnarray}\nonumber
\int_{B_{r}}|\n\tilde\varphi|^2&=&\int_{B_{r}\setminus{B_{r/2}}}|\n\tilde\varphi|^2
\\\nonumber&\leq&\int_{r/2}^{r}s\,{\rm d}s\int_0^{2\pi}\left\{\dfrac{4}{r^2}\left(\varphi_R(r,\theta)-m_R\right)^2+\dfrac{1}{s^2}\left[\frac{2(s-r)}{r}\right]^2|\p_{\theta}\varphi_R(r,\theta)|^2\right\}{\rm d}\theta
\\\label{KeyEstPolarRRR}&\leq&\frac{C_0}{\ln R}.
\end{eqnarray}
By noting that $\tr_{\p B_R}(v_R\e^{\imath\tilde\varphi_R})=\tr_{\p B_R}(v_R\e^{\imath\varphi_R})=\tr_{\p B_R}(v_{R'})$, with the help of $\tilde\varphi_R$ we construct $\tilde{v}_R\in H^1(\O_{R'},\S^1)$ :
\[
\tilde{v}_R=
\begin{cases}
v_{R'}&\text{in }B_{R'}\setminus\overline{B_R}
\\
v_{R}\e^{\imath\tilde\varphi_R}&\text{in }\O_R
\end{cases}.
\] 
From the minimality of  $v_{R'}$ and Lemma \ref{LemDecOmRmlkjh} we get
\begin{eqnarray}\nonumber
\dfrac{1}{2}\int_{\O_{R'}}\alpha|\n  v_{R'}|^2&\leq&\dfrac{1}{2}\int_{\O_{R'}}\alpha|\n \tilde v_R|^2
\\\nonumber&=&\dfrac{1}{2}\int_{\O_{R'}\setminus\overline{\O_R}}\alpha|\n \tilde v_R|^2+\dfrac{1}{2}\int_{\O_{R}}\alpha|\n \tilde v_R|^2
\\\label{CroisonslesDOIGTS}&=&\dfrac{1}{2}\int_{\O_{R'}\setminus\overline{\O_R}}\alpha|\n  v_{R'}|^2+\dfrac{1}{2}\int_{\O_{R}}\alpha|\n  v_R|^2+\int_{\O_R}\alpha|\n \tilde\varphi_R|^2.
\end{eqnarray}
Estimate \eqref{CroisonslesDOIGTS} implies: 
\[
\dfrac{1}{2}\int_{\O_{R}}\alpha|\n  v_{R'}|^2\leq\dfrac{1}{2}\int_{\O_{R}}\alpha|\n  v_R|^2+\dfrac{1}{2}\int_{\O_R}\alpha|\n \tilde\varphi_R|^2.
\]
The above inequality coupled with \eqref{CroisonsFORT} gives:
\[
\dfrac{1}{2}\int_{\O_R}\alpha|\n\varphi_R|^2\leq\dfrac{1}{2}\int_{\O_R}\alpha|\n \tilde\varphi_R|^2.
\]
On the other hand, from the definition of  $\tilde\varphi_R$ we have $\tilde\varphi_R=\varphi_R$ in $B_R\setminus\overline{B_{r}}$. Consequently we deduce:
\[
\dfrac{1}{2}\int_{\O_{r}}\alpha|\n\varphi_R|^2\leq\dfrac{1}{2}\int_{\O_{r}}\alpha|\n \tilde\varphi_R|^2.
\]
With \eqref{KeyEstPolarRRR} and since $r\in(R^{1/4},\sqrt R)$ we may conclude
\[
\dfrac{1}{2}\int_{\O_{R^{1/4}}}\alpha|\n\varphi_R|^2\leq \dfrac{C_0}{\ln R}.
\]
In particular, for a compact set $K\subset\R^2\setminus\o$ s.t. $\p\o\subset\p K$ we have for sufficiently large $R$
\[
\dfrac{1}{2}\int_{K}\alpha|\n\varphi_R|^2\leq \dfrac{C_0}{\ln R}.
\]
Since $\di\mint_{\p\o}\varphi_R=0$, we may use a  Poincaré type inequality to get:%pour un compact $K\subset\R^2\setminus\o$ tel que $\p\o\subset\p K$  : 
\[
\|\varphi_R\|_{H^1(K)}\to0\text{ when $R\to\infty$ independently of $R'>R$.}
\]
It suffices to note that $\varphi_R=\gamma_{R'}-\gamma_R$ in order to conclude that $(\gamma_R)_R$ is a Cauchy family in $H^1(K)$. Then $(\gamma_R)_R$ is a Cauchy family in $H_{\rm loc}^1(\R^2\setminus\o)$. The completeness of $H_{\rm loc}^1(\R^2\setminus\o,\R)$ allows to get the existence of $\gamma_\infty\in H_{\rm loc}^{1}(\R^2\setminus\o,\R)$ s.t. $\gamma_R\to\gamma_\infty$ in $H_{\rm loc}^{1}(\R^2\setminus\o)$.
\end{proof}
\begin{cor}\label{Cor-CVGammaR}   We have two direct consequences of Proposition \ref{PropTheArgum} :
\begin{enumerate}
\item $\tr_{\p\o}(\gamma_R)\to \tr_{\p\o}(\gamma_\infty)$ in $H^{1/2}(\p\o)$,
\item $v_R=\dfrac{x}{|x|}{\rm e}^{\imath\gamma_R}\to v_\infty:=\dfrac{x}{|x|}{\rm e}^{\imath\gamma_\infty} $ in $H_{\rm loc}^{1}(\R^2\setminus\o)$.% avec $\deg(v_\infty)=1$.
\end{enumerate}
\end{cor}
\subsection{The special solution in $\o_{\rho,{\bf z}}$}\label{Sec.SpecSolIn}
As for the special solution in $\O_R$, we first consider the minimization problem:
\begin{equation}\label{MinGlobInt}
\inf_{\substack{w\in H^1(\o_{\rho,{\bf z}},\S^1)\\\deg_{\p B(z_i,\rho)}(w)=d_i\,\forall i}}\frac{1}{2}\int_{\o_{\rho,{\bf z}}}|\n w|^2.
\end{equation}
\begin{comment}
\begin{prop}
Minimization problem \eqref{MinGlobInt} admits solution which are unique up to the multiplication by constant in $\S^1$. 
\end{prop}
\end{comment}
From Proposition \ref{PropExistResultMainTypPB}, we may fix $w_{\rho,{\bf z},{\bf d}}$, a unique solution of \eqref{MinGlobInt}, by imposing 
\begin{equation}\label{DefSolPartDansInclusmkj}
w_{\rho,{\bf z},{\bf d}}=\prod_{i=1}^N\left(\dfrac{x-z_i}{|x-z_i|}\right)^{d_i}{\rm e }^{\imath \gamma_{\rho,{\bf z},{\bf d}}}\text{ with }\int_{\p\o}\grz=0.
\end{equation}
For $i\in\{1,...,N\}$, we may locally define $\theta_i$ in $\R^2\setminus\{ z_i\}$ as a lifting  of $\dfrac{x-z_i}{|x-z_i|}$, {\it i.e.}, $\e^{\imath\theta_i}=\dfrac{x-z_i}{|x-z_i|}$. Moreover $\n\theta_i$ is globally defined.

We denote $\Theta:=d_1\theta_1+...+d_N\theta_N$ which is locally defined in $\R^2\setminus\{z_1,...,z_N\}$ and whose gradient is globally defined in $\R^2\setminus\{z_1,...,z_N\}$. It is clear that
\[
\prod_{i=1}^N\left(\dfrac{x-z_i}{|x-z_i|}\right)^{d_i}=\e^{\imath\Theta}\text{ in  }\R^2\setminus\{z_1,...,z_N\}.
\]
From the definition of $w_{\rho,{\bf z},{\bf d}}$ we have the following proposition.
\begin{prop}
$w_{\rho,{\bf z},{\bf d}}={\rm e}^{\imath(\Theta+\grz)}$ with $\grz\in H^1(\omrz)$ which is a solution of  
\begin{equation}\label{EquationInt}
\begin{cases}
-\Delta\grz=0\text{ in }\omrz
\\
\n(\Theta+\grz)\cdot\nu=0\text{ on }\p\omrz
\\
\int_{\p \o}\grz=0
\end{cases}.
\end{equation}
\end{prop}
In contrast with the previous section, the asymptotic behavior of  $w_{\rho,{\bf z},{\bf d}}$ is well known when $\rho\to0$. For example Lefter and R{\u{a}}dulescu proved the following theorem.
\begin{thm}\label{THM-LR}[Theorem 1 \cite{LR1}]
For  $\rho_0>\rho>0$ we let $w_\rho$ be a minimizer of \eqref{MinGlobInt} and we consider a sequence $\rho_n\downarrow0$. Up to pass to a subsequence, there exists $w_0\in C^\infty(\overline{\o}\setminus\{z_1,...,z_N\},\S^1)$ s.t. $w_{\rho_n}\to w_0$ dans $C^k_{\rm loc}(\overline{\o}\setminus\{z_1,...,z_N\})$ for all $k\geq0$.

Moreover the limits $w_0$ are unique up to the multiplication by a constant in $\S^1$.
\end{thm}
From Theorem \ref{THM-LR}, we get that the possible limits $w_0$'s are unique up to a constant rotation. Thus there exists a unique limit $w_{0,{\bf z},{\bf d}}$ [given by Theorem \ref{THM-LR}] which may be written:
\begin{equation}\label{DefSolPartLIMITEInclusmkj}
w_{0,{\bf z},{\bf d}}=\prod_{i=1}^N\left(\dfrac{x-z_i}{|x-z_i|}\right)^{d_i}{\rm e }^{\imath \gamma_{0,{\bf z},{\bf d}}}\text{ with }\int_{\p\o}\gamma_{0,{\bf z},{\bf d}}=0.
\end{equation}
On the other hand, it is easy to check that for $\rho=\rho_n\downarrow0$, if $w_{\rho,{\bf z},{\bf d}}={\rm e}^{\imath(\Theta+\grz)}\to w_{0}={\rm e}^{\imath(\Theta+\gamma_0)}$ in $C^1(\p\o)$ then ${\rm e}^{\imath\grz}\to {\rm e}^{\imath\gamma_0}$ in $C^1(\p\o)$. Moreover if we impose $\int_{\p\o}\gamma_0\in[0,2\pi[$ then we immediately get $\int_{\p\o}\gamma_0=0$.

We thus have the following corollary:

\begin{cor}\label{Cor-CVGammarz} Let $\gamma_{0,{\bf z},{\bf d}}\in H_{\rm loc}^{1}(\overline{\o}\setminus\{z_1,...,z_N\},\R)$ be defined by \eqref{DefSolPartLIMITEInclusmkj}. When $\rho\to0$ we have $\grz\to \gamma_{0,{\bf z},{\bf d}}$ in $H_{\rm loc}^{1}(\overline{\o}\setminus\{z_1,...,z_N\})$. 

Thus we also get $\tr_{\p\o}(\gamma_{\rho,{\bf z},{\bf d}})\to\tr_{\p\o}(\gamma_{0,{\bf z},{\bf d}})$ in $H^{1/2}(\p\o)$.
\end{cor}
\begin{proof}
Let $K\subset\overline{\o}\setminus\{z_1,...,z_N\}$ be a connected compact set s.t. $\p\o\subset\p K$ and let $\rho_n\downarrow0$ be s.t. $w_{\rho_n,{\bf z},{\bf d}}={\rm e}^{\imath(\Theta+\gamma_{\rho_n,{\bf z}, {\bf d}})}\to w_0={\rm e}^{\imath(\Theta+\gamma_{0})}$ in $C^1(K)$ for some $\gamma_0\in C^1(K)$. It suffices to prove that we may choose $\gamma_{0}=\gamma_{0,{\bf z},{\bf d}}$ defined by \eqref{DefSolPartLIMITEInclusmkj}.

On the one hand, we have $\n\gamma_{\rho_n,{\bf z}, {\bf d}}=w_{\rho_n,{\bf z},{\bf d}}\wedge\n w_{\rho_n,{\bf z},{\bf d}}-\n \Theta\to w_{0}\wedge\n w_0-\n \Theta=\n \gamma_0$ in $L^2(K)$. 
 Then $\gamma_{0}=\gamma_{0,{\bf z},{\bf d}}+\lambda$ for some $\lambda\in\R$. 

On the other hand $(\gamma_{\rho_n,{\bf z}, {\bf d}})_n$ is bounded in $H^1(K)$, consequently, up to pass to a subsequence, we have $\gamma_{\rho_n,{\bf z}, {\bf d}}\weak\gamma_0$ in $H^1(K)$. We the help of the previous paragraph, we get that the convergence is in fact strong. Thus $\tr_{\p\o}(\gamma_{\rho_n,{\bf z}, {\bf d}})\to\tr_{\p\o}(\gamma_{0})$ in $L^2(\p\o)$. 

In conclusion 
\[
0=\mint_{\p\o}\gamma_{\rho_n,{\bf z}, {\bf d}}\to\mint_{\p\o}\gamma_{0}=\lambda+\mint_{\p\o}\gamma_{0,{\bf z},{\bf d}}=\lambda=0.
\]
This means $\gamma_0=\gamma_{0,{\bf z},{\bf d}}$.
\end{proof}
About the asymptotic energetic expanding, Lefter and R{\u{a}}dulescu proved the following result:
\begin{thm}\label{THMDevAsyEnergInt}[Theorem 2 \cite{LR1}]
For $N\in\N^*$, there exists a map $W:\ost\times\Z^N\to\R$ s.t. for ${\bf d}\in\Z^N$ and ${\bf z}\in\ost$ when $\rho\to0$ we have:
\[
\inf_{\substack{w\in H^1(\o_{\rho,{\bf z}},\S^1)\\\deg_{\p B(z_i,\rho)}(w)=d_i\,\forall i}}\frac{1}{2}\int_{\o_{\rho,{\bf z}}}|\n w|^2=\left(\sum_{i=1}^Nd_i^2\right)\pi|\ln\rho|+W({\bf z},{\bf d})+o_\rho(1).
\]
\end{thm}
\section{Upper Bound}
We are now in position to start the proof of Theorem \ref{ThmMicroEstm}. To this end, the goal of this section is to identify a map   
 \[
 \K:\{h\in H^{1/2}(\p\o,\S^1)\,|\,\deg(h)=d\}\to\R
 \]
 s.t. for a fixed $h\in H^{1/2}(\p\o,\S^1)$ with $\deg(h)=d$, when $R\to\infty$ we have
\begin{eqnarray}\nonumber
\inf_{\substack{v\in H^1(\O_R,\S^1)\\\tr_{\p\o}(v)=h}}\frac{1}{2}\int_{\O_{R}}\alpha|\n v|^2&+&\inf_{\substack{w\in H^1(\o_{\rho,{\bf z}},\S^1)\\\tr_{\p\o}(w)=h\\\deg_{\p B(z_i,\rho)}(w)=d_i\,\forall\,i}}\frac{b^2}{2}\int_{\o_{\rho,{\bf z}}}|\n w|^2\\\label{FonctionDeH}&&=\K(h)+d^2{f}(R)+b^2\left[\left(\sum_{i=1}^Nd_i^2\right)|\ln\rho|+W({\bf z},{\bf d})\right]+o(1).
\end{eqnarray}
In the above estimate we have: 
\begin{itemize}
\item$\K$  is independent of $R,\rho$ ;%,{\bf z}$ et ${\bf d}$ ;
\item${f}$  is defined by Remark \ref{VeryFirstRk}.\ref{VeryFirstRk1} and is independent of $h,\rho,{\bf z},{\bf d}$ and $d=\sum d_i$ ;
\item$W$  is independent of $b,B,h,\rho$ and $R$.
\end{itemize}
For this purpose we fix $h\in H^{1/2}(\p\o,\S^1)$ s.t. $\deg(h)=d$. In this section we identify  $\K(h)$ in order to have a such expanding.

Note that from Corollaries \ref{Cor-CVGammaR} and \ref{Cor-CVGammarz}, we have the existence of  
\begin{itemize}
\item $\gamma_\infty\in H^{1/2}(\p\o)$  s.t. $\gamma_R\to \gamma_\infty$ in $H^{1/2}(\p\o)$,
\item  $\gamma_{0,{\bf z},{\bf d}}\in H^{1/2}(\p\o)$ s.t. $\grz\to \gamma_{0,{\bf z},{\bf d}}$ in $H^{1/2}(\p\o)$. 
\end{itemize}
It is important to claim that since  $\int_{\p\o} \gamma_R=0$ and $\int_{\p\o} \grz=0$, we have $\int_{\p\o}\gamma_\infty=0$ and $\int_{\p\o}\gamma_{0,{\bf z},{\bf d}}=0$.

\subsection{Study in the domain $\O_R$}\label{EtudeBorneSupOmeR}
 %On se permet un abus de notation  en écrivant $R\in[R_0,\infty[$ pour parler des termes de la suite et $R\in[R_0,\infty]$ pour parler des termes de la suite et du cas limite $R=\infty$.

For  $R\in[R_0,\infty[$ and $h\in H^{1/2}(\p\o,\S^1)$ s.t. $\deg(h)=d$ we consider
\begin{equation}\label{DecompositionEnergyPbExt}
I_R(h):= \inf_{\substack{v\in H^1(\O_R,\S^1)\\\tr_{\p\o}(v)=h}}\frac{1}{2}\int_{\O_{R}}\alpha|\n v|^2.
\end{equation}
Our goal is to estimate $I_R(h)$ when $R\to\infty$.

We let  
\[
g_0:=h\left(\dfrac{\overline x}{|x|}\right)^d\in H^{1/2}(\p\o).
\]
From Proposition \ref{PropRappelDeg}.\ref{PropRappelDeg2}, we have $\deg(g_0)=0$. Consequently, using Proposition \ref{PropRappelDeg}.\ref{PropRappelDeg1} we may fix a unique $\phi_0\in H^{1/2}(\p\o,\R)$ s.t. 
\begin{center}
$g_0={\rm e}^{\imath \phi_0}$ and $\di\mint_{\p\o}\phi_0\in[0,2\pi[$. 
\end{center}
\begin{remark}\label{ModKinou}It is clear that for $\beta\in\R$ we have $I_R(h)=I_R(\e^{\imath\beta} h)$. Thus, up to replace $h$ by  $\e^{\imath\beta} h$ with $\beta= \di-\mint_{\p\o}\phi_0$, we may assume that $\di\mint_{\p\o}\phi_0=0$.
\end{remark}

For  $R\in[R_0,\infty]$ we let
\[
g_R:=h \tr_{\p\o}(\overline{v^d_R})=g_0\e^{-\imath d\tr_{\p\o}(\gamma_R)},
\]
in order to have $h=g_R\tr_{\p\o}(v^d_R)$. [Note that $v_\infty$ is defined in Corollary \ref{Cor-CVGammaR}] 

Consequently we have $g_R=\e^{\imath(\phi_0-d\tr_{\p\o}(\gamma_R))}$. Finally we let $\phi_R:=\phi_0-d\tr_{\p\o}(\gamma_R)\in H^{1/2}(\p\o,\R)$ and so we get $h=\tr_{\p\o}(v^d_R)\e^{\imath\phi_R}$ and $\int_{\p\o}\phi_R=0$.
 
%Puisque $\deg(g_R)=0$, on peut fixer  $\phi_R\in H^{1/2}(\p\o,\R)$ tq $g_R={\rm e}^{\imath \phi_R}$ et $\mint_{\p\o}\phi_R\in[0,2\pi[$. 

%De plus, puisque $v_R=\dfrac{x}{|x|}{\rm e}^{\imath \gamma_R}$ avec $\int_{\p\o}\gamma_R=0$, on obtient que $\mint_{\p\o}\phi_R\in[0,2\pi[$ ne dépend pas de $R$.

From Corollary \ref{Cor-CVGammaR} we immediately obtain:

\begin{cor}\label{PropConvDephasageExt} $\phi_R\underset{R\to\infty}{\longrightarrow} \phi_\infty$ in $H^{1/2}(\p\o)$.
\end{cor}
%\begin{proof}
%D'une part, pour $R\in[1,\infty]$, $g_R={\rm e}^{\imath \phi_R}$ avec $\phi_R\in H^{1/2}(\p\o,\R)$ qui est uniquement définie en imposant $\mint_{\p\o}\phi_R\in[0,2\pi[$. 

%D'autre part, pour $R\in[1,\infty]$, $g_R=h \tr_{\p\o}(\overline{v_R})=g_0{\rm e}^{-\imath \gamma_R}={\rm e}^{-\imath (\phi_0-\gamma_R)}$.

%Par conséquent, pour $R\in[1,\infty]$, $\phi_R=\phi_0-\gamma_R$ d'où le résultat puisque $\gamma_R\to\gamma_\infty$ dans $H^{1/2}(\p\o)$.
%\end{proof}
For $R\in[R_0,+\infty[$ and $v\in H^1(\O_R,\S^1)$, from Proposition \ref{PropRappelDeg}, we may prove the following equivalence:
\[
\tr_{\p\o}(v)=h\Longleftrightarrow v=v^d_R{\rm e}^{\imath \varphi}\text{ with }\left|\begin{array}{c}\varphi\in H^1(\O_R,\R)\\\tr_{\p\o}(\varphi)=\phi_R\end{array}\right..
\]  
%Denoting by $\theta$ a determination of the argument of $x$ (locally defined in $\O_R$), we let $\psi_R=\theta+\gamma_R$ be s.t. $v_R={\rm e}^{\imath\psi_R}$ [$\psi_R$ is locally defined in $\O_R$ while its gradient is in $L^2(\O_R,\R^2)$].

On the other hand, for $v=v_R^d{\rm e}^{\imath \varphi}\in H^1(\O_R,\S^1)$, from Lemma \ref{LemDecOmRmlkjh} we have %tq $\tr_{\p\o}(\varphi)=\phi_R$ on a 
\begin{eqnarray}\label{DecompoExtkjh}
\frac{1}{2}\int_{\O_R}\alpha|\n v|^2%&=&\frac{1}{2}\int_{\O_R}\alpha|d\n\psi_R+\n\varphi|^2
%\\\nonumber&=&\frac{d^2}{2}\int_{\O_R}\alpha|\n\psi_R|^2+d\int_{\O_R}\alpha\n\psi_R\cdot\n\varphi+\frac{1}{2}\int_{\O_R}\alpha|\n\varphi|^2
%\\\label{DecompoExtkjh}{[\text{avec \eqref{EquationExt}}]}
&=&\frac{d^2}{2}\int_{\O_R}\alpha|\n v_R|^2+\frac{1}{2}\int_{\O_R}\alpha|\n\varphi|^2.
\end{eqnarray}
Therefore, one may obtain that $v=v^d_R{\rm e}^{\imath \varphi}$ with $\tr_{\p\o}(\varphi)=\phi_R$ is a solution of the minimization problem
\[
\inf_{\substack{v\in H^1(\O_R,\S^1)\\\tr_{\p\o}(v)=h}}\frac{1}{2}\int_{\O_{R}}\alpha|\n v|^2
\]
if and only if $\varphi\in  H^1(\O_R,\R)$ is a solution of the  minimization problem
\begin{equation}\label{MinAuxPahaseExtR}
\inf_{\substack{\varphi\in H^1(\O_R,\R)\\\tr_{\p\o}(\varphi)=\phi_R}}\frac{1}{2}\int_{\O_{R}}\alpha|\n \varphi|^2.
\end{equation}
It is easy to get:
\begin{prop}
For $R\in[R_0,\infty[$, Problem \eqref{MinAuxPahaseExtR} admits a unique solution denoted by ${\varphi_R}$. Moreover this minimizer is the unique solution of %De plus ce minimiseur est l'unique solution pour $R\in[R_0,\infty[$ de
\begin{equation}\nonumber%\label{EqMinAuxPahaseExtR}
\begin{cases}
-\Div(\alpha\n\varphi_R)=0\text{ in }\O_R
\\
\tr_{\p\o}(\varphi_R)=\phi_R
\\
\p_\nu\varphi_R=0\text{ on } \p B_R
\end{cases}.
\end{equation}
\end{prop}
For $R=\infty$, we denote $\O_\infty:=\R^2\setminus\overline\o$ and 
\begin{equation}\label{DefEspaceLimiteExt}
\H_{\phi_\infty}:=\{\varphi\in H^1_{\rm loc}(\overline{\O_\infty},\R)\,|\,\n \varphi\in L^2(\O_\infty)\text{ and }\tr_{\p\o}(\varphi)=\phi_\infty\}.
\end{equation}
We are now interested in the minimization problem:
\begin{equation}\label{MinAuxPahaseExtInf}
\inf_{\substack{\varphi\in \H_{\phi_\infty}}}\frac{1}{2}\int_{\O_\infty}\alpha|\n \varphi|^2.
\end{equation}
By direct minimization we get:
\begin{prop}\label{Prop-DefLimMinSol}Problem \eqref{MinAuxPahaseExtInf} admits a unique solution denoted by $\varphi_\infty$. Moreover $\varphi_\infty$ is a solution of
\begin{equation}\label{EqMinAuxPahaseExtinfty}
-\Div(\alpha\n\varphi_\infty)=0\text{ dans }\O_\infty.
\end{equation}
\end{prop}
We are now able to prove the main result of this section:
\begin{prop}\label{P-MinAuxPahaseExtR}
We have:
\begin{center}
 $\varphi_R\to\varphi_\infty$ in $H^1_{\rm loc}(\R^2\setminus\o)$ 
 \end{center}
 and
 \begin{center}
  $\n\varphi_R\1_{\O_R}\to\n\varphi_\infty$ in $L^2(\R^2\setminus\overline\o)$, with $\1_{\O_R}(x)=\begin{cases}1&\text{ if }x\in\O_R\\0&\text{ if }x\notin\O_R\end{cases}$.
\end{center}

And consequently:
\[
\int_{\O_R}\alpha|\n\varphi_R|^2=\int_{\O_\infty}\alpha|\n\varphi_\infty|^2+o_R(1).
\]
\end{prop}
\begin{proof}
From Corollary \ref{PropConvDephasageExt} we have  $\phi_R-\phi_\infty\to 0$ in $H^{1/2}(\p\o)$. Consequently, there exists $\xi_R\in H^1(\O_\infty,\R)$ s.t.
\[
\text{$\tr_{\p\o}(\xi_R)=\phi_R-\phi_\infty$ and $\|\xi_R\|_{H^1(\O_\infty)}\to0$.}
\]
The test function $\varphi_\infty+\xi_R$ satisfies the boundary condition of Problem \eqref{MinAuxPahaseExtR}, therefore:
\begin{equation}\label{EqREROBs}
\frac{1}{2}\int_{\O_{R}}\alpha|\n \varphi_R|^2\leq\frac{1}{2}\int_{\O_{R}}\alpha|\n (\varphi_\infty+\xi_R)|^2=\frac{1}{2}\int_{\O_{R}}\alpha|\n \varphi_\infty|^2+o(1).
\end{equation}
Note we used  $\displaystyle\int_{\O_{R}}\alpha|\n \varphi_\infty|^2\leq C_0:=\int_{\O_{\infty}}\alpha|\n \varphi_\infty|^2<\infty$.
From \eqref{EqREROBs}, we obtain
\begin{equation}\label{BorneSupRER}
\limsup_{R\to\infty}\frac{1}{2}\int_{\O_{R}}\alpha|\n \varphi_R|^2\leq\frac{1}{2}\int_{\O_{\infty}}\alpha|\n \varphi_\infty|^2.
\end{equation}
We now prove {\it the} "$\liminf$"-lower bound: 
\begin{equation}\label{BorneInfRER}
\liminf_{R\to\infty}\frac{1}{2}\int_{\O_{R}}\alpha|\n \varphi_R|^2\geq\frac{1}{2}\int_{\O_{\infty}}\alpha|\n \varphi_\infty|^2.
\end{equation}
On the one hand, for $R\in[R_0,\infty[$, sufficiently large we have $\di\int_{\O_{R}}\alpha|\n \varphi_R|^2\leq C_0+1$ and thus, up to pass to a subsequence, we have $\n\varphi_R\1_{\O_R}$ which weakly converges in $L^2(\R^2\setminus\overline{\o},\R^2)$. 

On the other hand, for a  connected compact set $K\subset\R^2\setminus\o$ s.t. $\p\o\subset\p K$, the test function $\varphi_\infty+\xi_R$ is bounded in $H^1(K)$. 

We let $\chi_R:=\varphi_R-(\varphi_\infty+\xi_R)\in H^1(K)$. It is easy to check that $\tr_{\p\o}(\chi_R)=0$. Then, from a Poincaré type inequality, there exists a constant $C_K>1$ s.t.% [depending on $K$ but independant of $R$]
\[
\|\chi_R\|_{L^2(K)}\leq C_K\|\n \chi_R\|_{L^2(K)}.
\]
Thus
\begin{eqnarray*}
\|\varphi_R\|_{L^2(K)}&\leq& C_K(\|\n \varphi_R\|_{L^2(K)}+\|\n \varphi_\infty\|_{L^2(K)}+\|\n \xi_R\|_{L^2(K)})+\| \varphi_\infty\|_{L^2(K)}+\| \xi_R\|_{L^2(K)}
\\&\leq& \tilde{C}_K.
\end{eqnarray*}
Consequently, with the help of an exhaustion by compacts set and a diagonal extraction process, we have the existence of a sequence $R_k\uparrow\infty$ and $\tilde\varphi_\infty\in H^1_{\rm loc}(\R^2\setminus\o,\R)$ s.t. 
\begin{equation}\label{CasesWeakConvComp}
\begin{cases}
\text{$\varphi_{R_k}\weak\tilde\varphi_\infty$ in $H^1_{\rm loc}(\R^2\setminus\o)$ and $\n\varphi_{R_k}\1_{\O_{R_k}}\weak\n\tilde\varphi_\infty$ in $L^2(\R^2\setminus\overline{\o})$.}
\\
\displaystyle\liminf_{R\to\infty}\int_{\O_R}\alpha|\n \varphi_R|^2=\lim_{R_k\to0}\int_{\O_{R_k}}\alpha|\n \varphi_{R_k}|^2%\geq\int_{\R^2\setminus\overline{\o}}\alpha|\n \tilde\varphi_\infty|^2.
\end{cases}.
\end{equation}
 We thus get $\n\tilde\varphi_\infty\in L^2(\R^2\setminus\overline{\o})$ and $\tr_{\p\o}(\tilde{\varphi}_\infty)=\phi_\infty$, {\it i.e.}, $\tilde\varphi_\infty\in\H_{\phi_\infty}$.
 
 From the definition of $\varphi_\infty$ [Proposition \ref{Prop-DefLimMinSol}] we have
\[
\frac{1}{2}\int_{\O_{\infty}}\alpha|\n \varphi_\infty|^2\leq\frac{1}{2}\int_{\O_{\infty}}\alpha|\n \tilde\varphi_\infty|^2\leq\liminf_{R\to\infty}\frac{1}{2}\int_{\O_{R}}\alpha|\n \varphi_R|^2.
\]
We thus obtained \eqref{BorneInfRER}. Therefore by combining \eqref{BorneSupRER} and \eqref{BorneInfRER} we have:
\begin{equation}\label{ConsEnAux54}
\int_{\O_R}\alpha|\n\varphi_R|^2=\int_{\O_\infty}\alpha|\n\varphi_\infty|^2+o_R(1).
\end{equation} 
The above estimate implies that a limiting map $\tilde\varphi_\infty\in\H_{\phi_\infty}$ as previously obtained satisfies:
\[
\frac{1}{2}\int_{\O_{\infty}}\alpha|\n \tilde\varphi_\infty|^2\leq\frac{1}{2}\int_{\O_{\infty}}\alpha|\n \varphi_\infty|^2.
\]
On the other hand $\varphi_\infty$ is the unique solution of Problem \eqref{MinAuxPahaseExtInf}. Therefore $\tilde\varphi_\infty=\varphi_\infty$. Consequently, the convergences in \eqref{CasesWeakConvComp} hold for $R\to\infty$ and from  \eqref{ConsEnAux54}, these convergences are strong. This ends the proof of the proposition.
\end{proof}

\subsection{Study in the domain $\o_{\rho,{\bf z}}$}\label{EtudeBorneSupomrhoz}

%On fixe une suite $\rho=\rho_n\downarrow0$ tel que $\grz\to \gamma_{0,{\bf z},{\bf d}}$ dans $H^{1/2}(\p\o)$. Notons que puisque l'on a $\int_{\p\o} \grz=0$, alors $\int_{\p\o}\gamma_{0,{\bf z},{\bf d}}=0$.

Recall that we fixed a map $h\in H^{1/2}(\p\o,\S^1)$ s.t. $\deg(h)=d$. We are interested in the minimizing problem
\begin{equation}\label{DecompositionEnergyPbInt}
I_{\rho,{\bf z},{\bf d}}(h)= \inf_{\substack{w\in H^1(\o_{\rho,{\bf z}},\S^1)\\\tr_{\p\o}(w)=h\\\deg_{\p B(z_i,\rho)}(w)=d_i\,\forall\,i}}\frac{1}{2}\int_{\o_{\rho,{\bf z}}}|\n w|^2.
\end{equation}
First note that letting
%\begin{comment} 
\[
g_{{\bf z},{\bf d}}=h\prod_{i=1}^N\left(\dfrac{|x-z_i|}{x-z_i}\right)^{d_i}
\]
we have  $\deg(g_{{\bf z},{\bf d}})={\bf 0}_{\Z^N}$. Thus, from Proposition \ref{PropRappelDeg}.\ref{PropRappelDeg1}, we may fix  $\phi^h_{{\bf z},{\bf d}}\in H^{1/2}(\p\o,\R)$ s.t. $g_{{\bf z},{\bf d}}={\rm e}^{\imath \phi^h_{{\bf z},{\bf d}}}$ and $\di\mint_{\p\o}\phi^h_{{\bf z},{\bf d}}\in[0,2\pi[$. It is clear that $\phi^h_{{\bf z},{\bf d}}$ is uniquely defined.
%\end{section}
\begin{remark}
As in the previous section [see Remark \ref{ModKinou}], for $\beta\in\R$ we have $I_{\rho,{\bf z},{\bf d}}(h)=I_{\rho,{\bf z},{\bf d}}(h\e^{\imath\beta})$. Thus up to replace $h$ by $h\e^{\imath\beta}$, with $\beta=-\di\mint_{\p\o}\phi^h_{{\bf z},{\bf d}}$, in order to estimate $I_{\rho,{\bf z},{\bf d}}(h)$, we may assume that  $\di\mint_{\p\o}\phi^h_{{\bf z},{\bf d}}=0$.
\end{remark}
For $\rho\in[0,\rho_0]$  we let
\[
g_{\rho,{\bf z},{\bf d}}:=h \tr_{\p\o}(\overline{w}_{\rho,{\bf z},{\bf d}})=g_{{\bf z},{\bf d}}\e^{-\imath\grz}
\]
in order to have $h=g_{\rho,{\bf z},{\bf d}}\tr_{\p\o}(w_{\rho,{\bf z},{\bf d}})$. Note that, for $\rho\in[0,\rho_0]$, ${w}_{\rho,{\bf z},{\bf d}}$ is defined in \eqref{DefSolPartDansInclusmkj} and \eqref{DefSolPartLIMITEInclusmkj}.

Thus letting  $\phi_{\rho,{\bf z},{\bf d}}:=\phi^h_{{\bf z},{\bf d}}-\grz\in H^{1/2}(\p\o,\R)$ we have $g_{\rho,{\bf z},{\bf d}}={\rm e}^{\imath \phi_{\rho,{\bf z},{\bf d}}}$ and $\di\mint_{\p\o}\phi_{\rho,{\bf z},{\bf d}}=0$

As in the previous section [Corollary \ref{PropConvDephasageExt}], we easily get  the following convergence result:
\begin{cor}\label{PropConvDephasageInt}
$\phi_{\rho,{\bf z},{\bf d}}\underset{\rho\to0}{\longrightarrow} \phi_{0,{\bf z},{\bf d}}$ in $H^{1/2}(\p\o)$.
\end{cor}
For $\rho\in]0,\rho_0]$ and $w\in H^1(\o_{\rho,{\bf z}},\S^1)$, we have
\[
\tr_{\p\o}(w)=h\Longleftrightarrow w=w_{\rho,{\bf z},{\bf d}}{\rm e}^{\imath \varphi}\text{ with }\left|\begin{array}{c}\varphi\in H^1(\o_{\rho,{\bf z}},\R)\\\tr_{\p\o}(\varphi)=\phi_{\rho,{\bf z},{\bf d}}\end{array}\right..
\]  

From Proposition \ref{PropRappelDeg}.\ref{PropRappelDeg5}, we have the existence of a map $\psi_{\rho,{\bf z},{\bf d}}$ locally defined in $\o_{\rho,{\bf z}}$ [whose gradient is in $L^2(\o_{\rho,{\bf z}},\R^2)$] s.t.  $w_{\rho,{\bf z},{\bf d}}={\rm e}^{\imath\psi_{\rho,{\bf z},{\bf d}}}$.% [$\psi_{\rho,{\bf z},{\bf d}}$ est localement définie tandis que son gradient est globalement défini].

For $w=w_{\rho,{\bf z},{\bf d}}{\rm e}^{\imath \varphi}\in H^1(\o_{\rho,{\bf z}},\S^1)$ we have %tq $\tr_{\p\o}(\varphi)=\phi_R$ on a 
\begin{eqnarray}\nonumber
\frac{1}{2}\int_{\o_{\rho,{\bf z}}}|\n w|^2&=&\frac{1}{2}\int_{\o_{\rho,{\bf z}}}|\n\psi_{\rho,{\bf z},{\bf d}}+\n\varphi|^2
\\\nonumber&=&\frac{1}{2}\int_{\o_{\rho,{\bf z}}}|\n\psi_{\rho,{\bf z},{\bf d}}|^2+\int_{\o_{\rho,{\bf z}}}\n\psi_{\rho,{\bf z},{\bf d}}\cdot\n\varphi+\frac{1}{2}\int_{\o_{\rho,{\bf z}}}|\n\varphi|^2
\\\label{mkdjfhgj11}{[\text{with \eqref{EquationInt}}]}&=&\frac{1}{2}\int_{\o_{\rho,{\bf z}}}|\n w_{\rho,{\bf z},{\bf d}}|^2+\frac{1}{2}\int_{\o_{\rho,{\bf z}}}|\n\varphi|^2.
\end{eqnarray}
Consequently a test function $w=w_{\rho,{\bf z},{\bf d}}{\rm e}^{\imath \varphi}$ with $\tr_{\p\o}(\varphi)=\phi_{\rho,{\bf z},{\bf d}}$ is a solution of the minimizing problem  \eqref{DecompositionEnergyPbInt} if and only if $\varphi\in  H^1(\o_{\rho,{\bf z}},\R)$ is a solution of the minimizing problem 
\begin{equation}\label{MinAuxPahaseIntR}
\inf_{\substack{\varphi\in H^1(\o_{\rho,{\bf z}},\R)\\\tr_{\p\o}(\varphi)=\phi_{\rho,{\bf z},{\bf d}}}}\frac{1}{2}\int_{\o_{\rho,{\bf z}}}|\n \varphi|^2.
\end{equation}
It is easy to get the following proposition:
\begin{prop}
For $\rho\in]0,\rho_0]$, the minimizing Problem \eqref{MinAuxPahaseIntR} admits a unique solution denoted by ${\varphi_{\rho,{\bf z},{\bf d}}}$. Moreover this solution satisfies:
\begin{equation}\nonumber%\label{EqMinAuxPahaseIntR}
\begin{cases}
-\Delta \varphi_{\rho,{\bf z},{\bf d}}=0\text{ in }\o_{\rho,{\bf z}}
\\
\tr_{\p\o}(\varphi_{\rho,{\bf z},{\bf d}})=\phi_{\rho,{\bf z},{\bf d}}
\\
\p_\nu\varphi_{\rho,{\bf z},{\bf d}}=-\p_\nu\Theta\text{ on } \p B(z_i,\rho),\,i=1,...,N
\end{cases}
\end{equation}
where $\Theta$ is defined by Proposition \ref{PropRappelDeg}.\ref{PropRappelDeg5} s.t.
\[
\prod_{i=1}^N\left(\dfrac{x-z_i}{|x-z_i|}\right)^{d_i}=\e^{\imath\Theta}\text{ in }\R^2\setminus\{z_1,...,z_N\}.
\]
\end{prop}
\begin{comment}
Dans le cas $\rho=0$, on pose $\o_{0,z}:=\o$ et 
\[
\tilde\H_{\phi_{0,z}}:=\{\varphi\in H^1(\o,\R)\,|\,\tr_{\p\o}(\varphi)=\phi_{0,z}\}.
\] 
On s'intéresse alors au problème de minimisation 
\begin{equation}\label{MinAuxPahaseIntInf}
\inf_{\substack{\varphi\in \tilde\H_{\phi_{0,z}}}}\frac{1}{2}\int_{\O_\infty}|\n \varphi|^2.
\end{equation}
Il est bien connu que l'unique solution de \eqref{MinAuxPahaseIntInf} est l'extension harmonique de $\phi_{0,z}$ notée $\tilde{\phi}_{0,z}$.
\end{comment}
About the asymptotic behavior of $\varphi_{\rho,{\bf z},{\bf d}}$ we have the following result:
\begin{prop}\label{PropConvEnInt}
When $\rho\to0$,  we have %$\n\varphi_{\rho,{\bf z}}\1_{\o_{\rho,{\bf z}}}\to\n\tilde{\phi}_{0,z}$  dans $L^2(\o)$ 
\[
\dfrac{1}{2}\int_{\o_{\rho,{\bf z}}}|\n\varphi_{\rho,{\bf z},{\bf d}}|^2=\dfrac{1}{2}\int_{\o}|\n\tilde{\phi}_{0,{\bf z},{\bf d}}|^2+o_\rho(1)
\]
where $\tilde{\phi}_{0,{\bf z},{\bf d}}$ is the harmonic extension of ${\phi}_{0,{\bf z},{\bf d}}$ in $\o$.
\end{prop}
\begin{proof}
Let $\xi_\rho$ be the harmonic extension of  $\phi_{0,{\bf z},{\bf d}}-\phi_{\rho,{\bf z},{\bf d}}$ in $\o$. Since $\|\phi_{0,{\bf z},{\bf d}}-\phi_{\rho,{\bf z},{\bf d}}\|_{H^{1/2}(\p\o)}\to0$, we have  $\xi_{\rho}\to0$ in $H^1(\o)$. 

\begin{comment}
On a posé pour $h\in H^{1/2}(\p\o,\C)$, on pose $\|h\|^2_{H^{1/2}(\p\o)}=\|h\|^2_{L^2(\p\o)}+|h|^2_{H^{1/2}(\p\o)}$ avec
\[
|h|^2_{H^{1/2}(\p\o)}:=\dfrac{1}{2}\int_\o|\n \tilde{h}|^2\text{ où $\tilde{h}$ est l'extension harmonique de $h$ dans $\o$.}
\] 
\end{comment}

We now prove the proposition. On the one hand, by minimality of  $\varphi_{\rho,{\bf z},{\bf d}}$ and since $\tr_{\p\o}(\tilde{\phi}_{0,{\bf z},{\bf d}}-\xi_\rho)=\phi_{\rho,{\bf z},{\bf d}}$ we get
\begin{eqnarray}\nonumber
\dfrac{1}{2}\int_{\o_{\rho,{\bf z}}}|\n\varphi_{\rho,{\bf z},{\bf d}}|^2&\leq&\dfrac{1}{2}\int_{\o_{\rho,{\bf z}}}|\n(\tilde{\phi}_{0,{\bf z},{\bf d}}-\xi_\rho)|^2
\\\label{Firstqslkdfh,,}&\leq&\dfrac{1}{2}\int_{\o}|\n\tilde{\phi}_{0,{\bf z},{\bf d}}|^2+o_\rho(1).
\end{eqnarray}
\begin{comment}
D'autre part, on peut facilement montrer que 
\[
\dfrac{1}{2}\int_{\omrz}\{|\n w_{\rho,{\bf z},{\bf d}}|^2+|\n (w_{\rho,{\bf z},{\bf d}}{\rm e}^{\imath\varphi_{\rho,{\bf z},{\bf d}}})|^2\}\leq2\pi d|\ln\rho|+C_0
\]
pour une certaine constante $C_0$ qui dépend de $h$, on déduit immédiatement que pour $i\in\{1,...,d\}$
\begin{eqnarray}\label{CalculStep1mqldkfgj}
\dfrac{1}{2}\int_{B(z_i,\sqrt\rho)\setminus \overline{B(z_i,\rho)}}\{|\n w_{\rho,{\bf z},{\bf d}}|^2+|\n (w_{\rho,{\bf z},{\bf d}}{\rm e}^{\imath\varphi_{\rho,{\bf z},{\bf d}}})|^2\}&\leq&2\pi\ln\dfrac{\sqrt\rho}{\rho}+C_0.
\end{eqnarray}
Ainsi, en remarquant que pour $i\in\{1,...,d\}$
\begin{itemize}
\item[$\bullet$]$|\n w_{\rho,{\bf z},{\bf d}}|^2=|\n\theta_i+\n(\Theta-\theta_i)+\n\grz|^2$ et $|\n(w_{\rho,{\bf z},{\bf d}}{\rm e}^{\imath\varphi_{\rho,{\bf z},{\bf d}}})|^2=|\n\theta_i+\n(\Theta-\theta_i)+\n\grz+\n\varphi_{\rho,{\bf z},{\bf d}}|^2$
\item[$\bullet$]$\displaystyle\int_{B(z_i,\sqrt\rho)\setminus \overline{B(z_i,\rho)}}|\n\theta_i|^2=2\pi|\ln\sqrt\rho|$ et $\displaystyle\int_{B(z_i,\sqrt\rho)\setminus \overline{B(z_i,\rho)}}|\n(\Theta-\theta_i)|^2=\mathcal{O}(\rho)$
\item  pour $\xi\in H^1( B(z_i,\sqrt\rho)\setminus \overline{B(z_i,\rho)},\R)$, on a 
\[
\displaystyle\int_{B(z_i,\sqrt\rho)\setminus \overline{B(z_i,\rho)}}\n\theta_i\cdot\n\xi=0,
\]
\end{itemize} 
\end{comment}
On the other hand, from the Estimate \eqref{Firstqslkdfh,,}, denoting  $C_0:=\displaystyle\int_{\o}|\n\tilde{\phi}_{0,{\bf z},{\bf d}}|^2+1$, for sufficiently small $\rho$ we get 

\begin{comment}
\begin{eqnarray}\label{CalculStep1mqldkfgj-bislih}
\dfrac{1}{2}\int_{B(z,\sqrt\rho)\setminus \overline{B(z,\rho)}}{|\n\grz|^2+|\n\grz+\n\varphi_{\rho,{\bf z},{\bf d}}|^2}&\leq&2C_0.\end{eqnarray}
Ainsi, en utilisant \eqref{CalculStep1mqldkfgj-bislih}, on déduit [quitte à modifier la valeur de $C_0$]
\end{comment}
\begin{equation}\label{TheGoodBornemsDFJ}
\sum_{i=1}^N\dfrac{1}{2}\int_{B(z_i,\sqrt\rho)\setminus \overline{B(z_i,\rho)}}{|\n\varphi_{\rho,{\bf z},{\bf d}}|^2}< C_0.
\end{equation}
Thus for small $\rho$, we get the existence of $\rho'\in]\rho,\sqrt\rho[$ s.t.:
\[
\sum_{i=1}^N\dfrac{1}{2}\int_{0}^{2\pi}| \p_\theta \varphi_{\rho,{\bf z},{\bf d}}(z_i+\rho'{\rm e}^{\imath\theta})|^2\leq \frac{2C_0}{{|\ln{\rho}|}}.
\]
For $i\in\{1,...,N\}$ we let
\[
m_{i,\rho}:=\dfrac{1}{2\pi}\displaystyle\int_0^{2\pi}{\varphi}(z_i+\rho'{\rm e}^{\imath\theta}).
\]
We now define $\tilde{\varphi}\in H^1(\o)$  by $\tilde{\varphi}=\varphi_{\rho,{\bf z},{\bf d}}$ in $\o\setminus{\cup_i\overline{B(z_i,\rho')}}$ and for $x=z_i+s{\rm e}^{\imath\theta}\in B(z_i,\rho')$  [with $i\in\{1,...,N\}$]
\[
\tilde{\varphi}(z_i+s{\rm e}^{\imath\theta})=\left|\begin{array}{cl}\dfrac{2s-\rho'}{\rho'}\varphi_{{\rho,{\bf z},{\bf d}}}(z_i+\rho'{\rm e}^{\imath\theta})+\dfrac{2(\rho'-s)}{\rho'}m_{i,\rho}&\text{ if }s\in]\dfrac{\rho'}{2},\rho'[\\m_{i,\rho}&\text{ if }s\leq\dfrac{\rho'}{2}
\end{array}\right..
\]
A direct calculation [similar to \eqref{KeyEstPolarRRR}] gives for $z\in\{z_1,...,z_N\}$ %[pour une constante $C_0$ indépendante des données du problème]
\[
\int_{B(z,\rho')}|\n \tilde\varphi|^2=\mathcal{O}\left[ \int_{0}^{2\pi}| \p_\theta \varphi_{\rho,{\bf z},{\bf d}}(z+\rho'{\rm e}^{\imath\theta})|^2\right]=o_\rho(1).
\]
Therefore we obtain
\[
\dfrac{1}{2}\int_{\o_{\rho,{\bf z}}}|\n\varphi_{\rho,{\bf z},{\bf d}}|^2\geq\dfrac{1}{2}\int_{\o}|\n\tilde\varphi|^2+o_\rho(1).
\]
But $\tr_{\p\o}(\tilde{\varphi}+\xi_\rho)=\phi_{0,{\bf z},{\bf d}}$ and consequently, from the Dirichlet principle, we have:
\[
\dfrac{1}{2}\int_{\o}|\n(\tilde{\varphi}+\xi_\rho)|^2\geq\dfrac{1}{2}\int_{\o}|\n\tilde{\phi}_{0,{\bf z},{\bf d}}|^2
\]
and thus:
\[
\dfrac{1}{2}\int_{\o}|\n\tilde{\varphi}|^2\geq\dfrac{1}{2}\int_{\o}|\n\tilde{\phi}_{0,{\bf z},{\bf d}}|^2+o_\rho(1).
\]
On the other hand, since $\tilde\varphi=\varphi_{\rho,{\bf z},{\bf d}}$ in $\o\setminus{\cup_i\overline{B(z_i,\rho')}}\subset\o_{\rho,{\bf z}}$ and $\di\dfrac{1}{2}\int_{\cup_iB(z_i,\rho')}|\n\tilde{\varphi}|^2=o_\rho(1)$ we obtain:
\[
\dfrac{1}{2}\int_{\o_{\rho,{\bf z}}}|\n\varphi_{\rho,{\bf z},{\bf d}}|^2\geq\dfrac{1}{2}\int_{\o\setminus{\cup_i\overline{B(z_i,\rho')}}}|\n\varphi_{\rho,{\bf z},{\bf d}}|^2\geq\dfrac{1}{2}\int_{\o}|\n\tilde{\phi}_{0,{\bf z},{\bf d}}|^2+o_\rho(1).
\]
Finally, using \eqref{Firstqslkdfh,,}, by matching upper bound and lower bound we conclude:
\[
\dfrac{1}{2}\int_{\o_{\rho,{\bf z}}}|\n\varphi_{\rho,{\bf z},{\bf d}}|^2=\dfrac{1}{2}\int_{\o}|\n\tilde{\phi}_{0,{\bf z},{\bf d}}|^2+o_\rho(1).
\]
The last estimates ends the proof of the proposition.
%Il suffit alors de  voir que $\n\varphi_{\rho,{\bf z}}\1_{\o_{\rho,{\bf z}}}$ est bornée indépendamment de $\rho$ assez petit dans $L^2(\o)$ [grace à l'estimation précédente] pour en déduire que quitte à extraire on a l'existe de $\ell\in
%$\chi_\rho=\varphi_{\rho,{\bf z}}-\tilde{\phi}_{0,z}+\xi_\rho$
%On commence par montrer que $\varphi_{\rho,{\bf z}}$ converge faiblement dans $H^1_{\rm loc}(\overline{\o}\setminus\{z\})$.
\end{proof}
\subsection{Conclusion}
For $h\in H^{1/2}(\p\o,\S^1)$ s.t. $\deg(h)=d$ we have from \eqref{DecompoExtkjh} and Proposition \ref{P-MinAuxPahaseExtR}:
\begin{equation}\label{FonctionDeH1}
\inf_{\substack{v\in H^1(\O_R,\S^1)\\\tr_{\p\o}(v)=h}}\frac{1}{2}\int_{\O_{R}}\alpha|\n v|^2=\frac{d^2}{2}\int_{\O_{R}}\alpha|\n v_R|^2+\inf_{\substack{\varphi\in \H_{\phi_\infty}}}\frac{1}{2}\int_{\O_\infty}\alpha|\n \varphi|^2+o_R(1).
\end{equation}

Recall that $\phi_\infty$ is defined in Corollary \ref{PropConvDephasageExt} and $\H_{\phi_\infty}$ in \eqref{DefEspaceLimiteExt}.

Using Theorem \ref{THMDevAsyEnergInt},  \eqref{mkdjfhgj11} and Proposition \ref{PropConvEnInt}, letting $\tilde{\phi}_{0,{\bf z},{\bf d}}$ be the harmonic extension of ${\phi}_{0,{\bf z},{\bf d}}$ in $\o$ [recall that ${\phi}_{0,{\bf z},{\bf d}}$ is defined in Corollary \ref{PropConvDephasageInt}], we have
\begin{equation}\label{FonctionDeH2}
\inf_{\substack{w\in H^1(\o_{\rho,{\bf z}},\S^1)\\\tr_{\p\o}(w)=h}}\frac{1}{2}\int_{\o_{\rho,{\bf z}}}|\n w|^2=\left(\sum_id_i^2\right)\pi|\ln\rho|+W({\bf z},{\bf d})+\dfrac{1}{2}\int_{\o}|\n\tilde{\phi}_{0,{\bf z},{\bf d}}|^2+o_\rho(1).
\end{equation}

We let $\K:\left\{h\in H^{1/2}(\p\o,\S^1)\,|\,\deg(h)=d\right\}\to\R^+$ be defined by:
\begin{equation}\label{DefKFiunction}
\K(h):=\inf_{\substack{\varphi\in \H_{\phi_\infty}}}\frac{1}{2}\int_{\O_\infty}\alpha|\n \varphi|^2+\dfrac{b^2}{2}\int_{\o}|\n\tilde{\phi}_{0,{\bf z},{\bf d}}|^2
\end{equation}
and
\[
{f}(R):=\frac{1}{2}\int_{\O_{R}}\alpha|\n v_R|^2
\]
which gives \eqref{FonctionDeH}.

Recall that, without loss of generality, the parameter "$R$" is considered as the major parameter writing $\rho=\rho(R)$. From \eqref{FonctionDeH}, we get for $h\in H^{1/2}(\p\o,\S^1)$ s.t. $\deg(h)=d$:
\begin{equation}\label{CCLBorneSup}
\limsup_{R\to\infty}\left\{ I(R,\rho,{\bf z},{\bf d})-\left[d^2{f}(R)+b^2\left(\sum_id_i^2\pi|\ln\rho|+W({\bf z},{\bf d})\right)\right]\right\}\leq\K(h).
\end{equation}

\section{Lower bound}
In this section we prove the existence of a map $h_\infty\in H^{1/2}(\p\o,\S^1)$ s.t. $\deg(h_\infty)=d$ and

 \begin{eqnarray}\label{EstiTOTALlmkjhsj}
\liminf_{R\to\infty}\left\{ I(R,\rho,{\bf z},{\bf d})-\left[d^2{f}(R)+b^2\left(\sum_id_i^2\pi|\ln\rho|+W({\bf z},{\bf d})\right)\right]\right\}&\geq&\K(h_\infty).% \inf_{\substack{\varphi\in \H_{\phi_\infty}}}\frac{1}{2}\int_{\O_\infty}\alpha|\n \varphi|^2+\dfrac{b^2}{2}\int_{\o}|\n\tilde{\phi}_{0,z}|^2.
\end{eqnarray}
%\subsection{Compacité de la trace des minimiseurs}
We let $R_n\uparrow\infty$ be a sequence which realizes the "$\liminf$" in the left hand side of  \eqref{EstiTOTALlmkjhsj}. %Recall that, without loss of generality, "$R$" is considered as the major parameter writing $\rho=\rho(R)$.

 In order to keep notations simple, we drop the subscript $n$ writing $R=R_n$ when it will not be necessary to specify the dependance on $n$.

Let $u_R$ be a minimizer of \eqref{AuxMicroPb} [Proposition \ref{PropExistResultMainTypPB}]. From Proposition \ref{PropRappelDeg}.\ref{PropRappelDeg4} we may decompose $u_R$ under the form $u_R=v_R^d\e^{\imath\varphi_R}$ where $\varphi_R\in H^1(\O_R,\R)$ and $v_R$ is defined in \eqref{DefSolPartExtInclusmkj}. 

Since $u_R$ is unique up to a multiplicative constant [Proposition \ref{PropExistResultMainTypPB}], we may freeze the non uniqueness by imposing  $\int_{\p\o}\varphi_R=0$.

\begin{comment}
Le but de cette section est de démontrer le résultat suivant :
\begin{prop}\label{PBorneMinkjh}On a $\tr_{\p\o}\varphi_R\in H^1(\p\o)$ et il existe $C$ indépendant de $R$ tel que $\|\p_\tau \varphi_R\|_{L^2(\p\o)}\leq C$.
\end{prop}
 Pour démontrer la proposition \ref{PBorneMinkjh} on commence par montrer que pour on raisonne par l'absurde et on suppose qu'il existe un suite $R=R_n\uparrow\infty$ tq $\|\p_\tau \varphi_R\|_{L^2(\p\o)}\to\infty$.
 \end{comment}

\begin{notation}For  sake of simplicity of the presentation we use the shorthands:
\begin{itemize}
\item[$\bullet$] $"R\in[R_0,\infty["$ to consider an arbitrary term of the sequence $(R_n)_n$; 
\item[$\bullet$] $"R\in[R_0,\infty]"$ to consider an arbitrary term of the sequence $(R_n)_n$ or the limiting case  $R=\infty$.
\end{itemize}
\end{notation}
We denote:
 \begin{itemize}
\item[$\bullet$]  $h_R:=\tr_{\p\o}u_R$, we thus have $h_R=\tr_{\p\o}\left[\left(\dfrac{x}{|x|}\right)^d\e^{\imath(d\gamma_R+\varphi_R)}\right]$ where $\di\int_{\p\o}\varphi_R=0$;
\item[$\bullet$] $\di g_{{\bf z},{\bf d}}:=\tr_{\p\o}\left[\left(\dfrac{|x|}{x}\right)^d\prod_{i=1}^N\left(\dfrac{x-z_i}{|x-z_i|}\right)^{d_i}\right]$. 
\end{itemize}
Since $g_{{\bf z},{\bf d}}\in C^\infty(\p\o,\S^1)$ and $\deg_{\p \o}(g_{{\bf z},{\bf d}})=0$, from Proposition \ref{PropRappelDeg}.\ref{PropRappelDeg1}, we may fix $\deph\in C^\infty(\p\o,\R)$ s.t. 
\begin{equation}\nonumber%\label{DefDephasageBord}
\text{$\e^{\imath\deph}=g_{{\bf z},{\bf d}}$ and $\mint_{\p\o}\deph\in]-2\pi,0]$.}
\end{equation}

%On commence par montrer que $\|\p_\tau \varphi_R\|_{L^2(\p\o)}=\mathcal{O}(1)$.
\subsection{Compatibility conditions}
%\subsubsection{Raccord tangentiel}
We write for $R\in[R_0,\infty[$
\[
h_R:=\tr_{\p\o}u_R=\tr_{\p\o}[v^d_R\e^{\imath\varphi_R}]=\tr_{\p\o}[w_{\rho,{\bf z},{\bf d}}\e^{\imath\varphi_{\rho,{\bf z},{\bf d}}}]
\]
where
\begin{itemize}
\item[$\bullet$] $\di w_{\rho,{\bf z},{\bf d}}=\prod_{i=1}^N\left(\dfrac{x-z_i}{|x-z_i|}\right)^{d_i}{\rm e }^{\imath \gamma_{\rho,{\bf z},{\bf d}}}$ is defined in \eqref{DefSolPartDansInclusmkj};
\item[$\bullet$]  $\varphi_{\rho,{\bf z},{\bf d}}\in H^1(\o_{\rho,{\bf z}},\R)$ is defined by Proposition \ref{PropRappelDeg}.\ref{PropRappelDeg4} s.t. $u_R=w_{\rho,{\bf z},{\bf d}}\e^{\imath\varphi_{\rho,{\bf z},{\bf d}}}$ in $\o_{\rho,{\bf z}}$ and $\di\mint_{\p\o}\varphi_{\rho,{\bf z},{\bf d}}\in[0,2\pi[$. 
\end{itemize}
By using Corollaries \ref{Cor-CVGammaR} and \ref{Cor-CVGammarz}, we have the existence of  $\gamma_\infty,\gamma_{0,{\bf z},{\bf d}}\in H^{1/2}(\p\o,\R)$ s.t.  $\gamma_R\to \gamma_\infty$ and $\grz\to \gamma_{0,{\bf z},{\bf d}}$ dans $H^{1/2}(\p\o)$. It is fundamental to note that   
\begin{itemize}
\item[$\bullet$] $\gamma_\infty$ and $\gamma_{0,{\bf z},{\bf d}}$ are independent of the sequence $(R_n)_n$;
\item[$\bullet$] $\di\int_{\p\o} \gamma_R=\int_{\p\o}\gamma_\infty=\int_{\p\o}\gamma_{0,{\bf z},{\bf d}}=\int_{\p\o} \grz=0$.
\end{itemize}

We have the following equivalences:
\begin{eqnarray}\nonumber
&\e^{\imath \tr_{\p\o}(\varphi_R-\varphi_{\rho,{\bf z},{\bf d}})}=\tr_{\p\o}[w_{\rho,{\bf z},{\bf d}}\overline{v_R^d}]&
\\\nonumber\Leftrightarrow&\e^{\imath \tr_{\p\o}(\varphi_R-\varphi_{\rho,{\bf z},{\bf d}})}=\e^{\imath[\deph+\tr_{\p\o}(\grz-d\gamma_R)]}&
%\\\nonumber\Leftrightarrow&\tr_{\p\o}(\varphi_R-\varphi_{\rho,{\bf z},{\bf d}})=\deph+\tr_{\p\o}(\grz-d\gamma_R)&\:\:[{\rm mod}\,2\pi]%\beta_R\:\:[{\rm mod}\,2\pi]&\text{ avec $\beta_R\in H^{1/2}(\p\o,\R)$ tq $\begin{cases}\tr_{\p\o}[w_\rho\overline{v_R^d}]=\e^{\imath\beta_R}\\\int_{\p\o}\beta_R\in[0,2\pi[\end{cases}$}
\\\label{CondOnK}\Leftrightarrow&\tr_{\p\o}(\varphi_R-\varphi_{\rho,{\bf z},{\bf d}})=\deph+\tr_{\p\o}(\grz-d\gamma_R)+2k_0\pi&\text{ with $k_0\in\Z$.}
%\\\Leftrightarrow&\p_\tau[\tr_{\p\o}(\varphi_R-\varphi_\rho)]=\p_\tau\beta_R&\text{ car $\beta_R\in H^1(\p\o)$}
\end{eqnarray}
%In order to get the last equivalence above we used the \textcolor{red}{Bethuel-Zeng Lemma}.

We thus have%On fixe $\varphi_{\rho,{\bf z},{\bf d}}\in H^1(\o_{\rho,{\bf z}},\R)$ tel que $h_R=\tr_{\p\o}[w_{\rho,{\bf z},{\bf d}}\e^{\imath\varphi_{\rho,{\bf z},{\bf d}}}]$ et
\[
-\mint_{\p\o}\varphi_{\rho,{\bf z},{\bf d}}=\mint_{\p\o}(\varphi_R-\varphi_{\rho,{\bf z},{\bf d}})=\mint_{\p\o}[\deph+\tr_{\p\o}(\grz-d\gamma_R)+2k_0\pi]=2k_0\pi+\mint_{\p\o}\deph.
\]
Since $\di \mint_{\p\o}\varphi_{\rho,{\bf z},{\bf d}}\in[0,2\pi[$ and $\di\mint_{\p\o}\deph\in]-2\pi,0]$, the above equalities imply that $k_0=0$ in  \eqref{CondOnK}.

Consequently we get:
\begin{equation}\label{ConditionCompatibilitePhasejj}
\tr_{\p\o}(\varphi_R-\varphi_{\rho,{\bf z},{\bf d}})=\deph+\tr_{\p\o}(\grz-d\gamma_R).
\end{equation}
\subsection{Asymptotic estimate of the energy}
\begin{comment}
Pour $R>1$ on note $u_R$ un minimiseur de \eqref{AuxMicroPb}. On rappelle que 
\begin{itemize}
\item$v_R$ est la solution spéciale dans $\O_R$ [définie dans \eqref{DefSolPartExtInclusmkj}]. On a $v_R=\dfrac{x}{|x|}\e^{\imath\gamma_R}$ dans $\O_R$ avec $\displaystyle\int_{\p\o}\gamma_R=0$;
\item $w_{\rho,{\bf z},{\bf d}}$ est la solution spéciale dans $\o_{\rho,{\bf z}}$ [définie dans \eqref{DefSolPartDansInclusmkj}]. On a 
\[
w_{\rho,{\bf z},{\bf d}}=\prod_{i=1}^d\left(\dfrac{x-z_i}{|x-z_i|}\right)^{d_i}{\rm e }^{\imath \gamma_{\rho,{\bf z},{\bf d}}}\text{ avec }\int_{\p\o}\grz=0.
\]
\end{itemize}
Pour $R\to\infty$ on a :
\begin{itemize}
\item $\tr_{\p\o}\gamma_R\to \gamma_\infty$ dans $H^{1/2}(\p\o)$,
\item $\tr_{\p\o}\grz\to \gamma_{0,{\bf z},{\bf d}}$ dans $H^{1/2}(\p\o)$.
%\item $\tr_{\p\o}\varphi_R\to\varphi_\infty$ dans $H^{1/2}(\p\o)$
\end{itemize}
\end{comment}
By using \eqref{DecompoExtkjh} and \eqref{mkdjfhgj11}, we have the following decoupling:

\begin{eqnarray}\nonumber
I(R,\rho,{\bf z},{\bf d})&=&\frac{1}{2}\int_{\dom_{R,{\bf z}}}\alpha|\n u_R|
\\\nonumber&=&\frac{1}{2}\int_{\O_R}\alpha|\n (v^d_R\e^{\imath\varphi_R})|^2+\frac{b^2}{2}\int_{\o_{\rho,{\bf z}}}|\n w_{\rho,{\bf z},{\bf d}}\e^{\imath\varphi_{\rho,{\bf z},{\bf d}}}|^2
\\\label{EstiIntlmkjhsj1}&=&d^2f(R)+\frac{1}{2}\int_{\O_R}\alpha|\n \varphi_R|^2+\frac{b^2}{2}\int_{\o_{\rho,{\bf z}}}|\n w_{\rho,{\bf z},{\bf d}}|^2+\frac{b^2}{2}\int_{\o_{\rho,{\bf z}}}|\n \varphi_{\rho,{\bf z},{\bf d}}|^2.
\end{eqnarray}

Recall that we denoted: ${f}(R)=\displaystyle\frac{1}{2}\int_{\O_R}\alpha|\n v_R|^2$.

\begin{comment} et on a utilisé  le Théorème \ref{THMDevAsyEnergInt} afin d'obtenir
\begin{equation}\nonumber%\label{EstiIntlmkjhsj3}
\frac{1}{2}\int_{\o_{\rho,{\bf z}}}|\n w_{\rho,{\bf z},{\bf d}}|^2=\sum_id_i^2\pi|\ln\rho|+W({\bf z},{\bf d})+o_\rho(1).
\end{equation}
\end{comment}
%\subsection{Préliminaires}
From the minimality of $u_R$ and by using \eqref{CCLBorneSup}, letting $C_0:=\K\left(\dfrac{x^d}{|x|^d}\right)+1$, for sufficiently large $R$, we have: 
\begin{equation}\label{EstBor,neSupePhaseExtlkhjb}
 I(R,\rho,{\bf z},{\bf d})-\left[d^2f(R)+\frac{b^2}{2}\int_{\o_{\rho,{\bf z}}}|\n w_{\rho,{\bf z},{\bf d}}|^2\right]=\frac{1}{2}\int_{\O_R}\alpha|\n \varphi_R|^2+\frac{b^2}{2}\int_{\o_{\rho,{\bf z}}}|\n \varphi_{\rho,{\bf z},{\bf d}}|^2\leq C_0.
\end{equation}
Since $\displaystyle\mint_{\p\o} \varphi_R=0$ [resp. $\displaystyle\mint_{\p\o} \varphi_{\rho,{\bf z},{\bf d}}\in[0,2\pi[$] for $K_1$ a connected compact set of $\R^2\setminus\o$ [resp. $K_2$ a connected compact set of $\overline{\o}\setminus\{z_1,...,z_N\}$] s.t. $\p\o\subset\p K_1$ [resp. $\p\o\subset\p K_2$], there exists $C_1>0$ [resp. $C_2>0$] s.t. for large $R$ we have
\[
\int_{K_1}|\varphi_R|^2\leq C_1\int_{K_1}|\n\varphi_R|^2\!\left[\text{resp. }\int_{K_2}|\varphi_{\rho,{\bf z},{\bf d}}|^2\leq C_2\int_{K_2}|\n\varphi_{\rho,{\bf z},{\bf d}}|^2\right].
\]

Consequently :
\begin{itemize}
\item $(\varphi_R)_R$ is bounded in $H^1_{\rm loc}(\R^2\setminus\o)$. Thus there exists $\varphi_\infty\in H^1_{\rm loc}(\R^2\setminus\o)$ s.t., up to pass to a subsequence, we have  
\begin{equation}\label{AddedEq1}
\text{$\varphi_R\weak\varphi_\infty$ in $H^1_{\rm loc}(\R^2\setminus\o)$. }
\end{equation}
\item $(\varphi_{\rho,{\bf z},{\bf d}})_R$ is bounded in $H^1_{\rm loc}(\overline{\o}\setminus\{z_1,...,z_N\}$.  Thus there exists $\varphi_{0,{\bf z},{\bf d}}\in H^1_{\rm loc}(\overline{\o}\setminus\{z_1,...,z_N\})$ s.t., up to pass to a subsequence, we have  
\begin{equation}\label{AddedEq2}
\text{$\varphi_{\rho,{\bf z},{\bf d}}\weak\varphi_{0,{\bf z},{\bf d}}$ in $H^1_{\rm loc}(\overline{\o}\setminus\{z_1,...,z_N\})$.}
\end{equation}
\end{itemize}
From \eqref{ConditionCompatibilitePhasejj}, we have 
\[
\tr_{\p\o}(\varphi_R-\varphi_{\rho,{\bf z},{\bf d}})=\deph+\tr_{\p\o}(\grz-d\gamma_R)
\]
where $\deph+\tr_{\p\o}(\grz-d\gamma_R)$ is strongly converging to $\deph+\tr_{\p\o}(\gamma_{0,{\bf z},{\bf d}}-d\gamma_\infty)$ in $H^{1/2}(\p\o)$ [Corollaries \ref{Cor-CVGammaR}$\&$\ref{Cor-CVGammarz}]. Consequently we get the same for   $\tr_{\p\o}(\varphi_R-\varphi_{\rho,{\bf z},{\bf d}})$, namely $\tr_{\p\o}(\varphi_R-\varphi_{\rho,{\bf z},{\bf d}})$ is strongly convergent in $H^{1/2}(\p\o)$ to  
\[
\tr_{\p\o}(\varphi_\infty-\varphi_{0,{\bf z},{\bf d}})=\deph+\tr_{\p\o}(\gamma_{0,{\bf z},{\bf d}}-d\gamma_\infty).
\]
We thus may deduce:
\[
\e^{\imath\tr_{\p\o}(\varphi_\infty-\varphi_{0,{\bf z},{\bf d}})}=\e^{\imath[\deph+\tr_{\p\o}(\gamma_{0,{\bf z},{\bf d}}-d\gamma_\infty)]}
\]
{\it i.e.} 
\begin{equation}\label{Compatiblitelimklj}
\left(\dfrac{x}{|x|}\right)^d\e^{\imath\tr_{\p\o}(d\gamma_\infty+\varphi_\infty)}=\prod_{i=1}^N\left(\frac{x-z_i}{|x-z_i|}\right)^{d_i}\e^{\imath\tr_{\p\o}(\gamma_{0,{\bf z},{\bf d}}+\varphi_{0,{\bf z},{\bf d}})}.
\end{equation}
We now define:
\[
h_\infty:=\tr_{\p\o}\left[\left(\dfrac{x}{|x|}\right)^d\e^{\imath(d\gamma_\infty+\varphi_\infty)}\right]\in H^{1/2}(\p\o,\S^1).
\]
It is clear that $\deg(h_\infty)=d$.

We prove in the three next subsections [Sections \ref{SubSectI}$\&$\ref{SubSectII}$\&$\ref{SubSectIII}] that $h_\infty$ satisfies \eqref{EstiTOTALlmkjhsj}.
\subsection{Calculations in $\R^2\setminus\overline{\o}$}\label{SubSectI}
From \eqref{EstBor,neSupePhaseExtlkhjb}, we get that $\n \varphi_R\1_{\O_R}$ is bounded in $L^2(\R^2\setminus\overline{\o})$ and thus, up to pass to a subsequence,  $\n \varphi_{R}\1_{\O_{R}}$ weakly converge in $L^2(\R^2\setminus\overline{\o})$. Consequently, we may improve the convergence in \eqref{AddedEq1}, up to pass to a subsequence, we obtain that $\n \varphi_{R}\1_{\O_{R}}\weak\n\varphi_\infty$ in $L^2(\R^2\setminus\overline{\o})$. In particular we obtain $\n\varphi_\infty\in L^2(\R^2\setminus\overline{\o})$.

Consequently, denoting $\phi_\infty:=\tr_{\p\o}(\varphi_\infty)$ we obtain  $\varphi_\infty\in\H_{\phi_\infty}$ [see \eqref{DefEspaceLimiteExt} for the definition of $\H_{\phi_\infty}$]. Therefore, letting  $\O_\infty=\R^2\setminus\overline{\o}$, we have:
\begin{eqnarray}\nonumber
\liminf_{R_n\to\infty}\left\{\frac{1}{2}\int_{\O_{R_n}}\alpha|\n u_{R_n}|^2-\frac{d^2}{2}\int_{\O_{R_n}}\alpha|\n v_{R_n}|^2\right\}&=&\liminf_{R_n\to\infty}\frac{1}{2}\int_{\O_{R_n}}\alpha|\n \varphi_{R_n}|^2
\\\nonumber&\geq&\frac{1}{2}\int_{\O_\infty}\alpha|\n \varphi_\infty|^2
\\\label{ConclusionBorneInfExt}&\geq&\inf_{\substack{\varphi\in \H_{\phi_\infty}}}\frac{1}{2}\int_{\O_\infty}\alpha|\n \varphi|^2.
\end{eqnarray}
%Puisque le minorant obtenue dans \eqref{ConclusionBorneInfExtP} est indépendant de la suite $(R_n)_n$ on obtient
%\begin{equation}\label{ConclusionBorneInfExt}
%\liminf_{R\to\infty}\left\{\frac{1}{2}\int_{\O_R}\alpha|\n u_R|^2-\frac{1}{2}\int_{\O_R}\alpha|\n v_R|^2\right\}\geq \inf_{\substack{\varphi\in \H_{\phi_\infty}}}\frac{1}{2}\int_{\O_\infty}\alpha|\n \varphi|^2
%\end{equation}
\subsection{Calculations on $\o$}\label{SubSectII}
We continue the calculations by proving:
%On continue les calculs en considérant une suite $(R_n)$ qui réalise la limite inférieure dans le membre de gauche de \eqref{EstiTOTALlmkjhsj} et tel qu'en notant $R=R_n\&\rho=\rho(R_n)$ on  $\varphi_R\weak\varphi_\infty$ dans $H^1_{\rm loc}(\R^2\setminus\o)$ et $\varphi_{\rho,{\bf z},{\bf d}}\weak\varphi_{0,{\bf z},{\bf d}}$ dans $H^1_{\rm loc}(\overline{\o}\setminus\{z_1,...,z_N\})$.
\begin{comment}
\subsubsection{Estimation de l'énergie de $u_R$ dans $\o_{\rho,{\bf z}}$}
Quitte à multiplier $u_R$ par une constante multiplicative, on peut supposer que dans $\o_{\rho,{\bf z}}$ on a 
\[
u_R=w_{\rho,{\bf z},{\bf d}}\e^{\imath\varphi_{\rho,{\bf z},{\bf d}}}\text{ avec }\int_{\p\o}\varphi_{\rho,{\bf z},{\bf d}}=0.
\]
En utilisant la proposition \ref{PBorneMinkjh} avec une inégalité de Poincaré-Wirtinger, quitte à passer à une sous suite on a $\tr_{\p\o}\varphi_{\rho,{\bf z},{\bf d}}\to\varphi_{0,{\bf z},{\bf d}}$ dans $H^{1/2}(\p\o)$.

Ainsi, en procédant comme dans la Section \ref{EtudeBorneSupomrhoz}, on a le découplage
\begin{eqnarray}\label{EstiIntlmkjhsj1}
\frac{1}{2}\int_{\o_{\rho,{\bf z}}}|\n u_R|^2&=&\frac{1}{2}\int_{\o_{\rho,{\bf z}}}|\n w_{\rho,{\bf z},{\bf d}}|^2+\frac{1}{2}\int_{\o_{\rho,{\bf z}}}|\n \varphi_{\rho,{\bf z},{\bf d}}|^2
\end{eqnarray}

En utilisant exactement le même argument que pour la démonstration de la proposition \ref{PropConvEnInt}, on a\end{comment}

 %On démontre à présent que
\begin{equation}\label{EstiIntlmkjhsj2}
\frac{1}{2}\int_{\o_{\rho,{\bf z}}}|\n \varphi_{\rho,{\bf z},{\bf d}}|^2\geq\frac{1}{2}\int_{\o}|\n \tilde\phi_{0,{\bf z},{\bf d}}|^2+o_{\rho}(1)
\end{equation}
where $\tilde\phi_{0,{\bf z},{\bf d}}$ is the harmonique extension of $\phi_{0,{\bf z},{\bf d}}:=\tr_{\p\o}\varphi_{0,{\bf z},{\bf d}}$ in $\o$, $\varphi_{0,{\bf z},{\bf d}}$ is defined in \eqref{AddedEq2}.

In order to get \eqref{EstiIntlmkjhsj2}, we adapt the argument done to prove Proposition \ref{PropConvEnInt}. 

From \eqref{EstBor,neSupePhaseExtlkhjb}, we have
\begin{equation}\nonumber%\label{TheGoodBornemsDFJBisBiosmojh}
\sum_{i=1}^N\dfrac{1}{2}\int_{B(z_i,\sqrt\rho)\setminus \overline{B(z_i,\rho)}}{|\n\varphi_{\rho,{\bf z},{\bf d}}|^2}\leq C_0.
\end{equation}
Thus, from a mean value argument, there exists $\rho'\in]\rho,\sqrt\rho[$ s.t. %[up to modify the value of $C_0$ independently of $\rho$] 
\[
\sum_{i=1}^N\dfrac{1}{2}\int_{0}^{2\pi}| \p_\theta \varphi_{\rho,{\bf z},{\bf d}}(z_i+\rho'{\rm e}^{\imath\theta})|^2{\rm d}\theta\leq \frac{2C_0}{{|\ln{\rho}|}}.
\]
We now define  $\tilde{\varphi_\rho}\in H^1(\o)$  by $\tilde{\varphi_\rho}=\varphi_{\rho,{\bf z},{\bf d}}$ in $\o\setminus\cup_i\overline{B(z_i,\rho')}$ and for $i\in\{1,...,N\}\,\&\,x=z_i+s{\rm e}^{\imath\theta}\in B(z_i,\rho')$ we let  
\[
\tilde{\varphi_\rho}(z_i+s{\rm e}^{\imath\theta})=\left|\begin{array}{cl}2\dfrac{s-\rho'/2}{\rho'}\varphi_{{\rho,{\bf z},{\bf d}}}(z_i+\rho'{\rm e}^{\imath\theta})+\dfrac{\rho'-s}{\pi\rho'}\displaystyle\int_0^{2\pi}{\varphi_{\rho,{\bf z},{\bf d}}}(z_i+\rho'{\rm e}^{\imath\theta}){\rm d}\theta&\text{ if }s\in]\dfrac{\rho'}{2},\rho'[\\\dfrac{1}{2\pi}\displaystyle\int_0^{2\pi}{\varphi_{\rho,{\bf z},{\bf d}}}(z_i+\rho'{\rm e}^{\imath\theta}){\rm d}\theta&\text{ if }s\leq\dfrac{\rho'}{2}
\end{array}\right..
\]
A direct calculation gives: 
\begin{equation}\label{DurDurLesMathsljk3}
\sum_{i=1}^N\int_{B(z_i,\rho')}|\n \tilde\varphi_\rho|^2=\mathcal{O}\left[ \sum_{i=1}^N\int_{0}^{2\pi}| \p_\theta \varphi_{\rho,{\bf z},{\bf d}}(z_i+\rho'{\rm e}^{\imath\theta})|^2\right]=o_\rho(1).
\end{equation}
Thus, letting $\o_{\rho',{\bf z}}=\o\setminus\cup_{i=1}^N\overline{B(z_i,\rho')}$ and $\tilde{\dom}_{\rho'}=\cup_{i=1}^N B(z_i,\rho')\setminus\overline{B(z_i,\rho)}$, we obtain:
\begin{eqnarray}\nonumber
\int_{\o_{\rho,{\bf z}}}|\n {\varphi_{\rho,{\bf z},{\bf d}}}|^2&=& \int_{\o_{\rho',{\bf z}}}|\n \tilde{\varphi_\rho}|^2+\int_{\tilde{\dom}_{\rho'}}|\n {\varphi_{\rho,{\bf z},{\bf d}}}|^2
\\\nonumber&\geq&\int_{\o_{\rho',{\bf z}}}|\n \tilde{\varphi_\rho}|^2
\\\label{DurDurLesMathsljk1}&\stackrel{\eqref{DurDurLesMathsljk3}}{=}&\int_{\o}|\n \tilde{\varphi_\rho}|^2+o_\rho(1).
\end{eqnarray}
Since $\tilde{\varphi_\rho}$ is bounded in $H^1(\o)$, up to pass to a subsequence, we may assume the existence of $\tilde{\varphi_0}\in H^1(\o)$ s.t. $\tilde{\varphi_\rho}\weak\tilde{\varphi_0}$ in $H^1(\o)$. 
%Puisque $\tilde{\varphi_\rho}=\varphi_{\rho,{\bf z},{\bf d}}$ dans $\o\setminus\cup_{i=1}^N\overline{B(z_i,\rho')}$, on a $\tilde{\varphi_0}=\varphi_{0,{\bf z},{\bf d}}$ dans $\o\setminus\{z_1,...,z_N\}$. En particulier

On the other hand, it is clear that  $\tr_{\p\o}\tilde{\varphi_0}=\tr_{\p\o}{\varphi}_{0,{\bf z},{\bf d}}=\phi_{0,{\bf z},{\bf d}}$. Consequently from the Dirichlet principle we get%En invoquant la semi continuité inférieure de la fonctionnelle de Dirichlet ainsi que le principe de Dirichlet on obtient%[puisque $\tr_{\p\o}(\tilde{\varphi_R})=\tr_{\p\o}({\varphi}_R)$] on obtient 
\begin{equation}\label{DurDurLesMathsljk2}
\liminf_{\rho=\rho(R_n)\to0}\int_{\o}|\n \tilde{\varphi_\rho}|^2\geq\int_{\o}|\n \tilde{\varphi_0}|^2 \geq\int_{\o}|\n\tilde\phi_{0,{\bf z},{\bf d}}|^2.
\end{equation}

%afin de conclure à partir  que 
%\begin{equation}\label{DurDurLesMathsljk4}
%\int_{\o}|\n \tilde{\varphi}|^2=\int_{\o_{\rho,{\bf z}}}|\n \tilde{\varphi}|^2+o_{\rho}(1).
%\end{equation}
By combining \eqref{DurDurLesMathsljk1} and  \eqref{DurDurLesMathsljk2}  we obtain \eqref{EstiIntlmkjhsj2}.

From \eqref{EstiIntlmkjhsj1} and \eqref{EstiIntlmkjhsj2} we may write [denoting $\rho_n=\rho(R_n)$]

\begin{eqnarray}\nonumber
\liminf_{\rho_n\to0}\left\{\frac{1}{2}\int_{\o_{\rho_n,{\bf z}}}|\n u_{R_n}|^2-\frac{1}{2}\int_{\o_{\rho_n,{\bf z}}}|\n w_{\rho_n,{\bf z},{\bf d}}|^2\right\}&=&\liminf_{\rho_n\to0}\frac{1}{2}\int_{\o_{\rho_n,{\bf z}}}|\n \varphi_{\rho_n,{\bf z},{\bf d}}|^2
\\\label{ConclusionBorneInfInt}&\geq&\frac{1}{2}\int_{\o}|\n \tilde\phi_{0,{\bf z},{\bf d}}|^2.
%\\\label{ConclusionBorneInfExt}&\geq&\inf_{\substack{\varphi\in \H_{\phi_\infty}}}\frac{1}{2}\int_{\O_\infty}\alpha|\n \varphi|^2.
\end{eqnarray}

\begin{comment}
\begin{equation}\label{EstiIntlmkjhsj4}
\frac{1}{2}\int_{\o_{\rho,{\bf z}}}|\n u_R|^2=\pi|\ln\rho|+W({\bf z},{\bf d})+\frac{1}{2}\int_{\o}\alpha|\n \tilde\phi_{0,{\bf z},{\bf d}}|^2+o_{\rho}(1).
\end{equation}
\end{comment}
\subsection{Conclusion}\label{SubSectIII}
Using \eqref{ConclusionBorneInfExt}, \eqref{ConclusionBorneInfInt}, the definition of the sequence $(R_n)_n$ and letting $\displaystyle{f}(R)=\frac{1}{2}\int_{\O_{R}}\alpha|\n v_R|^2$ we get 
\begin{eqnarray}\nonumber
&&\liminf_{R\to\infty}\left\{ I(R,\rho,{\bf z},{\bf d})-\left(d^2{f}(R)+\frac{b^2}{2}\int_{\o_{\rho,{\bf z}}}|\n w_{\rho,{\bf z},{\bf d}}|^2\right)\right\}
\\\nonumber&=&\lim_{R_n\to\infty}\left\{ \dfrac{1}{2}\int_{\dom_{R_n,{\bf z}}}\alpha|\n u_{R_n}|^2-\left(d^2{f}(R_n)+\frac{b^2}{2}\int_{\o_{\rho_n,{\bf z}}}|\n w_{\rho_n,{\bf z},{\bf d}}|^2\right)\right\}
\\\nonumber&\geq& \liminf_{R_n\to\infty}\left\{\frac{1}{2}\int_{\O_{R_n}}\alpha|\n u_{R_n}|^2-d^2f(R_n)\right\}+b^2\liminf_{\rho_n\to0}\left\{\frac{1}{2}\int_{\o_{\rho_n,{\bf z}}}|\n u_{R_n}|^2-\frac{1}{2}\int_{\o_{\rho_n,{\bf z}}}|\n w_{\rho_n,{\bf z},{\bf d}}|^2\right\}
\\\label{AIEOUILLE1}&\geq& \inf_{\substack{\varphi\in \H_{\phi_\infty}}}\frac{1}{2}\int_{\O_\infty}\alpha|\n \varphi|^2+\dfrac{b^2}{2}\int_{\o}|\n\tilde{\phi}_{0,{\bf z},{\bf d}}|^2.
\end{eqnarray}
Recall that  
\[
h_\infty:=\left(\dfrac{x}{|x|}\right)^d\e^{\imath(d\gamma_\infty+\phi_\infty)}\in H^{1/2}(\p\o,\S^1).
\]
 Therefore from \eqref{DefKFiunction} and \eqref{Compatiblitelimklj} we may write
\[
\K(h_\infty)=\inf_{\substack{\varphi\in \H_{\phi_\infty}}}\frac{1}{2}\int_{\O_\infty}\alpha|\n \varphi|^2+\dfrac{b^2}{2}\int_{\o}|\n\tilde{\phi}_{0,{\bf z},{\bf d}}|^2.
\]
Consequently \eqref{AIEOUILLE1} becomes 
\begin{equation}\label{AIEOUILLE2}
\liminf_{R\to\infty}\left\{ I(R,\rho,{\bf z},{\bf d})-\left(d^2{f}(R)+\frac{b^2}{2}\int_{\o_{\rho,{\bf z}}}|\n w_{\rho,{\bf z},{\bf d}}|^2\right)\right\}\geq\K(h_\infty).
\end{equation}
It suffices now to see that, from Theorem \ref{THMDevAsyEnergInt} we have
\begin{equation}\nonumber%\label{EstiIntlmkjhsj3}
\frac{1}{2}\int_{\o_{\rho,{\bf z}}}|\n w_{\rho,{\bf z},{\bf d}}|^2=\sum_id_i^2\pi|\ln\rho|+W({\bf z},{\bf d})+o_\rho(1),
\end{equation} this combined with \eqref{CCLBorneSup} gives
\[
\lim_{R\to\infty}\left\{ I(R,\rho,{\bf z},{\bf d})-\left[{f}(R)+b^2\left(\pi\left(\sum_{i=1}^Nd_i^2\right)|\ln\rho|+W({\bf z},{\bf d})\right)\right]\right\} \text{ exists}
\] and 
\[
\lim_{R\to\infty}\left\{ I(R,\rho,{\bf z},{\bf d})-\left[{f}(R)+b^2\left(\pi\left(\sum_{i=1}^Nd_i^2\right)|\ln\rho|+W({\bf z},{\bf d})\right)\right]\right\}=\K(h_\infty).
\]
\begin{remark}\label{Remark.IndConst}
It is direct to get that  $h_\infty$ is a minimizer of $\K: \{h\in H^{1/2}(\p\o,\S^1)\,|\,\deg(h)=d\}\to\R$.
\end{remark}
We now define:
\begin{equation}\label{DefMicrEnREn}
W^{\rm micro}({\bf z},{\bf d}):= b^2W({\bf z},{\bf d})+\min_{\substack{h\in H^{1/2}(\p\o,\S^1)\\\deg(h)=d}}\K(h)
\end{equation}
in order to write
\[
I(R,\rho,{\bf z},{\bf d})=d^2f(R)+b^2\pi\left(\sum_{i=1}^Nd_i^2\right)|\ln\rho|+W^{\rm micro}({\bf z},{\bf d})+o_\rho(1).
\]
The last equality ends the proof of Theorem \ref{ThmMicroEstm}.

%\end{document}
\section{The case of the radially symmetric impurity: $\o=\D$}\label{RadialSection}
%\section{Les données du problème circulaire}
In this section we focus on the circular case with   $\o=\D=B(0,1)$ is the unit disc and for $b\in(0,\infty)$ we let 
\[
\begin{array}{cccc}\alpha:&\R^2&\to&\{b^2;1\}\\&x&\mapsto&\begin{cases}b^2&\text{if }x\in\D\\1&\text{if }x\in\R^2\setminus\D\end{cases}.
\end{array}
\] 

We fix
\begin{itemize}
%\item[$\bullet$] $b\in(0,\infty)$ et on note $\alpha:\R^2\to\{b^2,1\},\,x\mapsto\alpha(x)=\begin{cases}b^2&\text{dans }\D\\1&\text{dans }\R^2\setminus\D\end{cases}$,
\item[$\bullet$] $N\in\N^*$, ${\bf d}=(d_1,...,d_N)\in\Z^N$ and we let $ \di d:=\sum_{i=1}^N d_i\in\Z$;%\neq0$
\item[$\bullet$]${\bf z}\in\dst:=\{(z_1,...,z_N)\in\D^N\,|\,z_i\neq z_j\text{ for }i\neq j\}$.
\end{itemize}
\subsection{Explicit expression of the special solutions}
We use the same notations as in Section \ref{SectionSpeSol}. 
%\begin{comment}
\begin{notation}
In this section and in the next sections, in order to keep notations simple, we use the shorthand "$x$" to stand the identity map. Namely we use the abuse of notation ${\rm Id}=x$ where ${\rm Id}:U\to U,\,x\mapsto{\rm Id}(x)=x$  and $U\subset\R^2\simeq\C$ is an arbitrary set .
\end{notation}
%\end{comment}
We let $v_\infty$ be the limiting function obtained in Corollary \ref{Cor-CVGammaR}. It is easy to prove that  $v_\infty(x)=\dfrac{x}{|x|}$, {\it i.e.} $\gamma_\infty\equiv0$.

We let $\di w_{0,{\bf z},{\bf d}}=\prod_{i=1}^N\left(\dfrac{x-z_i}{|x-z_i|}\right)^{d_i}{\rm e }^{\imath \gamma_{0,{\bf z},{\bf d}}}$ be the function defined in \eqref{DefSolPartLIMITEInclusmkj}. This function is the canonical harmonique map in  $\D$ associated to the singularities $({\bf z},{\bf d})$. 

On the unit circle $\S^1$ we have $\tr_{\S^1}(w_{0,{\bf z},{\bf d}})=\e^{\imath\psi_{0,{\bf z},{\bf d}}}$ with 
\[
\p_\tau \psi_{0,{\bf z},{\bf d}}=\p_\nu\left[\sum_{j=1}^Nd_j\left(\ln|x-z_j|-\ln|1-\overline{z_j}x|\right)\right].%\text{ [en identifiant $\R^2$ à $\C$]}
\]
This result comes from \cite{LamyMironescu2014} Eq. (2.25) et (4.1).

From (4.14) in  \cite{LamyMironescu2014} we have
\[
\p_\tau \psi_{0,{\bf z},{\bf d}}=\sum_{j=1}^Nd_j\left[2\p_\nu\left(\ln|x-z_j|\right)-1\right].
\]
Thus
\[
\p_\tau \psi_{0,{\bf z},{\bf d}}=\sum_{j=1}^Nd_j\left[2\p_\tau\left({\rm arg}(x-z_j)\right)-1\right]
\]
with $\dfrac{x-z_j}{|x-z_j|}=\e^{\imath {\rm arg}(x-z_j)}$.

Consequently we get
\begin{equation}\label{ExpFormSpeSolInt}
\tr_{\S^1}(w_{0,{\bf z},{\bf d}})=\e^{\imath\psi_{0,{\bf z},{\bf d}}}={\rm Cst}\times x^{-d}\prod_{j=1}^N\left(\dfrac{x-z_j}{|x-z_j|}\right)^{2d_j}
\end{equation}
where ${\rm Cst}\in\S^1$ is a constant.
\subsection{Expression of the dephasing}
%\begin{notation}
%Until the end of this subsection ${\rm Cst}\in\R$ or ${\rm Cst}\in\S^1$ is a constant which may change line to line and which is independent of the variable $\sigma=\e^{\imath\theta}\in\S^1$. 
%\end{notation}
For $h_\infty\in H^{1/2}(\S^1,\S^1)$ we have [see \eqref{DefEspaceLimiteExt} and \eqref{DefKFiunction}]
\[
\K(h_\infty)=\inf_{\substack{\varphi\in \H_{\phi_\infty}}}\frac{1}{2}\int_{\O_\infty}|\n \varphi|^2+\dfrac{b^2}{2}\int_{\o}|\n\tilde{\phi}_{0}|^2,
\]
where:
\begin{itemize}
\item[$\bullet$] on the unit circle we have
\begin{equation}\label{CondCompExplicit}
h_\infty=x^d\e^{\imath\phi_\infty}=w_{0,{\bf z},{\bf d}}\e^{\imath\phi_{0}},
\end{equation}
\item[$\bullet$] $\tilde{\phi}_{0}$ is the harmonic extension of  ${\phi}_{0}$ in $\D$.
\end{itemize}

Condition \eqref{CondCompExplicit} is a compatibility condition between the fonctions $\phi_\infty$ et $\phi_0$. It is clear that from the definition of  $\K$ we may slightly modify Condition \eqref{CondCompExplicit} by imposing 
\begin{equation}\label{CondCompExplicitbis}
\left(\dfrac{x}{|x|}\right)^d\e^{\imath\phi_\infty}={\rm Cst}\times w_{0,{\bf z},{\bf d}}\e^{\imath\phi_{0}}\text{ with }{\rm Cst}\in\S^1.
\end{equation}
We may easily prove that
\[
\inf_{\substack{\varphi\in \H_{\phi_\infty}}}\frac{1}{2}\int_{\O_\infty}|\n \varphi|^2=\frac{1}{2}\int_{\O_\infty}|\n \hat{\phi}_\infty|^2
\] where for  $\phi\in H^{1/2}(\S^1,\R)$, $\hat{\phi}\in H^1_{\rm loc}(\R^2\setminus\overline{\D})$ is the unique solution of
\begin{equation}\nonumber%\label{DefExtHarmExt}
\begin{cases}
-\Delta\varphi=0&\text{in }\R^2\setminus\overline{\D}\\\tr_{\S^1}(\varphi)=\phi,&\n\varphi\in L^2(\R^2\setminus\overline{\D})
\end{cases}.
\end{equation}
[See Proposition \ref{BasicPropFourier} for more details about  $\hat{\phi}$]

From \eqref{ExpFormSpeSolInt}, an equivalent reformulation of  \eqref{CondCompExplicitbis}  is
\[
{\rm Cst}\prod_{j=1}^N\left(\dfrac{x-z_j}{|x-z_j|\times x}\right)^{2d_j}=\e^{\imath(\phi_\infty-\phi_0)}\text{ with }{\rm Cst}\in\S^1.
\]
The above condition is equivalent to the compatibility condition:
\begin{equation}\label{COmpatibilielkjuklhjb}
\phi_\infty-\phi_0=\Deph_{{\bf z},{\bf d}}+{\rm Cst}\text{ where }{\rm Cst}\in\R
\end{equation}
with  $\Deph_{{\bf z},{\bf d}}\in C^\infty(\S^1,\R)$ which is a lifting of 
\[
\prod_{j=1}^N\left(\dfrac{x-z_j}{|x-z_j|\times x}\right)^{2d_j}.
\]
Here we used Proposition \ref{PropRappelDeg}.\ref{PropRappelDeg1} and the smoothness of $\e^{\imath \Deph_{{\bf z},{\bf d}}}$.

With a direct calculation, for $z_0\in\D$ and $x\in\S^1$, we have
\[
\left(\dfrac{x-z_0}{|x-z_0|x}\right)^{2}=\dfrac{x-z_0}{\overline{x-z_0}\times x^2}=\dfrac{x-z_0}{1-\overline{z_0} x}\times\dfrac{1}{x}=M_{z_0}(x)\times\dfrac{1}{x}
\]
where $M_{z_0}:\D\to\D$ is the Moebius function defined by  $M_{z_0}(x)=\dfrac{x-z_0}{1-\overline{z_0} x}$.

In \cite{dos2015microscopic}, it is proved [Section 7] that if $z_0\in\D\cap\R^+$ then for $\e^{\imath\theta}\in\S^1$
\[
M_{z_0}(\e^{\imath\theta})\e^{-\imath\theta}=\e^{\Deph_{z_0,1}(\e^{\imath\theta})}\text{ where } \Deph_{z_0,1}(\e^{\imath\theta})=\sum_{n\in\Z^*}\frac{z_0^{|n|}}{\imath n}\e^{\imath n\theta}+{\rm Cst},\,{\rm Cst}\in\R.
\]
%Here ${\rm Cst}\in\R$ is a constant independent of $ \e^{\imath\theta}\in\S^1$.

In the general case $z_0=t\e^{\imath \gamma}\in\D$ [with $t\geq0, \gamma\in\R$] we easily deduce from the previous equality:
\begin{eqnarray*}
M_{z_0}(\e^{\imath\theta})\e^{-\imath\theta}&=&\dfrac{\e^{\imath\theta}-t\e^{\imath \gamma}}{(1-t\e^{-\imath \gamma} \e^{\imath\theta})\e^{\imath\theta}}
\\&=&\dfrac{\e^{\imath(\theta-\gamma)}-t}{(1-t \e^{\imath(\theta-\gamma)})\e^{\imath(\theta-\gamma)}}
\\&=&M_{t}[\e^{\imath(\theta-\gamma)}]\e^{-\imath(\theta-\gamma)}.
\end{eqnarray*}
Then  
\begin{eqnarray}\nonumber
 \Deph_{z_0,1}(\e^{\imath\theta})&=& \Deph_{t,1}(\e^{\imath(\theta-\gamma)})+{\rm Cst}
 \\\nonumber&=&\sum_{n\in\Z^*}\frac{t^{|n|}}{\imath n}\e^{\imath n(\theta-\gamma)}+{\rm Cst}
 \\\nonumber&=&\sum_{n\in\N^*}\left[\frac{\overline{z_0}^{n}}{\imath n}\e^{\imath n\theta}-\frac{{z_0}^{n}}{\imath n}\e^{-\imath n\theta}\right]+{\rm Cst},\,{\rm Cst}\in\R.
\end{eqnarray}
%Here ${\rm Cst}\in\R$ is a constant independent of $ \e^{\imath\theta}\in\S^1$.

It is easy to prove that we have  $\Deph_{{\bf z},{\bf d}}=\sum_{j=1}^Nd_j \Deph_{z_0,1}+{\rm Cst}$ [${\rm Cst}\in\R$] and then 
\begin{equation}\label{DecopFourierDeph}
\Deph_{{\bf z},{\bf d}}(\e^{\imath\theta})={\rm Cst}+\sum_{n\in\N^*}\sum_{j=1}^Nd_j\left[\frac{\overline{z_j}^{n}}{\imath n}\e^{\imath n\theta}-\frac{{z_j}^{n}}{\imath n}\e^{-\imath n\theta}\right].
\end{equation}
We are now in position to reformulate the compatibility condition \eqref{COmpatibilielkjuklhjb}. 

Let $\phi_0,\phi_\infty\in H^{1/2}(\S^1,\S^1)$, consider their Fourier decompositions:
\begin{equation}\label{FourierDec}
\phi_0(\e^{\imath\theta})=\sum_{n\in\Z}c_{0,n}\e^{\imath n\theta}\text{ and }\phi_\infty(\e^{\imath\theta})=\sum_{n\in\Z}c_{\infty,n}\e^{\imath n\theta}.
\end{equation}
We have
\begin{align}
&&\left(\dfrac{x}{|x|}\right)^d\e^{\imath\phi_\infty}={\rm Cst} \times w_{0,{\bf z},{\bf d}}\e^{\imath\phi_{0}},\,{\rm Cst}\in\S^1\tag{\ref{CondCompExplicitbis}}\\&\Leftrightarrow&\phi_\infty-\phi_0=\Deph_{{\bf z},{\bf d}}+{\rm Cst},\,{\rm Cst}\in\R\tag{\ref{COmpatibilielkjuklhjb}}
\\&\Leftrightarrow&\forall\,n\in\Z^*,\, c_{\infty,n}-c_{0,n}=\begin{cases}\di\sum_{j=1}^N d_j\frac{\overline{z_j}^{n}}{\imath n}&\text{if $n>0$}\\\di-\sum_{j=1}^N d_j\frac{{z_j}^{n}}{\imath n}&\text{if $n<0$}\end{cases}.
\end{align}
\subsection{Explicit expression of the minimal value of $\K$ }
For $\phi_\infty,\phi_0\in H^{1/2}(\S^1,\R)$ we use Notation \eqref{FourierDec} for their Fourier coefficients:
\begin{itemize}
\item the Fourier coefficients of  $\phi_\infty$ are denoted by $(c_{\infty,n})_{n\in\Z}$,
\item the Fourier coefficients of  $\phi_0$ are denoted by $(c_{0,n})_{n\in\Z}$.
\end{itemize}

Before going further we recall some basic facts. 
\begin{prop}\label{BasicPropFourier}
Let $\phi\in H^{1/2}(\S^1,\R)$ and consider $\phi(\e^{\imath\theta})=\sum_{n\in\Z} c_n\e^{\imath n\theta}$ be its Fourier decomposition.

Then we have
\begin{enumerate}
\item\label{FourierProp1} $c_n=\overline{c_{-n}}$
\item\label{FourierProp2}  $\di\sum_{n\in\Z}|n||c_n|^2<\infty$ and we may choose the quantity $\di\sqrt{\pi\sum_{n\in\Z}|n||c_n|^2}$ as a semi-norm in $H^{1/2}(\S^1,\R)$.
\item\label{FourierProp3}  The map
\[
\begin{array}{cccc}
\tilde{\phi}:&\D&\to&\R\\&r\e^{\imath\theta}&\mapsto&\di\sum_{n\in\Z} c_n r^{|n|}\e^{\imath n\theta}
\end{array}
\]
is the harmonic extension of $\phi$. Moreover
\[
\dfrac{1}{2}\int_\D|\n\tilde \phi|^2=\pi{\sum_{n\in\Z}|n||c_n|^2}.
\]
\item\label{FourierProp4}  The map
\[
\begin{array}{cccc}
\hat{\phi}:&\R^2\setminus\overline{\D}&\to&\R\\&r\e^{\imath\theta}&\mapsto&\di\sum_{n\in\Z} c_n r^{-|n|}\e^{\imath n\theta}
\end{array}
\]
is an exterior harmonic extension of $\phi$. Moreover
\[
\dfrac{1}{2}\int_{\R^2\setminus\overline{\D}}|\n\hat \phi|^2=\pi{\sum_{n\in\Z}|n||c_n|^2}.
\]
\item\label{FourierProp5}  $\hat{\phi}$ is the unique solution of 
\begin{equation}\label{MinExteriorHarmExt}
\begin{cases}
-\Delta\varphi=0\text{ in }\R^2\setminus\overline{\D},
\\\varphi\in H^1_{\rm loc}(\R^2\setminus\D,\R)
\\\tr_{\S^1}(\varphi)=\phi,\,\n\varphi\in L^2(\R^2\setminus\overline{\D},\R^2)
\end{cases}.
\end{equation}
Therefore it is also the unique solution of the problem
\begin{equation}\nonumber\label{MinExteriorHarmExtPB}
\inf_{\substack{\varphi\in H^1_{\rm loc}(\R^2\setminus\D,\R)\\\tr_{\S^1}(\varphi)=\phi,\,\n\varphi\in L^2(\R^2\setminus\overline{\D},\R^2)}}\dfrac{1}{2}\int_{\R^2\setminus\overline{\D}}|\n \varphi|^2.
\end{equation}
\end{enumerate}
\end{prop}

\begin{proof}
Assertions \ref{FourierProp1} and \ref{FourierProp2} are quite standard. Assertions \ref{FourierProp3} and \ref{FourierProp4} follow from standard calculations.

We now prove Assertion \ref{FourierProp5}. Let $\phi\in H^{1/2}(\S^1,\R)$ and let $\hat\phi$ be defined by Assertion \ref{FourierProp4}. It is clear that $\hat\phi$ solves \eqref{MinExteriorHarmExt}.  Assume that $\varphi_0$ is a solution of \eqref{MinExteriorHarmExt} and let $\eta:=\hat\phi-\varphi_0$. Then $\eta$ satisfies:
\begin{equation}\nonumber
\begin{cases}
-\Delta\eta=0\text{ in }\R^2\setminus\overline{\D},
\\\eta\in H^1_{\rm loc}(\R^2\setminus\D,\R)
\\\tr_{\S^1}(\eta)=0,\,\n\eta\in L^2(\R^2\setminus\overline{\D},\R^2)
\end{cases}.
\end{equation}
From \cite{simader1996dirichlet} [Theorem II.6.2.ii] we get $\eta=0$. This clearly gives the uniqueness of the solution of \eqref{MinExteriorHarmExt}.

On the one hand, by direct minimization we know that Problem \ref{MinExteriorHarmExtPB} admits solution(s). It is standard to check that a minimizer for \eqref{MinExteriorHarmExtPB} solves \eqref{MinExteriorHarmExt}. Consequently $\hat\phi$ is the unique solution of \eqref{MinExteriorHarmExtPB}. 
\end{proof}
\begin{notation}
From now on, for $\phi\in H^{1/2}(\S^1,\R)$ with Fourier decomposition $\phi(\e^{\imath\theta})=\sum_{n\in\Z} c_n\e^{\imath n\theta}$, we let  
\[
|\phi|_{H^{1/2}}:=\sqrt{\pi{\sum_{n\in\Z}|n||c_n|^2}}=\sqrt{2\pi{\sum_{n\in\N^*}|n||c_n|^2}}=\sqrt{{\dfrac{1}{2}\int_{\R^2\setminus\overline\D}|\n\hat\phi|^2}}=\sqrt{\dfrac{1}{2}\int_{\D}|\n\tilde\phi|^2}.
\]
\end{notation}
\begin{comment}
Thus, using the fact that functions $\phi_0$ and $\phi_\infty$ are real valued we have:
\[
\dfrac{1}{2}\int_{\D}|\n\tilde\phi_0|^2=|\phi_0|_{H^{1/2}}^2=\pi\sum_{n\in\Z}|n||c_{0,n}|^2=2\pi\sum_{n\in\N^*}n|c_{0,n}|^2
\]
and 
\[
\dfrac{1}{2}\int_{\R^2\setminus\overline{\D}}|\n\hat\phi_\infty|^2=|\phi_\infty|_{H^{1/2}}^2=\pi\sum_{n\in\Z}|n||c_{\infty,n}|^2=2\pi\sum_{n\in\N^*}n|c_{\infty,n}|^2.
\]
\end{comment}
%On a utilisé que puisque $\phi_0$ et $\phi_\infty$ sont à valeurs réelles on a $c_{0,n}=\overline{c_{0,-n}}$ et $c_{\infty,n}=\overline{c_{\infty,-n}}$ pour $n\in\N$.
For $n\in\N^*$, letting $\gamma_n=\sum_{j=1}^N d_j\frac{\overline{z_j}^{n}}{\imath n}$, {\it i.e.} $\Deph_{{\bf z},{\bf d}}(\e^{\imath\theta})={\rm Cst}+\sum_{n\in\Z^*}\gamma_n\e^{\imath n\theta}$ [see \eqref{DecopFourierDeph}], we get
\begin{eqnarray}\nonumber
\inf_{\substack{h\in H^{1/2}(\S^1,\S^1)\\\deg(h)=d}}\K(h)&=&\inf_{\substack{\phi_0,\phi_\infty\in H^{1/2}(\S^1,\R)\\x^d\e^{\imath\phi_\infty}={\rm Cst} \times w_{0,{\bf z},{\bf d}}\e^{\imath\phi_{0}}}}\left(\frac{1}{2}\int_{\O_\infty}|\n \hat\phi_\infty|^2+\dfrac{b^2}{2}\int_{\o}|\n\tilde{\phi}_{0}|^2\right)
\\\nonumber&=&2\pi\inf_{\substack{(c_{0,n})_{n\in\N^*},(c_{\infty,n})\in \ell^2(\N^*)\\c_{\infty,n}-c_{0,n}=\gamma_n\,\forall\,n\in\N^*}}\left(\sum_{n\in\N}n|c_{0,n}|^2+b^2\sum_{n\in\N}n|c_{\infty,n}|^2\right)
\\\nonumber&=&2\pi\sum_{n\in\N^*}\left[n\times\inf_{\substack{c_{0,n},c_{\infty,n}\in\C\\c_{\infty,n}-c_{0,n}=\gamma_n\,\forall\,n\in\N^*}}\left(|c_{0,n}|^2+b^2|c_{\infty,n}|^2\right)\right]
\\\nonumber&=&2\pi\sum_{n\in\N^*}\left[n\times\inf_{\substack{c_{0,n}\in\C}}\left(|c_{0,n}|^2+b^2|c_{0,n}+\gamma_n|^2\right)\right]
\\\nonumber&=&2\pi\sum_{n\in\N^*}\left[n\times\left(\left|\dfrac{-b^2}{1+b^2}\gamma_n\right|^2+b^2\left|\dfrac{-b^2}{1+b^2}\gamma_n+\gamma_n\right|^2\right)\right]
\\\label{ExpreMinKInf}&=&\frac{b^2}{1+b^2}2\pi\sum_{n\in\N^*}n|\gamma_n|^2=\frac{b^2}{1+b^2}|\Deph_{{\bf z},{\bf d}}|_{H^{1/2}}^2.
\end{eqnarray}
\subsection{Explicit expression of $W^{\rm micro}$: Proof of Proposition \ref{Prop-Exp-Circ}}\label{ProofOf-Prop-Exp-Circ}
We first recall the expression of $W({\bf z},{\bf d})$ [see Proposition 1 in \cite{LR1}]:
\[
W({\bf z},{\bf d})=-\pi\sum_{i\neq j}d_id_j\ln|z_i-z_j|+\pi\sum_{i=1}^Nd_i^2\ln(1-|z_i|^2)+\pi\sum_{i\neq j}d_id_j\ln|1-z_i\overline{z_j}|.
\]
From \eqref{DefMicrEnREn} we have
\[
W^{\rm micro}({\bf z},{\bf d})= b^2W({\bf z},{\bf d})+\min_{\substack{h\in H^{1/2}(\S^1,\S^1)\\\deg(h)=d}}\K(h).
\]
By combining \eqref{DecopFourierDeph} and \eqref{ExpreMinKInf} we may write
\[
\min_{\substack{h\in H^{1/2}(\S^1,\S^1)\\\deg(h)=d}}\K(h)=\frac{2b^2}{1+b^2}\pi\sum_{n\in\N^*}n\left|\sum_{j=1}^N d_j\frac{\overline{z_j}^{n}}{\imath n}\right|^2=\frac{2b^2}{1+b^2}\pi\sum_{n\in\N^*}\dfrac{1}{n}\left|\sum_{j=1}^N d_j{z_j^{n}}\right|^2.
\]
For $n\in\N^*$ we have the following expanding
\[
\left|\sum_{j=1}^N d_j{z_j^{n}}\right|^2=\sum_{j=1}^N d_j^2|z_j|^{2n}+\sum_{i\neq j}d_id_j(z_i\overline{z_j})^n=\sum_{j=1}^N d_j^2|z_j|^{2n}+2{\rm Re}\left[\sum_{i< j}d_id_j(z_i\overline{z_j})^n\right].
\]

Therefore we obtain
\begin{eqnarray*}
\sum_{n\in\N^*}\dfrac{1}{n}\left|\sum_{j=1}^N d_j{z_j^{n}}\right|^2&=&\sum_{j=1}^N d_j^2\left(\sum_{n\in\N^*}\dfrac{1}{n}|z_j|^{2n}\right)+
2\sum_{i< j}d_id_j{\rm Re}\left[\sum_{n\in\N^*}\dfrac{1}{n}(z_i\overline{z_j})^n\right]
\\&=&-\sum_{j=1}^N d_j^2\ln(1-|z_j|^2)-2\sum_{i< j}d_id_j{\rm Re}\left[\ln(1-z_i\overline{z_j})\right]
\\&=&-\sum_{j=1}^N d_j^2\ln(1-|z_j|^2)-\sum_{i\neq j}d_id_j\ln|1-z_i\overline{z_j}|.
\end{eqnarray*}
We may thus conclude:
\begin{eqnarray*}
W^{\rm micro}({\bf z},{\bf d})&=& b^2\pi\left[-\sum_{i\neq j}d_id_j\ln|z_i-z_j|+\sum_{i=1}^Nd_i^2\ln(1-|z_i|^2)+\sum_{i\neq j}d_id_j\ln|1-z_i\overline{z_j}|-\right.\\&&\phantom{aaaaaaaffff}\left.-\frac{2}{1+b^2}\left(\sum_{j=1}^N d_j^2\ln(1-|z_j|^2)+\sum_{i\neq j}d_id_j\ln|1-z_i\overline{z_j}|\right)\right]
\\&=&-b^2\pi\left[\sum_{i\neq j}d_id_j\ln|z_i-z_j|+\dfrac{1-b^2}{1+b^2}\sum_{j=1}^N d_j^2\ln(1-|z_j|^2)+\dfrac{1-b^2}{1+b^2}\sum_{i\neq j}d_id_j\ln|1-z_i\overline{z_j}|\right].
\end{eqnarray*}
These calculations end the proof of Proposition \ref{Prop-Exp-Circ}.
\subsection{Minimization of $W^{\rm micro}$ in some particular cases}\label{PartMinCasPartDisque}
We first claim that if ${\bf d}={\bf 0}_{\Z^N}$ then $W^{\rm micro}(\cdot,{\bf d})\equiv0$. In the following we consider ${\bf d}\in\Z^N\setminus\{{\bf 0}_{\Z^N}\}$.
\subsubsection{The case $N=1$ and the case $N\geq2\&\exists! k_0\in\{1,...,N\}$ s.t. $d_{k_0}\neq0$}%\,\&\,d_l=0$ $l\neq k_0$} 
We first treat the case  $N=1$. In this situation, we have for $z\in\D$ and $d\in\Z$ :
\[
W^{\rm micro}({ z},{ d})=-\dfrac{b^2(1-b^2)}{1+b^2}\pi d^2\ln(1-|z|^2)
\]
Therefore, if $b<1$ then $z=0$ is the unique minimizer of $W^{\rm micro}$. 

\begin{remark} This simple fact is the main result of \cite{dos2015microscopic} [where the explicit expression of $W^{\rm micro}$ was unknown].
\end{remark}

If $b=1$ then  $W^{\rm micro}(\cdot,{ d})\equiv0$.

If $b>1$ then 
\[
W^{\rm micro}({ z},{ d})=\dfrac{b^2(b^2-1)}{1+b^2}\pi d^2\ln(1-|z|^2).
\]
Consequently $W^{\rm micro}({ z},{ d})\to-\infty$ when $|z|\to1$. This implies that $W^{\rm micro}({ \cdot},{ d})$ does not admit minimizers.

\begin{remark}
We may conclude that the condition $b<1$ creates a confinement effect for the points of minimum of $W^{\rm micro}({ \cdot},{ d})$. This confinement effect does not hold for $b\geq1$.
\end{remark}

We now consider the case $N\geq2$. We assume that  $d_1\neq0$ and $d_l=0$ for $l\neq1$. 

This situation is similar to the above one since for ${\bf z}=(z_1,...,z_N)\in\ost$ we have $W^{\rm micro}({\bf z},{\bf d})=W^{\rm micro}({ z_1},{ d_1})$. Consequently as previously we have:
\begin{itemize}
\item If $b<1$ then the set of global minimizers of $W^{\rm micro}$ is $\{{\bf z}\in\ost\,|\,z_1=0\}$. 
\item If $b=1$ then  $W^{\rm micro}(\cdot,{\bf d})\equiv0$.
\item If $b>1$ then $W^{\rm micro}({\bf z},{\bf d})\to-\infty$ when $|z_1|\to1$.
\end{itemize}

\subsubsection{The case $N\geq2$ and there exist $k, l$ s.t. $d_kd_l<0$}
Let  ${\bf d}\in\Z^N$ s.t. there exist $k\neq l$ satisfying $d_kd_l<0$. 

In this situation we have 
\begin{eqnarray*}
\inf_{{\bf z}\in\ost}W^{\rm micro}({\bf z},{\bf d})&=& -\infty.
\end{eqnarray*}
Indeed, without loss of generality, we may assume that $d_1d_2<0$. We thus consider $z_1^{(n)}=-1/n$,  $z_2^{(n)}=1/n$ and for $k\in\{1,...,N\}\setminus\{1,2\}$, $z_k=\e^{\imath2k\pi/N}/2$. 

With direct calculations, we obtain $\lim_n W({\bf z}_n,{\bf d})=-\infty$.
\begin{remark}
This fact underline that if we impose $d_1d_2<0$ then the main part of the {\it optimal} energy $I(R,\rho,{\bf z},{\bf d})$ is not 
\[
\left(\sum_{i=1}^Nd_i\right)^2f(R)+b^2\left(\sum_{i=1}^Nd_i^2\right)|\ln\rho|.
\]
Indeed when we consider {\it very near} singularities $z_1\&z_2$ we may optimize the divergent term $b^2\left(\sum_{i=1}^Nd_i^2\right)|\ln\rho|$. The key argument is that with degrees having different signs ({\it e.g} $d_1d_2<0$) we have
\[
\sum_{i=1}^Nd_i^2>(d_1+d_2)^2+\sum_{i=3}^Nd_i^2.
\]
This is an example of the standard attractive effect of singularities having degrees with different signs.
\end{remark}
\subsubsection{The case $b=1$, $N\geq2$, $d_k d_l\geq0$ $\forall k,l$ and there exist $k_0,l_0$ s.t.  $d_{k_0}d_{l_0}>0$}
When $b=1$, for $({\bf z},{\bf d})\in\ost\times\Z^N$ we have
\begin{eqnarray*}
W^{\rm micro}({\bf z},{\bf d})&=& -\pi\sum_{i\neq j}d_id_j\ln|z_i-z_j|.
\end{eqnarray*}
Thus
\[
\inf_{{\bf z}\in\ost}W^{\rm micro}({\bf z},{\bf d})>-\infty
\]
but the lower bound is not attained.

Indeed, it is easy to check for ${\bf z}\in\ost$
\[
0>\inf_{{\bf z}\in\ost}W^{\rm micro}({\bf z},{\bf d})> -\pi\sum_{i\neq j}d_id_j\ln2.
\]
Consequently $W^{\rm micro}(\cdot,{\bf d})$ is bounded from below.

We now prove that the lower bound is not reached.  Let ${\bf z}\in\ost$, and consider $\tilde{\bf z}\in\ost$ be s.t. $\tilde{z}_k=\lambda z_k$ with $\lambda:=\dfrac{2}{1+\max\{|z_l|,l\in\{1,...,N\}\}}$. It is easy to check that  $\tilde{\bf z}\in\ost$.

We get 
\[
W^{\rm micro}(\tilde{\bf z},{\bf d})= W^{\rm micro}({\bf z},{\bf d})-\pi\ln\lambda\sum_{i\neq j}d_id_j.
\]
Since $\lambda>1$, we have $W^{\rm micro}(\tilde{\bf z},{\bf d})< W^{\rm micro}({\bf z},{\bf d})$. This fact implies that the lower bound is not reached.
\begin{remark}
When $b=1$, the impurity $\o=\D$ does not play any role. Then, due to the standard repulsion effect between vortices, the more the vortices are distant the smaller the energy. Consequently, for fixed degrees having all the same sign, minimal sequences of singularities go to the boundary of the impurity which is not an admissible configuration in this framework.
\end{remark}
\subsubsection{The case $b>1$ and $N\geq2$}
If $b>1$ then for $({\bf z},{\bf d})\in\ost\times\Z^N$ we have
\begin{eqnarray*}
W^{\rm micro}({\bf z},{\bf d})&=& b^2\pi\left[-\sum_{i\neq j}d_id_j\ln|z_i-z_j|+\dfrac{b^2-1}{1+b^2}\sum_{j=1}^N d_j^2\ln(1-|z_j|^2)+\dfrac{b^2-1}{1+b^2}\sum_{i\neq j}d_id_j\ln|1-z_i\overline{z_j}|\right]
\end{eqnarray*}
Taking, for $k\in\{1,...,N\}$, $z^{(n)}_k=(1-{1}/{n})\e^{\imath{2\pi k}/{N}}$ we have 
\[
W^{\rm micro}({\bf z},{\bf d})=\mathcal{O}(1)+\dfrac{b^2-1}{1+b^2}\sum_{j=1}^N d_j^2\ln(1-|z_j|^2)\to-\infty\text{ when }n\to\infty.
\]
\begin{remark}
The case $b>1$ corresponds to an impurity $\o=\D$ which have a repulsive effect on the singularities.
\end{remark}
\subsubsection{The case $0<b<1$, $N\geq2$ and ${\bf d}\in\N^N$}

This situation is the most challenging. 

Note that with the help of \cite{Publi3} we may obtain the existence of minimizers for $W^{\rm micro}(\cdot,{\bf d})$ with $d_i=1$ for $i\in\{1,...,N\}$, $N\in\N^*$. But \cite{Publi3} does not give any information on the location of minimizers and for other configurations of degrees.

From technical issues, we restrict the study to $N=2$ and $p,q\in\N^*$. Note that the case $p,q<0$ is obviously symmetric.

We are going to prove that there exist minimizers and there are unique up to a rotation [see \eqref{ExprSComplique}$\&$\eqref{ExprSCompliqueComplete}].

We may assume $p\leq q$. For $z_1,z_2\in\D$ we have, writing $({\bf z},{\bf d})=((z_1,p),(z_2,q))$
\begin{eqnarray*}
\dfrac{W^{\rm micro}({\bf z},{\bf d})}{-b^2\pi}&=&2pq\ln|z_1-z_2|+\dfrac{1-b^2}{1+b^2}\left[p^2\ln(1-|z_1|^2)+q^2\ln(1-|z_2|^2)+2pq\ln|1-z_1\overline{z_2}|\right].
\end{eqnarray*}

We let:
\begin{itemize}
\item[$\bullet$] $\B:=\dfrac{1-b^2}{1+b^2}$ and $\A:=\dfrac{p}{q}\leq1$;
\item[$\bullet$] $f(z_1,z_2)=2\ln|z_1-z_2|+\B\left[\A\ln(1-|z_1|^2)+\A^{-1}\ln(1-|z_2|^2)+2\ln|1-z_1\overline{z_2}|\right]$.
\end{itemize}
Note that $W^{\rm micro}[(z_1,z_2),(p,q)]=-b^2pq\pi f(z_1,z_2)$. Consequently, in order to study minimizing points of $W^{\rm micro}[\cdot,(p,q)]$, we have to maximize $f(\cdot)$.

Since $z_1\neq z_2$ and since for $t\in\R$ we have $f(z_1,z_2)=f(z_1\e^{\imath t},z_2\e^{\imath t})$, we may assume that  $z_1=s\geq0$. We thus have for $z_2=\rho\e^{\imath\theta}$ [$0\leq\rho<1,\theta\in\R$]
\[
f(s,\rho\e^{\imath\theta})=\ln\left[s^2+\rho^2-2s\rho\cos\theta\right]+\B\left[\A\ln(1-s^2)+\A^{-1}\ln(1-\rho^2)+\ln(1+s^2\rho^2-2s\rho\cos\theta)\right].
\]
We first claim that if $s=0$ then $\rho>0$ and for $\v>0$ we have
\[
f(\v,-\rho)=f(0,\rho\e^{\imath\theta})+\v({\rho}^{-1}+2\beta\rho)+\mathcal{O}(\v^2).
\]
Consequently, for $\v>0$ sufficiently small we have $f(\v,-\rho)>f(0,\rho\e^{\imath\theta})$. Therefore, if $(s,\rho\e^{\imath\theta})$ maximizes $f$, then $s\in(0;1)$. Using a similar argument, we may prove that for $s>0$, if $(s,\rho\e^{\imath\theta})$ maximizes $f$, then $\rho\in(0;1)$.

On the other hand, from direct checking, for $s,\rho>0$, the map  $\theta\in[0,2\pi]\mapsto f(s,\rho\e^{\imath\theta})$  is maximal if and only if $\theta=\pi$.
 
 Consequently, we focus on the map 
 \[
 \begin{array}{cccc}
 g:&(0;1)^2&\to&\R\\&(s,t)&\mapsto&f(s,-t)=2\ln\left(s+t\right)+\B\left[\A\ln(1-s^2)+\A^{-1}\ln(1-t^2)+2\ln(1+st)\right]
 \end{array}.
 \]
 We first look for critical points of $g$:
\begin{eqnarray}\nonumber
\n g(s,t)={\bf 0}&\Leftrightarrow&\begin{cases}\displaystyle\dfrac{1}{s+t}+\B\left(\dfrac{-\A s}{1-s^2}+\dfrac{t}{1+st}\right)=0\\\displaystyle\dfrac{1}{s+t}+\B\left(\dfrac{-\A^{-1}t}{1-t^2}+\dfrac{s}{1+st}\right)=0\end{cases}
\\\label{SystCrPoint1}
&\Leftrightarrow&\begin{cases}\displaystyle(1-s^2)(1+st)+\B\left[-\A s(1+st)(s+t)+{t}(1-s^2)(s+t)\right]=0\\\displaystyle(1-t^2)(1+st)+\B\left[-\A^{-1} t(1+st)(s+t)+{s}(1-t^2)(s+t)\right]=0\end{cases}
\end{eqnarray}
By considering the difference of both lines in \eqref{SystCrPoint1} we get:
\begin{eqnarray}\nonumber
&&(t^2-s^2)(1+st)+\B\left[(\A^{-1} t-\A s)(1+st)(s+t)+(t-s^2t-s+st^2)(s+t)\right]=0
\\\nonumber
&\Longleftrightarrow&(1+st)(s+t)\left[t-s+\B((\A^{-1}+1) t-(\A+1) s)\right]=0
\\\nonumber
&\stackrel{[s,t>0]}{\Longleftrightarrow}&[1+\B(\A^{-1}+1)] t-[1+\B(\A+1)] s=0
\\\label{SystCrPoint2}
&{\Longleftrightarrow}&t=\lambda s\text{ with }\lambda:=\frac{1+\B(\A+1)}{1+\B(\A^{-1}+1)}.
\end{eqnarray} 
\begin{remark}
It is important to note that $0<\lambda\leq1$. Moreover $\lambda=1$ if and only if $p=q$.
\end{remark}
Using \eqref{SystCrPoint2} in the first line of  \eqref{SystCrPoint1} we have
\begin{equation}\label{LasetEqkjb}
(1-s^2)(1+\lambda s^2)+\B\left[-\A s^2(1+\lambda s^2)(1+\lambda)+{\lambda s^2}(1-s^2)(1+\lambda)\right]=0.
\end{equation}
Thus, letting $\sigma=s^2$, we get the following equation:
\begin{equation}\label{PrePreExprSComplique}
[\lambda+(\A+1)\B\lambda(1+\lambda)]\sigma^2+[1-\lambda+(\A-\lambda)\B(1+\lambda)]\sigma-1=0
\end{equation}
We let
\[
\Delta:=[1-\lambda+(\A-\lambda)\B(1+\lambda)]^2+4[\lambda+(\A+1)\B\lambda(1+\lambda)].
\]
Note that $\Delta>0$ and $\sqrt\Delta>1-\lambda+(\A-\lambda)\B(1+\lambda)$.

We obtain immediately that
\begin{equation}\label{PreExprSComplique}
\sigma_0=\dfrac{-[1-\lambda+(\A-\lambda)\B(1+\lambda)]+\sqrt{\Delta}}{2[\lambda+(\A+1)\B\lambda(1+\lambda)]}
\end{equation}
is the unique positive solution of \eqref{PrePreExprSComplique}.

Consequently
\begin{equation}\label{ExprSComplique}
s_0=\sqrt{\dfrac{-[1-\lambda+(\A-\lambda)\B(1+\lambda)]+\sqrt{\Delta}}{2[\lambda+(\A+1)\B\lambda(1+\lambda)]}}
\end{equation}
is the unique positive solution of \eqref{LasetEqkjb}.

In order to prove that $(s_0,-\lambda s_0)\in(\D^2)^*$, since $0<\lambda\leq1$ and $s_0=\sqrt{\sigma_0}$, it suffices to check that the positive roots $\sigma_0$ given in \eqref{PreExprSComplique} satisfies $\sigma_0<1$. To this end we let $P$ be the quadratic polynomial function expresses in the LHS of \eqref{PrePreExprSComplique} with variable $\sigma$. With direct computations we get $P(0)=-1<0$ and $P(1)=\B(1+\lambda)^2\A>0$. Therefore the equation \eqref{PrePreExprSComplique} admits at least a solution $\tilde{\sigma}\in(0;1)$. Since $\sigma_0$ given in \eqref{PreExprSComplique} is the unique positive solution of \eqref{PrePreExprSComplique} we get $\sigma_0\in(0;1)$.

In conclusion, the set of global minimizers of $W^{\rm micro}[\cdot,(p,q)]$ is
\begin{equation}\label{ExprSCompliqueComplete}
\left\{\left(s_0\e^{\imath\theta};-\lambda s_0\e^{\imath\theta}\right)\in(\D^2)^*\,|\,\theta\in\R\right\}
\end{equation}
where $s_0$ is given by $\eqref{ExprSComplique}$ and $\lambda$ by \eqref{SystCrPoint2}.

In particular, if $p=q$ then $\A=\lambda=1$ and in this case the set of minimizers of $W^{\rm micro}[\cdot,(p,p)]$ is
\[
\left\{\left(\left(1+4\dfrac{1-b^2}{1+b^2}\right)^{-1/4}\e^{\imath\theta};-\left(1+4\dfrac{1-b^2}{1+b^2}\right)^{-1/4}\e^{\imath\theta}\right)\in(\D^2)^*\,|\,\theta\in\R\right\}.
\]
\begin{remark}
It is interesting to note that if $((z_1,z_2),(p,q))\in(\D^2)^*\times(\N^*)^2$ is a minimizers for $W^{\rm micro}$, then we have:
\[
|z_1|\leq|z_2|\Longleftrightarrow p\geq q
\]
and
\[
|z_1|=|z_2|\Longleftrightarrow p= q.
\]
\end{remark}

\appendix
\section{Proof of Lemma \ref{AhhAhhAhh1}}\label{SecProofAsyBeha}

The key ingredient to get Lemma \ref{AhhAhhAhh1} is Proposition C.4 in \cite{Publi4}. For the convenience of the reader we state this proposition:
\begin{prop}\label{PropCited}[Proposition C.4 in \cite{Publi4}]

Let $\alpha\in L^\infty(\R^2,[B^2;B^{-2}])$ and $R>r>0$ we denote:
\begin{itemize}
\item[$\bullet$] $\Ring:=B_R\setminus \overline{B_r}$,
\item[$\bullet$] $\di \mu^{\rm Dir}(\Ring):=\inf\left\{\dfrac{1}{2}\int_{\Ring}\alpha|\n w|^2\,\left|\begin{array}{c}w\in H^1(\Ring,\S^1)\text{ s.t. },\,w(r\e^{\imath\theta})=\e^{\imath\theta},\\ w(R\e^{\imath\theta})=\e^{\imath(\theta+\theta_0)},\,\theta_0\in\R\end{array}\right.\right\}$,
\item[$\bullet$] $\di\mu(\Ring):=\inf\left\{\dfrac{1}{2}\int_{\Ring}\alpha|\n w|^2\,\left|\begin{array}{c}w\in H^1(\Ring,\S^1)\\\text{ s.t. }\deg(w)=1\end{array}\right.\right\}$.
\end{itemize}
There exists a constant $C_B$ depending only on $B$ s.t.
\[
\mu(\Ring)\leq\mu^{\rm Dir}(\Ring)\leq\mu(\Ring)+C_B.
\]
\end{prop}
\begin{remark}
In \cite{Publi4}, Proposition C.4, was initially stated for $\tilde\alpha\in L^\infty(\R^2,[b^2;1])$ and $b\in(0;1)$. Some obvious modifications allow to get the aforementioned formulation.
\end{remark}
Lemma \ref{AhhAhhAhh1} is equivalent to
\begin{equation}\label{AhhAhhAhh2}
\dfrac{1}{2}\int_{\O_R}\alpha|\n u_{R'}|^2-\dfrac{1}{2}\int_{\O_R}\alpha|\n u_R|^2\leq C_{B,\o}.
\end{equation}
Recall that $R_0:=\max\{1;10^2\cdot{\rm diam}({\o})\}$, thus $\overline{\o}\subset B_{R_0}$.

We let 
\begin{equation}\label{FirstContAhAH}
C_\o:=\dfrac{1}{2}\int_{B_{R_0}\setminus\overline{\o}}\left|\n \left(\frac{x}{|x|}\right)\right|^2.
\end{equation}
It is obvious that we have:
\[
\dfrac{1}{2}\int_{\O_{R'}}\alpha|\n u_{R'}|^2\geq\mu(B_{R'}\setminus\overline{B_R})+\mu(B_{R}\setminus\overline{B_{R_0}}).%+\inf\left\{\dfrac{1}{2}\int_{B_{R_0}\setminus\overline{\o}}\alpha|\n w|^2\,\left|\begin{array}{c}w\in H^1(B_{R_0}\setminus\overline{\o},\S^1)\\\text{ s.t. }\deg(w)=1\end{array}\right.\right\}.
\]
Using Proposition \ref{PropCited} we have:
\[
\dfrac{1}{2}\int_{\O_{R'}}\alpha|\n u_{R'}|^2\geq\mu^{\rm Dir}(B_{R'}\setminus\overline{B_R})+\mu^{\rm Dir}(B_{R}\setminus\overline{B_{R_0}})-2C_B.
\]
It is easy to check, {\it e.g.} using the direct method of minimization, that the minima $\mu^{\rm Dir}(B_{R'}\setminus\overline{B_R})$ and $\mu^{\rm Dir}(B_{R}\setminus\overline{B_{R_0}}) $ are reached. Let $u_1$ [resp. $u_2$] be a minimizer of $\mu^{\rm Dir}(B_{R'}\setminus\overline{B_R})$ [resp. $\mu^{\rm Dir}(B_{R}\setminus\overline{B_{R_0}})$].

Up to multiply $u_1$ by a constant rotation we may assume $\tr_{\p B_R}u_1=\tr_{\p B_R}u_2$.

We are now in position to define
\[
u=\begin{cases}
u_1&\text{in }B_{R'}\setminus\overline{B_R}\\u_2&\text{in }B_{R}\setminus\overline{B_{R_0}}\\\dfrac{x}{|x|}&\text{in }B_{R_0}\setminus\overline{\o} 
\end{cases}.
\]
It is clear that $u\in H^1(\O_{R'},\S^1)$ and $\deg(u)=1$. Consequently
\begin{eqnarray*}
\dfrac{1}{2}\int_{\O_{R'}}\alpha|\n u_{R'}|^2&\leq&\dfrac{1}{2}\int_{\O_{R'}}\alpha|\n u|^2
\\&=&\mu^{\rm Dir}(B_{R'}\setminus\overline{B_R})+\mu^{\rm Dir}(B_{R}\setminus\overline{B_{R_0}})+\dfrac{1}{2}\int_{B_{R_0}\setminus\overline{\o}}\alpha\left|\n \left(\frac{x}{|x|}\right)\right|^2
\\\text{[Prop. \ref{PropCited}$\&$ Eq. \eqref{FirstContAhAH}]}&\leq&\mu(B_{R'}\setminus\overline{B_R})+\mu(B_{R}\setminus\overline{B_{R_0}})+2C_B+B^{-2}C_\o.
\end{eqnarray*}
Since $\di \mu(B_{R'}\setminus\overline{B_R})\leq \frac{1}{2}\int_{B_{R'}\setminus\overline{B_R}}\alpha|\n u_{R'}|^2$ and $\mu(B_{R}\setminus\overline{B_{R_0}})\leq\di\frac{1}{2}\int_{\O_R}\alpha|\n u_{R}|^2$ we obtain:
\[
\dfrac{1}{2}\int_{\O_R}\alpha|\n u_{R'}|^2\leq\frac{1}{2}\int_{\O_R}\alpha|\n u_{R}|^2+2C_B+B^{-2}C_\o.
\]
Letting $C_{B,\o}:=2C_B+B^{-2}C_\o$ the above inequality is exactly \eqref{AhhAhhAhh2}.\\

\noindent{\bf Acknowledgements.} %}\normalsize
The author would like to thank Petru Mironescu for fruitful discussions.
%\end{acknowledgements}

\bibliographystyle{amsalpha}
\bibliography{bibHomRen}
\end{document}